\documentclass[10pt,a4paper]{article}
\usepackage[T1]{fontenc}
\usepackage[utf8]{inputenc}
\usepackage{amsmath,amsfonts,amsthm,mathrsfs,amssymb,cite,bm}
\usepackage{graphicx,appendix}
\usepackage{showkeys}
\usepackage{float}
\usepackage{placeins}
\usepackage{color}
\usepackage{booktabs,multirow}
\usepackage{longtable}
\usepackage{authblk}
\usepackage{caption}
\usepackage{subcaption}

\numberwithin{equation}{section}
\newtheorem{theorem}{Theorem}[section]
\newtheorem{lemma}[theorem]{Lemma}

\newtheorem{remark}[theorem]{Remark}

\allowdisplaybreaks[4]

\begin{document}
	\begin{titlepage}
          \title{\bfseries A direct imaging method for inverse scattering problem of biharmonic wave with phased and phaseless data}
		\author{Tielei Zhu\thanks{Corresponding author: zhutielei@126.com}}
		\author{Zhihao Ge}
                \affil{School of Mathematics and Statistics, Henan University, Kaifeng, Henan, 475004,  China. }
		
		\date{\today}
	\end{titlepage}
	\maketitle
	\vspace{.2in}

\begin{abstract}
This paper investigates the inverse biharmonic scattering problems of identifying the shape and location of the obstacle with phased and phaseless measurement data. A direct imaging method based on reverse time migration is proposed for reconstructing the extended obstacle with one of four types of boundary conditions on the obstacle. The newly developed imaging functions are constructed by utilizing merely one of various measurement data, including  the scattered field, its normal derivative, the bending moment, the  transverse force, its far-field and the phaseless total field. Our resolution analysis demonstrates that these imaging functions have a contrast  when sampling points are near or far from the boundary of the obstacle. Numerical experiments  are further  presented to show  the algorithm's efficiency to accurately reconstruct complex scatter geometries and its robustness to noise.

\vspace{.2in} {\bf Keywords:} Inverse biharmonic scattering problem, reverse time migration method, resolution analysis.

\vspace{.2in} {\bf AMS subject classifications}: 35Q74, 74J20
\end{abstract}

\section{Introduction}\label{sec:intro}
The inverse obstacle scattering problems, which aims to determine the shape and location of the obstacle from the phased or phaseless measurement data have received significant attention in various fields including radar, medical imaging and geophysical exploration. Over the past decades, a large amount of numerical methods have been developed for the reconstruction of obstacles, such as Newton-type method in \cite{Hohage_newton_1998}, Kirsch-Kress method in \cite{colton_inverse_2019}, linear sampling method in \cite{colton_linear_2011}, factorization method in \cite{kirsch_factorization_2008} and direct sampling method in \cite{ito_direct_2012,potthast_study_2010}.

Most existing research has focused on inverse acoustic, elastic, and electromagnetic wave scattering problems. In contrast to these inverse scattering problems associated with second-order differential operators, limited work has been done on biharmonic inverse scattering problems involving the biharmonic wave equation (a fourth-order partial differential equation). However, the biharmonic inverse scattering problems have attracted growing research interest in recent years. Bourgeois and Recoquillay \cite{bourgeois_linear_2020} established the uniqueness of Dirichlet-type (or Neumann-type) obstacles using multi-static data, namely the scattered field and its normal derivative. More recently, Dong and Li derived two uniqueness results based on phased far-field patterns and phaseless near-field data in \cite{dong_uniqueness_2024}. Regarding numerical methods for biharmonic wave inverse problems, Bourgeois and Recoquillay introduced a linear sampling method for obstacle reconstruction from multi-static near-field data. Chang and Guo \cite{chang_optimization_2023} developed an optimization method for locating obstacles using measurement data of the scattered field and its Laplacian on a measured curve. Additionally, Guo \emph{et al.} \cite{guo_direct_2025} proposed a linear sampling method employing far-field data. It is noteworthy that the aforementioned research addressed the frequency-domain problems. For the time domain problems, we refer to the recent work \cite{gao_inverse_2023}.

In this paper, we aim to develop new direct imaging method for the inverse biharmonic wave scattering problem with phased and phaseless data. Roughly speaking, the key to the direct imaging method lies in designing imaging functions that present distinct characteristics depending on whether the sampling point is near or far from the obstacle's boundary.  This method has two primary advantages over iterative methods: (1) it does not require prior geometric assumptions about the scatterer, and (2) it avoids the need to solve the direct problems numerically. Very recently, Harris \emph{et al.} proposed a direct imaging method for reconstructing the obstacle with the clamped boundary condition using far-field data in  \cite{harris_direct_2025}. However, our imaging functions are based on reverse time migration method (RTM), which has been widely applied in seismic imaging and investigated in inverse scattering for acoustic, elastic and electromagnetic waves  \cite{chen_reverse_2013,chen_reverse_2015,chen_reverse_2013b}, and depend on only one of various measurement data, such as the scattered field, its normal derivative, the bending moment, the transverse force, its far-field and the phaseless total field. It should be emphasized that previous studies about the RTM can not be extended to the case of inverse biharmonic wave scattering directly. Furthermore, we propose novel imaging functions, which are involved only with the fundamental solution of the Helmholtz equation or an acoustic plane wave rather than the fundamental solution of the biharmonic wave equation or a biharmonic plane wave apart from the measurement data, and this formulation reduces computational complexity.  We rely on new asymptotic behaviors to establish the resolution analysis of the imaging functions for the inverse scattering of the biharmonic wave, which is applicable to four types of boundary conditions.

The rest of the paper is organized as follows. In Section 2, we describe the direct and inverse scattering problems of biharmonic wave and introduce several important lemmas.  Next, we propose and analyze the imaging functions with phased and phaseless data in Section 3 and Section 4, respectively. In Section 5, numerical experiments will  be reported to show the validity of these methods.

\section{Formulation of the biharmonic wave scattering problem}
\label{sec:formulation}
\subsection{The direct scattering problem}
\label{direct}
Let $D\subset\mathbb{R}^2$ be a bounded domain with $C^3$-smooth boundary $\Gamma$ such that $D^{c}:=\mathbb{R}^2\setminus\overline{D}$ is connected.  When the obstacle $D$ is illuminated by the incident wave $u^{in}$, the scattered wave $u^{sc}$ is excited, which satisfies the biharmonic wave equation 
\begin{equation}
\label{eq:bi}
\Delta^2u^{sc}-\kappa^4u^{sc}=0 \quad\mathrm{in}\; D^{c},
\end{equation}
together with the  boundary conditions
\begin{equation}
\label{eq:boundary}
(\mathcal{B}_1u^{sc},\mathcal{B}_2u^{sc}) = 
-(\mathcal{B}_1u^{in},\mathcal{B}_2u^{in}) \quad\mathrm{on}\;\Gamma,
\end{equation}
where $\kappa>0$ is the wave number and $\mathcal{B}_1$ and $\mathcal{B}_2$ are boundary differential operators to be defined latter. 
 Moreover, the scattered wave $u^{sc}$ satisfies the Sommerfeld radiation condition
\begin{equation}
\label{eq:radiation}
\lim\limits_{r=|\bm{x}|\to\infty}\sqrt{r}\left(\frac{\partial u^{sc}}{\partial r}
  -\mathrm{i}\kappa u^{sc}\right)=0
\end{equation}
uniformly for all $\hat{\bm x}=\bm{x}/|\bm{x|}\in\mathbb{S}:=\{\bm{x}\in\mathbb{R}^2:|\bm{x}|=1\}$.
As pointed out in \cite{bourgeois_well_2020}, the Sommerfeld radiation condition for $u^{sc}$ is sufficient. Due to \eqref{eq:radiation},
the scattered wave $u^{sc}$ has the following asymptotic expansion:
\begin{equation}
  \label{eq:asy}
  u^{sc}(\bm{x}) = \dfrac{e^{\pi \mathrm{i}/4}}{\sqrt{8\pi\kappa}}
  \dfrac{e^{\mathrm{i}\kappa|\bm{x}|}}{|\bm{x}|^{1/2}}
u^{\infty}(\hat{\bm{x}}) + O\left(\dfrac{1}{|\bm{x}|^{3/2}} \right) \quad |\bm{x}|\to\infty,
\end{equation}
where $u^{\infty}$ is defined as the far-field pattern of the scattered wave $u^{sc}$.

To clarity the definitions of the boundary differential operators $\mathcal{B}_1$
and $\mathcal{B}_2$, let $\bm{ n }=(n_1,n_2)^T$ be the exterior unit normal vector of the boundary $\Gamma$. Define the boundary operators  $M_0v$ and $N_0v$ as follows:
\begin{equation*}
  M_0v = n_1^2 \dfrac{\partial^2 v }{\partial x_1^2} + 2n_1n_2 \dfrac{\partial^2v }{\partial x_1\partial x_2}
          + n_2^2 \dfrac{\partial^2 v }{\partial x_2^2}
\end{equation*}
and
\begin{equation*}
  N_0v = -\Big\{  \big( \dfrac{\partial^2 v }{\partial x_1^2}-  \dfrac{\partial^2 v }{\partial x_2^2} \big)n_1n_2  - \dfrac{\partial^2v }{\partial x_1\partial x_2}(n_1^2 - n_2^2)
          \Big\}.
\end{equation*}
Then the boundary operators $Mv$ and $Nv$ are defined as follows:
\begin{equation*}
M v = \nu \Delta v + (1-\nu) M_0 v
\end{equation*}
and
\begin{equation*}
N v = -\partial_{\bm{n} }\Delta v - (1-\nu) \partial_sN_0 v,
\end{equation*}
where $\nu\in[0,\frac{1}{2})$ is the Poisson ratio.
Here the normal and tangential derivatives $\partial_{\bm{n} }$  and $\partial_s$ are given by
\begin{equation*}
  \partial_{\bm{n} } = n_1\partial_{x_1} + n_2\partial_{x_2} \quad
  \mathrm{ and } \quad
  \partial_s=  -n_2\partial_{x_1}  + n_1\partial_{x_2}.
\end{equation*}
From a physical perspective, $Mv$ is the bending moment, whereas $Nv$ is the transverse force (consisting of the shear force and twisting moment) \cite{chen_boundary_2010}.

In this paper, we consider the following boundary conditions:
\begin{enumerate}
\item Clamped boundary condition: $\mathcal{B}_1=I$ and $\mathcal{B}_2=\partial_{\bm{n} }$
\item Simply supported boundary condition: $\mathcal{B}_1=I$ and $\mathcal{B}_2=M$
\item Roller supported boundary condition: $\mathcal{B}_1=\partial_{\bm{n}}$ and $ \mathcal{B}_2=N$
\item Free boundary condition: $\mathcal{B}_1=M$ and $ \mathcal{B}_2=N$.
 \end{enumerate}
It is important to emphasize that the aforementioned boundary conditions, rooted in plate theory, impose different requirements on the domain $D$. Specifically, the domain $D$ should be at least of class $C^2$ in the boundary conditions \uppercase\expandafter{\romannumeral 1}-\uppercase\expandafter{\romannumeral 3}, whereas the domain should be at least of class $C^3$ in the boundary condition \uppercase\expandafter{\romannumeral 4}, which is necessary in the proof of Theorem 3.4 in \cite{bourgeois_well_2020}.

The direct problem has been studied in \cite{bourgeois_well_2020}.
First, we state the well-posedness of a more general boundary value problem
  \begin{subequations}
    \label{eq:bvp}
    \begin{align}
      \label{eq:bvp-1}
    \Delta^2w - \kappa^4w &= 0 \quad \mathrm{in}\; D^{c},\\
       \label{eq:bvp-2}
    (\mathcal{B}_1w,\mathcal{B}_2w)&=(f,g)  \quad \mathrm{on}\; \Gamma,\\
        \label{eq:bvp-3}
    \lim\limits_{r\to\infty}\sqrt{r}\left(\frac{\partial w}{\partial r}
      -\mathrm{i}\kappa w \right)&=0,
    \end{align}
  \end{subequations}
 which could become  \eqref{eq:bi}-\eqref{eq:radiation} if $(f,g)=-(\mathcal{B}_1u^{in},\mathcal{B}_2 u^{in})$.
\begin{lemma}[see \cite{bourgeois_well_2020}]
  \label{lem:well}
  The boundary value problem \eqref{eq:bvp} admits a unique solution $w$ in $H^{2}_{\mathrm{ loc } }(D^c)$
  \begin{itemize}
  \item for any  $(f,g)\in H^{3/2}(\Gamma)\times H^{1/2}(\Gamma)$ and  $\kappa>0$ if $\mathcal{B}_1=I$ and $\mathcal{B}_2=\partial_{\bm{n} }$. Moreover, the the solution satisfies the following estimation:
    \begin{align*}
  \|w\|_{H^2_{\mathrm{loc}}(D^{c})} 
 & \leq C (\|f \|_{H^{3/2}(\Gamma)} + \|g \|_{H^{1/2}(\Gamma)}),\\
 \|M w\|_{H^{-1/2}(\Gamma)} + \|N w \|_{H^{-3/2}(\Gamma)}
 &\leq C (\|f \|_{H^{3/2}(\Gamma)} + \|g \|_{H^{1/2}(\Gamma)}),
\end{align*}
where $C>0$ is some constant depending on wave number $\kappa$ and the domain $D$.
    \item for any  $(f,g)\in H^{3/2}(\Gamma)\times H^{-1/2}(\Gamma)$ and  $\kappa>0$ if $\mathcal{B}_1=I$ and $\mathcal{B}_2=M$. Moreover, the the solution satisfies the following estimation:
    \begin{align*}
  \|w\|_{H^2_{\mathrm{loc}}(D^{c})} 
 & \leq C (\|f \|_{H^{3/2}(\Gamma)} + \|g \|_{H^{-1/2}(\Gamma)}),\\
 \|\partial_{\bm{n} } w\|_{H^{1/2}(\Gamma)} + \|N w \|_{H^{-3/2}(\Gamma)}
 &\leq C (\|f \|_{H^{3/2}(\Gamma)} + \|g \|_{H^{-1/2}(\Gamma)}),
    \end{align*}
    where $C>0$ is some constant depending on wave number $\kappa$ and the domain $D$.
\item for any  $(f,g)\in H^{1/2}(\Gamma)\times H^{-3/2}(\Gamma)$ and  $\kappa>0$ if $\mathcal{B}_1=\partial_{\bm{n} } $ and $\mathcal{B}_2=N$. Moreover, the the solution satisfies the following estimation:
    \begin{align*}
  \|w\|_{H^2_{\mathrm{loc}}(D^{c})} 
 & \leq C (\|f \|_{H^{1/2}(\Gamma)} + \|g \|_{H^{-3/2}(\Gamma)}),\\
 \| w\|_{H^{3/2}(\Gamma)} + \|M w \|_{H^{-1/2}(\Gamma)}
 &\leq C (\|f \|_{H^{1/2}(\Gamma)} + \|g \|_{H^{-3/2}(\Gamma)}),
    \end{align*}
    where $C>0$ is some constant depending on wave number $\kappa$ and the domain $D$.
    \item for any  $(f,g)\in H^{-1/2}(\Gamma)\times H^{-3/2}(\Gamma)$ and  $\kappa>0$ except an at most countable sequence of positive wave number $(\kappa_j)_{j\in  \mathbb{N}} \subset \mathbb{R}_{>0}$ without
finite accumulation point, if $\mathcal{B}_1= M$ and $\mathcal{B}_2=N$. Moreover, the the solution satisfies the following estimation:
    \begin{align*}
  \|w\|_{H^2_{\mathrm{loc}}(D^{c})} 
 & \leq C (\|f \|_{H^{-1/2}(\Gamma)} + \|g \|_{H^{-3/2}(\Gamma)}),\\
 \| w\|_{H^{3/2}(\Gamma)} + \|\partial_{\bm{n}} w \|_{H^{1/2}(\Gamma)}
 &\leq C (\|f \|_{H^{-1/2}(\Gamma)} + \|g \|_{H^{-3/2}(\Gamma)}),
    \end{align*}
    where $C>0$ is some constant depending on wave number $\kappa$ and the domain $D$.
    \end{itemize}
\end{lemma}
In this paper, to ensure the uniqueness of the scattering problem \eqref{eq:bvp}, we always assume that \emph{$\kappa$ is not in the discrete set of wave number for boundary condition \uppercase\expandafter{\romannumeral 4}.}

The fundamental solution to the biharmonic wave equation is given by
\begin{equation}
  \label{eq:funda-bi}
  G(\bm{x},\bm{y}) = \frac{1}{2\kappa^{2}}
  \left(\Phi_{\mathrm{i}\kappa}(\bm{x},\bm{y})- \Phi_{\kappa}(\bm{x},\bm{y})\right),
\end{equation}
i.e., it satisfies
\begin{equation*}
\Delta^{2} G( \bm{ x }, \bm{ y }) -\kappa^4 G( \bm{ x },\bm{ y }) = -\delta(\bm{ x }-\bm{ y }).
\end{equation*}
Here $\Phi_{\kappa}(\bm{ x },\bm{ y })=\frac{\mathrm{ i }}{4} H^{(1)}_0(\kappa | \bm{ x }-\bm{ y } |)$ is the fundamental solution of two-dimensional Helmholtz equation and $H^{(1)}_0$ is the Hankel function of the first kind with the order zero.

Let us introduce the Green formula (see \cite{hsiao_bounday_2021}). Assume that a bounded domain $\Omega$ is $C^2$-smooth with $\bm{ n }$ the outward unit normal.
For $v\in H^2(\Omega,\Delta^2):=\{v \in H^2(\Omega):\Delta^2v \in L^{2}(\Omega)\}$ and $w \in H^2(\Omega)$, we can have
\begin{equation} \label{eq:green-1}
\int_{ \Omega } (\Delta^{2} v) w \;\mathrm{d}\bm{x} = a_{\Omega}(v,w) - \int_{ \partial \Omega } \left[(Mv)\partial_{\bm{n} } w + (Nv) w \right] \;\mathrm{d}s,
\end{equation}
where
\begin{equation*}
  a_{ \Omega}(v,w)= \nu \int_{ \Omega} \Delta v \Delta w \;\mathrm{d} \bm{ x} +
  (1-\nu)\int_{ \Omega} \left( \dfrac{\partial^2 v}{\partial x_1^2}\dfrac{\partial^2 w}{\partial x_1^2}
    + 2 \dfrac{\partial^2 v}{\partial x_1 \partial x_2 }\dfrac{\partial^2 w}{\partial x_1\partial x_2  }
    +\dfrac{\partial^2 v}{\partial x_2^2}\dfrac{\partial^2 w}{\partial x_2^2}
     \right) \;\mathrm{d}\bm{ x}.
\end{equation*}
Similarly, for $v,w \in H^2(\Omega,\Delta^2)$ we could obtain
\begin{equation} \label{eq:green-2}
  \int_{ \Omega }\Big[ w (\Delta^{2} v) -  v(\Delta^{2} w)\Big] \;\mathrm{d}\bm{x} =   \int_{ \partial \Omega }\Big\{
   \left[(Mw)\partial_{\bm{n} } v + (Nw) v \right]
  - \left[(Mv)\partial_{\bm{n} } w + (Nv) w \right]\Big\} \;\mathrm{d}s.
\end{equation}

Based on the fundamental solution $G(\bm{ x },\bm{ y })$ and Green formula, we can get the representations of the scattered field and its far-field pattern in the following lemma, which is similar to Proposition 1 in \cite{bourgeois_linear_2020} and Theorem 3.2 in \cite{dong_uniqueness_2024}. For the sake of completeness, we provide a proof.
\begin{lemma}
  \label{lem:rep}
  Assume that $w\in H^2_{\mathrm{loc}}(D^{c})$ is the radiation solution, i.e., it satisfies \eqref{eq:bi}
  and \eqref{eq:radiation}, and $w^{\infty}$ is far-field pattern of $w$. Then
  \begin{align}
   \notag
    w(\bm{x})  =&
                \int_{\Gamma}\Big\{
                  \big[ G(\bm{x},\bm{y}) N w(\bm{y})
                  + [\partial_{\bm{n}(\bm{y})}G(\bm{x},\bm{y})] M w(\bm{y}) \big]\\
    \label{eq:Green-rep}
              &\qquad\quad  - \big[ [N_{\bm{y}} G(\bm{x},\bm{y})]    w(\bm{y})
                    + [M_{\bm{y}} G(\bm{x},\bm{y})]   \partial_{\bm{n}} w(\bm{y}) \big]
                \Big\} \;\mathrm{d}s(\bm{y}) \quad \bm{x}\in D^{c},  \\   \notag
    w^{\infty}(\hat{\bm x}) = & -\frac{1}{2\kappa^{2}}
                \int_{\Gamma}\Big\{
                  [  e^{-\mathrm{i}\kappa\hat{\bm x}\cdot \bm{y}}N w(\bm{y})
                           + (\partial_{\bm{n}(\bm{y})}e^{-\mathrm{i}\kappa\hat{\bm x}\cdot \bm{y}})  M w(\bm{y}) ] \\
    \label{eq:far-rep}
             & \qquad\qquad\quad   - [ (N_{\bm{y}}e^{-\mathrm{i}\kappa\hat{\bm x}\cdot \bm{y}})    w(\bm{y})
                    + (M_{\bm{y}}e^{-\mathrm{i}\kappa\hat{\bm x}\cdot \bm{y}})  \partial_{\bm{n}} w(\bm{y}) ]
                \Big\} \;\mathrm{d}s(\bm{y}) \quad \hat{\bm x}\in \mathbb{S}.
  \end{align}
\end{lemma}
\begin{proof}
  Let $\bm{ x }\in D^c$ and choose a sufficiently large $r>0$ and a sufficiently small $\rho>0$ such that $\overline{D}\subset B(\bm{ 0 },r)$ and $\overline{B(\bm{ x },\rho)}\subset B(\bm{ 0 },r)\setminus \overline{D}$, where
  the unit normal $\bm{ n }$ of $\partial B(\bm{ 0 },r)$  is directed toward the outside of the ball  $B(\bm{ 0 },r)$
  and  the unit normal $\bm{ n }$ of  $\partial B(\bm{ x},\rho)$   is directed toward the inside of the ball  $B(\bm{ x },\rho)$.  
    We apply the Green formula \eqref{eq:green-2} to $w$ and $G(\bm{ x },\cdot)$  in $D_{\rho,r}:=B(\bm{ 0 },r)\setminus (\overline{D}\cup \overline{B(\bm{ x },\rho)})$ to obtain 
  \begin{align} \notag
    0 = & \int_{ D_{\rho,r}  } [G(\bm{ x }, \bm{ y }) \Delta^2 w(\bm{ y })
          - w (\bm{ y })  \Delta^2_{\bm{ y }}  G(\bm{ x }, \bm{ y })]
          \;\mathrm{d}\bm{ y } 
    =  \int_{ \Gamma \cup \partial B(\bm{ x },\rho)\cup \partial B(\bm{ 0 },r) }
        \mathcal{U}(\bm{ x },\bm{ y }) \;\mathrm{d}s(\bm{ y })\\
    \label{eq:greendr}
    = & \left( -\int_{ \Gamma} +\int_{ \partial B(\bm{ x },\rho)} + \int_{\partial B(\bm{ 0 },r) }\right)
        \mathcal{U}(\bm{ x },\bm{ y }) \;\mathrm{d}s(\bm{ y }),
  \end{align}
  where $$\mathcal{U}(\bm{ x },\bm{ y })
  =\left[ (M_{\bm{ y }}G(\bm{ x }, \bm{ y }))\partial_{\bm{n}(\bm{ y }) } w(\bm{ y })
    + (N_{\bm{ y }}G(\bm{ x }, \bm{ y }) ) w(\bm{ y }) \right]
  - \left[\partial_{\bm{n}(\bm{ y }) } G(\bm{ x }, \bm{ y }) (M_{\bm{ y }}w(\bm{ y }))  +  G(\bm{ x }, \bm{ y }) (N_{\bm{ y }}w(\bm{ y })) \right].$$
  It follows from Propositions 2.1 and 2.2 in \cite{bourgeois_well_2020} that the radiation solution $w$ has the decomposition $w=w^{pr}+w^{ev}$ and their asymptotic behaviors are
  \begin{equation*}
    |w^{pr}(\bm{ x })|=O(| \bm{ x } | ^{-1/2}),\quad
    |\partial_rw^{pr}-\mathrm{ i } \kappa w^{pr}|=O(| \bm{ x } |^{-3/2})\quad
    \mathrm{ and } \quad | w^{ev}(\bm{ x }) | = O(e^{-\kappa | \bm{ x } |}) \quad  \mathrm{as}\quad | \bm{ x } |\to \infty,
  \end{equation*}
  uniformly with respect to $\theta \in [0,2\pi)$. 
  Note that the boundary operators $M$ and $N$ in the polar coordinates can be rewritten as follows:
  \begin{align}
    \label{eq:Mpolar}
  Mw  = & \Delta w - (1-\nu)\left( \frac{1}{r} \dfrac{\partial w}{\partial r} + \frac{1}{r^{2}} \dfrac{\partial^2 w}{\partial \theta^2}\right),   \\
    \label{eq:Npolar}
 N w  = & -\dfrac{\partial }{\partial r}\Delta w - (1-\nu) \frac{1}{r^{2}} \dfrac{\partial }{\partial \theta} \left(\dfrac{\partial^2 w}{\partial r \partial \theta} - \frac{1}{r} \dfrac{\partial w}{\partial \theta}\right).
  \end{align}
  Then we have
  \begin{equation}\label{eq:MNw}
    M w(\bm{ x }) = -\kappa^2 w^{pr}(\bm{ x }) + O(r^{-3/2}) \quad \mathrm{ and } \quad
    N w(\bm{ x }) = \kappa^{2}\partial_r  w^{pr}(\bm{ x }) + O(r^{-5/2}) \quad  \mathrm{as}\quad r=| \bm{ x } | \to \infty,
  \end{equation}
  where we use the fact that $\Delta w^{pr}+ \kappa^2 w^{pr}=0$ in $D^c$. Similarly,
  the asymptotic behaviors of the Hankel function \cite[Chapter 7, section 4.1]{olver_asymptotic_1974} as $z\to\infty$, i.e.,
\begin{equation}\label{eq:Hankel}
H^{(1)}_{n}(z ) = \sqrt{\frac{2}{\pi z}}e^{\mathrm{ i } (z-n\pi/2-\pi/4)}\left(1+O(|z|^{-1})\right) \quad \mathrm{ for } \quad\mathrm{arg}\, z \in [-\frac{\pi}{2}, \frac{\pi}{2}] \quad \mathrm{ and } \quad n=0,1,\cdots, 
\end{equation}
and the recurrence relation
\begin{equation*}
(H^{(1)}_m)'(z)= H^{(1)}_{m-1}(z)-\frac{m}{z} H^{(1)}_m(z)
\end{equation*}
  imply that
  \begin{equation}\label{eq:MG}
    M_{\bm{ y }}  G(\bm{ x }, \bm{ y })  = -\kappa^2 G^{pr}(\bm{ x }, \bm{ y })  + O(| \bm{ y } | ^{-3/2})  \quad  \mathrm{as}\quad | \bm{ y } | \to \infty,
  \end{equation}
  and
  \begin{equation}
    \label{eq:NG}
    N_{\bm{ y }}  G(\bm{ x }, \bm{ y })  = \kappa^{2}\partial_{| \bm{ y } | }  G^{pr}(\bm{ x }, \bm{ y }) + O(| \bm{ y } | ^{-5/2}) \quad  \mathrm{as}\quad | \bm{ y } | \to \infty,
  \end{equation}
where $G^{pr}(\bm{ x }, \bm{ y }) = -1/(2 \kappa^2)\Phi_{\kappa}(\bm{ x }, \bm{ y })$.
Letting $r\to \infty$ and combining \eqref{eq:MNw}-\eqref{eq:NG}, we have
\begin{align}
  \notag
  &\lim\limits_{r\to\infty} \int_{\partial B(\bm{ 0 },r) }  \mathcal{U}(\bm{ x },\bm{ y })
    \;\mathrm{d}s(\bm{ y })\\
\notag
  = & 2\kappa^2 \lim\limits_{r\to\infty}  \int_{\partial B(\bm{ 0 },r) } \Big\{ 
      -G^{pr}(\bm{ x }, \bm{ y })\left[ \partial_{\bm{n}(\bm{ y }) } w^{pr}(\bm{ y }) - \mathrm{ i } \kappa w^{pr}(\bm{ y}) \right]\\
  &\qquad \qquad \qquad  + w^{pr}(\bm{ y }) \left[\partial_{\bm{n}(\bm{ y }) }G^{pr}(\bm{ x }, \bm{ y }) -\mathrm{ i } \kappa G^{pr}(\bm{ x }, \bm{ y })   \right]
    \Big\} \;\mathrm{d}s(\bm{ y }) =0,
\label{eq:greendr-1}
\end{align}
since that $\Phi_{\kappa}(\bm{ x }, \bm{ y })$ satisfies the Sommerfeld radiation condition \eqref{eq:radiation}.
Observe that for $\bm{ y }\in \partial B(\bm{ x },\rho)$ the following formulas hold:
  \begin{align*}
  \partial_{\bm{n}(\bm{ y })} G(\bm{ x },\bm{ y })  = & -\frac{\mathrm{ i } }{8 \kappa} [\mathrm{ i } H^{(1)}_1(\mathrm{ i } \kappa \rho) -  H^{(1)}_1(\kappa \rho)], \\
   M_{\bm{ y }} G(\bm{ x },\bm{ y }) = & \frac{\mathrm{ i } }{8} [H^{(1)}_0(\mathrm{ i } \kappa \rho) +  H^{(1)}_0(\kappa \rho)] + (1-\nu) \frac{1}{\rho} \partial_{\bm{n}(\bm{ y })} G(\bm{ x },\bm{ y }),  \\
    N_{\bm{ y }} G(\bm{ x },\bm{ y }) = & \frac{\mathrm{ i } }{8} [  \dfrac{\partial }{\partial \rho} H^{(1)}_0(\mathrm{ i } \kappa \rho)+\dfrac{\partial }{\partial \rho} H^{(1)}_0(\kappa \rho)] . 
  \end{align*}
  Using the asymptotic behaviors of Hankel functions as $\rho\to 0$, we obtain
  \begin{align}
    \label{eq:greendr-2}
    \lim\limits_{\rho\to 0} \int_{\partial B(\bm{ x },\rho) }  \mathcal{U}(\bm{ x },\bm{ y })
      \;\mathrm{d}s(\bm{ y })
    = \lim\limits_{\rho\to 0} \int_{\partial B(\bm{ x},\rho) } (N_{\bm{ y }}G(\bm{ x }, \bm{ y }) ) w(\bm{ y })  \;\mathrm{d}s(\bm{ y }) = -w(\bm{ x }).
  \end{align}
The Green representation \eqref{eq:Green-rep} follows by substituting \eqref{eq:greendr-1} and \eqref{eq:greendr-2} into \eqref{eq:greendr}. By the representation \eqref{eq:Green-rep} we derive \eqref{eq:far-rep} due to the definition of the far-field pattern and the asymptotic behaviors of the Hankel functions.
 \end{proof}
Moreover, we have the following property of the radiation solution, which is corresponding to the sign property of the imaginary part of the Dirichlet-to-Neumann mapping in the acoustic and elastic wave scattering problems (c.f. \cite[Lemma 3.3]{chen_reverse_2013} and \cite[Lemma 3.5]{chen_reverse_2015})  and has been proved essentially in Proposition 3.2 of \cite{bourgeois_well_2020}. For the convenience of the readers, we provide the details here.
\begin{lemma}
  \label{lem:sign}
  Let $w$ satisfy \eqref{eq:bi} and \eqref{eq:radiation} and $w^{\infty}$ be the far-field pattern of $w$. Then
  \begin{equation*}
    \mathrm{Im} \int_{\Gamma}[\overline{w}\, N w +  \partial_{\bm{n}}\overline{w}\, M w] \;\mathrm{d}s
    =   \mathrm{Im} \int_{\Gamma}[\overline{w}\, N w -M  \overline{w}\,  \partial_{\bm{n}}w] \;\mathrm{d}s
    = \frac{\kappa^{2}}{4\pi}\int_{\mathbb{S}}|w^{\infty}|^{2}\mathrm{d}s.
  \end{equation*}
\end{lemma}
\begin{proof}
  It suffices to prove
\begin{equation*}
    \mathrm{Im} \int_{\Gamma}[\overline{w}\, N w +  \partial_{\bm{n}}\overline{w}\, M w] \;\mathrm{d}s
    = \frac{\kappa^{2}}{4\pi}\int_{\mathbb{S}}|w^{\infty}|^{2}\mathrm{d}s,
  \end{equation*}
  due to the fact that $\mathrm{Im}\,(z_1+z_2)=\mathrm{Im}\, (z_1-\overline{z_{2}})$ for any $z_1$, $z_2 \in \mathbb{C}$.
  Let  $B(\bm{ 0 },r)$ denote an open disk  centered at the origin with radius  $r$  such that the disk $B(\bm{ 0 },r)$ contains $\overline{D}$.
  By the Green formula \eqref{eq:green-1} in $B_r\setminus \overline{D}$, we have
 \begin{align*}  
   \left( \int_{\partial B(\bm{ 0 },r)}- \int_{\Gamma}\right)
   [\overline{w}\, N w + \partial_{\bm{n}}\overline{w} \, M w]\;\mathrm{d}s
   = a_{D_r}(w,\overline{w}) - \kappa^4  \int_{D_r}| w | ^2\;\mathrm{d} \bm{ x }
\end{align*}
  where $D_r:=B(\bm{ 0 },r)\setminus \overline{D}$, which implies 
  \begin{align}  \label{eq:lem-1}
   \mathrm{Im}\, \int_{\Gamma}[\overline{w}\, N w + \partial_{\bm{n}}\overline{w} \, M w]\;\mathrm{d}s
     =\mathrm{Im}\, \int_{\partial B(\bm{ 0 },r)}[\overline{w}\, N w +\partial_{r}\overline{w} \,   M w]\;\mathrm{d}s 
  \end{align}
From the asymptotic behavior of $w$, i.e., \eqref{eq:asy} with $u^{\infty}$ (or $u^{sc}$) replaced by $w^{\infty}$ (or $w$),
together with the identities \eqref{eq:Mpolar} and \eqref{eq:Npolar},
we can obtain the asymptotic behaviors of $M w$ and $N w$ as
follows:
\begin{align*}
  M w \left( \bm{ x } \right)
  &= -\kappa^2  \dfrac{e^{\mathrm{i} \pi/4   }}{\sqrt{8\pi \kappa}}
        \dfrac{e^{\mathrm{i}\kappa | \bm{x} | }}{| \bm{ x }  |^{1/2} } w^{\infty} \left( \hat{\bm{ x }} \right)
        + O \left( | \bm{ x } |^{-3/2}   \right) \quad | \bm{ x }  |\to \infty, \\
  N w \left( \bm{ x } \right) =& = \mathrm{ i } \kappa^3  \dfrac{e^{\mathrm{i} \pi/4   }}{\sqrt{8\pi \kappa}}
                                 \dfrac{e^{\mathrm{i}\kappa | \bm{x} | }}{| \bm{ x }  |^{1/2} } w^{\infty} \left( \hat{\bm{ x }} \right)
                                  + O \left( | \bm{ x } |^{-3/2}   \right) \quad | \bm{ x }  |\to \infty.
\end{align*}
Straightforward calculations shows that
\begin{align*}
  \overline{w}\left(\bm{ x } \right) N w\left(\bm{ x  }\right)
  + M w\left(\bm{ x } \right) \partial_{r}\overline{w\left(\bm{ x } \right)} =
  \frac{\mathrm{ i }  \kappa^{2}}{4\pi} |w^{\infty} \left( \hat{\bm{ x }} \right) |^2
  + O \left( | \bm{ x } |^{-2}\right) \quad  | \bm{ x }  |\to \infty.
\end{align*}
Taking the imaginary part of \eqref{eq:lem-1} and the limits $r=| \bm{ x } | \to \infty$ yield the conclusion.
\end{proof}

\subsection{The inverse scattering problems for extended obstacles}

Let $B_r$ (or $B_s$) be the open disk with center at the origin and radius $R_r>0$ (or $R_s>0$)  such that $\overline{D}\subset B_r$ (or $\overline{D}\subset B_s$), and we denote by $\Gamma_r$ (or $\Gamma_s$) the boundary of  $B_r$ (or $B_s$) with $\bm{ n }_r$ being the outward unit normal of $\Gamma_r$. Let $\Omega$ be the sampling domain containing the closure  $\overline{D}$ of the obstacle $D$. We assume that $\overline{\Omega}$ is entirely contained within either $B_r$ (or $B_s$), i.e., $\overline{\Omega}\subset B_r$ (or $B_s$).

Consider the plane wave $u^{in}(\bm{ x }, \bm{ d })=e^{\mathrm{ i } \kappa \bm{ x }\cdot \bm{ d }}$ as the incident wave, where $\bm{ d }\in \mathbb{S}$ is the incident direction, and we denote by $u^{sc}(\bm{ x }, \bm{ d })$ and $u^{\infty}(\hat{\bm  x }, \bm{ d })$ the scattered field and far-field pattern respectively. Besides, we also consider the point source located on $\Gamma_s$ as the incident wave. Let  $u^{in}(\bm{ x },\bm{ x }_s)=\Phi_{\kappa}(\bm{ x },\bm{ x }_s)$ ($\bm{ x }_s \in \Gamma_{s}$) and then  $u^{sc}(\bm{ x }, \bm{ x }_s)$ and $u^{\infty}(\hat{\bm  x }, \bm{ x }_s)$ are the scattered field and far-field pattern respectively.
 Define $u(\bm{ x }, \bm{ x }_s):= u^{sc}(\bm{ x }, \bm{ x }_s)+u^{in}(\bm{ x }, \bm{ x }_s)$ as the total field.

In the remainder of  this  paper,  we  are  interested in the 
following eleven inverse biharmonic scattering problems.
\begin{itemize}
\item \textbf{IP-1}: Determine the shape and location of  the obstacle $D$
  from
  $u^{sc}(\bm{ x }_r, \bm{ x }_s)$
  for all observation points $\bm{ x }_r \in \Gamma_r$ and all source points $\bm{ x }_s \in \Gamma_s$.
  
\item \textbf{IP-2}: Determine the shape and location of  the obstacle $D$
  from $\partial_{\bm{n}(\bm{ x }_{r}) } u^{sc}(\bm{ x }_r, \bm{ x }_s)$
  for all observation points $\bm{ x }_r \in \Gamma_r$ and all source points $\bm{ x }_s \in \Gamma_s$.
  
  \item \textbf{IP-3}: Determine the shape and location of the obstacle $D$
    from 
    $M_{\bm{ x }_r} u^{sc}(\bm{ x }_r, \bm{ x }_s)$
  for all observation points $\bm{ x }_r \in \Gamma_r$ and all source points $\bm{ x }_s \in \Gamma_s$.
  
  \item \textbf{IP-4}: Determine the shape and location of the obstacle $D$
    from 
    $N_{\bm{ x }_r} u^{sc}(\bm{ x }_r, \bm{ x }_s)$
  for all observation points $\bm{ x }_r \in \Gamma_r$ and all source points $\bm{ x }_s \in \Gamma_s$.
  
  \item \textbf{IP-5}: Determine the shape and location of the obstacle $D$
    from 
    $u^{sc}(\bm{ x }_r, \bm{ d })$
  for all observation points $\bm{ x }_r \in \Gamma_r$ and all incident directions $\bm{ d } \in \mathbb{S}$.
  
  \item \textbf{IP-6}: Determine the shape and location of the obstacle $D$
    from 
    $\partial_{\bm{n}(\bm{ x }_{r}) } u^{sc}(\bm{ x }_r,  \bm{ d })$
  for all observation points $\bm{ x }_r \in \Gamma_r$ and  all incident directions $\bm{ d } \in \mathbb{S}$.
  
  \item \textbf{IP-7}: Determine the shape and location of the obstacle $D$
    from 
    $M_{\bm{ x }_r} u^{sc}(\bm{ x }_r,  \bm{ d })$
  for all observation points $\bm{ x }_r \in \Gamma_r$ and  all incident directions $\bm{ d } \in \mathbb{S}$.
  
  \item \textbf{IP-8}: Determine the shape and location of  the obstacle $D$
    from 
    $N_{\bm{ x }_r} u^{sc}(\bm{ x }_r,  \bm{ d })$
  for all observation points $\bm{ x }_r \in \Gamma_r$ and  all incident directions $\bm{ d } \in \mathbb{S}$.
  
  \item \textbf{IP-9}: Determine the shape and location of the obstacle $D$
    from 
    $u^{\infty}(\hat{ \bm x },  \bm{ x }_{s})$
  for all observation directions $\hat{ \bm x } \in \mathbb{S}$ and  all source points $\bm{ x }_s \in \Gamma_s$.
  
  \item \textbf{IP-10}:  Determine the shape and location of the obstacle $D$
    from 
    $u^{\infty}(\hat{ \bm x }, \bm{ d })$
  for all observation directions $\hat{ \bm x } \in \mathbb{S}$ and  all incident directions  $\bm{ d } \in \mathbb{S}$.
  
  \item \textbf{IP-11}: Determine the shape and location of the obstacle $D$
    from 
    $|u(\bm{ x }_r, \bm{ x }_s)|$
  for all observation points $\bm{ x }_r \in \Gamma_r$ and all source points $\bm{ x }_s \in \Gamma_s$.
\end{itemize}

The inverse scattering problems can be divided into two classes: phased and phaseless inverse scattering problems.
For phased inverse scattering problems, numerous studies have focused on reconstructing obstacles from the scattered field generated by point sources or far-field patterns generated by plane waves. Furthermore, we  notice that some works about the direct imaging method are concerned with the case of  the Cauchy data (i.e., the scattered field and its normal derivative) generated by point sources in \cite{Harris_direct_2022} and those generated by plane waves in \cite{Nguyen_numerical_2023}. Similarly, for phaseless inverse scattering problems,  numerical reconstruction could be based on various measurement data such as the phaseless far-field data generated by point sources  \cite{yang_bayesian_2020},  the phaseless near-field data generated by plane waves \cite{zhang_approximate_2020}, the phaseless near-field data generated by point sources \cite{chen_phaseless_2017} and the phaseless far-field data generated by the superposition of the plane waves and a point source \cite{ji_acoustic_2019}.

\section{The direct imaging method with phased data}
\label{sec:phased}
This section is devoted to the direct imaging method for \textbf{IP-1} through \textbf{IP-10}. 
To address the aforementioned ten inverse problems \textbf{IP-1} through \textbf{IP-10}, the following imaging functions are defined:
\begin{align}
  \label{fun-1}
  I_1(\bm{ z }) =& -2 \kappa^4 ~\mathrm{Im}\, \int_{ \Gamma_{r} } \int_{ \Gamma_{s} } \Phi_{\kappa}(\bm{ z },\bm{ x }_r)  \Phi_{\kappa}(\bm{ z },\bm{ x }_s)
                   \overline{u^{sc}(\bm{ x }_r,\bm{ x }_s)} \;\mathrm{d}s(\bm{ x }_s)\mathrm{d}s(\bm{ x }_r), \\
   \label{fun-2}
  I_2(\bm{ z }) =&- 2 \kappa^3 ~\mathrm{Re}\; \int_{ \Gamma_{r} } \int_{ \Gamma_{s} } \Phi_{\kappa}(\bm{ z },\bm{ x }_r)  \Phi_{\kappa}(\bm{ z },\bm{ x }_s)
                   \overline{\partial_{\bm{n}(\bm{ x }_{r}) } u^{sc}(\bm{ x }_r,\bm{ x }_s)}
                   \;\mathrm{d}s(\bm{ x }_s)\mathrm{d}s(\bm{ x }_r), \\
   \label{fun-3}
   I_3(\bm{ z }) =& 2 \kappa^2 ~\mathrm{Im}\, \int_{ \Gamma_{r} } \int_{ \Gamma_{s} } \Phi_{\kappa}(\bm{ z },\bm{ x }_r)  \Phi_{\kappa}(\bm{ z },\bm{ x }_s)
                   \overline{M_{\bm{ x }_{r} } u^{sc}(\bm{ x }_r,\bm{ x }_s)}
                    \;\mathrm{d}s(\bm{ x }_s)\mathrm{d}s(\bm{ x }_r), \\
   \label{fun-4}
   I_4(\bm{ z }) =& -2 \kappa ~\mathrm{Re}\, \int_{ \Gamma_{r} } \int_{ \Gamma_{s} } \Phi_{\kappa}(\bm{ z },\bm{ x }_r)  \Phi_{\kappa}(\bm{ z },\bm{ x }_s)
                   \overline{N_{\bm{ x }_{r} } u^{sc}(\bm{ x }_r,\bm{ x }_s)}
                    \;\mathrm{d}s(\bm{ x }_s)\mathrm{d}s(\bm{ x }_r), \\
   \label{fun-5}
  I_5(\bm{ z }) =& -\frac{ \kappa^3}{4\pi} ~\mathrm{Im}\, \int_{ \Gamma_{r} } \int_{\mathbb{S}  } \Phi_{\kappa}(\bm{ z },\bm{ x }_r) e^{\mathrm{ i } \kappa \bm{ z }\cdot \bm{ d }}
                   \overline{u^{sc}(\bm{ x }_r,\bm{ d })} \;\mathrm{d}s(\bm{ d })\mathrm{d}s(\bm{ x }_r), \\
   \label{fun-6}
  I_6(\bm{ z }) =&  -\frac{ \kappa^2}{4\pi}\mathrm{Re}\, \int_{ \Gamma_{r} } \int_{\mathbb{S}  } \Phi_{\kappa}(\bm{ z },\bm{ x }_r) e^{\mathrm{ i } \kappa \bm{ z }\cdot \bm{ d }}
                   \overline{\partial_{\bm{n}(\bm{ x }_{r}) }u^{sc}(\bm{ x }_r,\bm{ d })}
                   \;\mathrm{d}s(\bm{ d })\mathrm{d}s(\bm{ x }_r), \\
   \label{fun-7}
   I_7(\bm{ z }) =&  ~ \frac{\kappa}{4\pi} \mathrm{Im}\, \int_{ \Gamma_{r} } \int_{\mathbb{S}  } \Phi_{\kappa}(\bm{ z },\bm{ x }_r) e^{\mathrm{ i } \kappa \bm{ z }\cdot \bm{ d }}
                    \overline{M_{\bm{ x }_{r} }u^{sc}(\bm{ x }_r,\bm{ d })} \;\mathrm{d}s(\bm{ d })\mathrm{d}s(\bm{ x }_r), \\
   \label{fun-8}
   I_8(\bm{ z }) =&  -\frac{1}{4\pi}\mathrm{Re}\, \int_{ \Gamma_{r} } \int_{\mathbb{S}  } \Phi_{\kappa}(\bm{ z },\bm{ x }_r) e^{\mathrm{ i } \kappa \bm{ z }\cdot \bm{ d }}
                    \overline{N_{\bm{ x }_{r} } u^{sc}(\bm{ x }_r,\bm{ d })}
                    \;\mathrm{d}s(\bm{ d })\mathrm{d}s(\bm{ x }_r), \\
   \label{fun-9}
  I_9(\bm{ z }) =& -\frac{ \kappa^3}{4\pi} ~\mathrm{Im}\,\int_{\mathbb{S}  } \int_{ \Gamma_s }  e^{-\mathrm{ i } \kappa \bm{ z }\cdot \bm{\hat{ x} }} \Phi_{\kappa}(\bm{ z },\bm{ x }_s)
                   \overline{u^{\infty}(\bm{\hat{ x} },\bm{ x }_s)} \;\mathrm{d}s(\bm{ x }_s)\mathrm{d}s(\bm{\hat{ x} }), \\
   \label{fun-10}
  I_{10}(\bm{ z }) =& -\frac{ \kappa^2}{32\pi^2} ~ \mathrm{Im}\, \int_{ \mathbb{S}  } \int_{\mathbb{S}  } e^{\mathrm{ i } \kappa \bm{ z }\cdot (\bm{ d }-\hat{\bm x})}
                   \overline{u^{\infty}(\bm{\hat{ x} },\bm{ d })} \;\mathrm{d}s(\bm{ d })\mathrm{d}s(\bm{\hat{ x} }).
\end{align}
 We start with some key lemmas that will occupy a central position  in the rigorous analysis of our imaging functions.
 \begin{lemma}[see \cite{chen_reverse_2013}]
   \label{lem:asymptotic}
   \begin{align}
     \label{eq:asy-1}
     \kappa \int_{ \Gamma_{r} }\overline{ \Phi_{\kappa}(\bm{ x },\bm{ x }_r)}\Phi_{\kappa}(\bm{ z },\bm{ x }_r) \;\mathrm{d}s(\bm{ x }_r)
     & =  \mathrm{Im}\, \Phi_{\kappa} (\bm{ x }, \bm{ z }) + w_{r,1} (\bm{ x }, \bm{ z }) \quad \forall  \bm{ x }, \bm{ z } \in  \Omega,\\
     \label{eq:asy-2}
     \kappa \int_{ \Gamma_{s} }\overline{ \Phi_{\kappa}(\bm{ x },\bm{ x }_s)}\Phi_{\kappa}(\bm{ z },\bm{ x }_s) \;\mathrm{d}s(\bm{ x }_r)
     & =  \mathrm{Im}\, \Phi_{\kappa} (\bm{ x }, \bm{ z }) + w_{s,1} (\bm{ x }, \bm{ z }) \quad \forall  \bm{ x }, \bm{ z } \in  \Omega,     
   \end{align}
   where $|\nabla_{\bm{ x }}^3w_{r,1}( \bm{ x }, \bm{ z })|\leq C R_{r}^{-1}$ and $|\nabla_{\bm{ x }}^3 w_{s,1}( \bm{ x }, \bm{ z })|\leq C R_{s}^{-1}$ hold uniformly for any $\bm{ x }$, $\bm{ z }\in \Omega$. Here we employ the notation $|\nabla^3v(\bm{ x })|=\sum_{|\alpha|=0}^3 | \partial^{|\alpha|}v(\bm{ x })/\partial \bm{ x }^{\alpha} | $.
 \end{lemma}
 Note that Lemma \ref{lem:asymptotic} is slightly different from Lemma 3.2 in \cite{chen_reverse_2013}, which is a direct consequence of the following asymptotic relation
 \begin{align}
   \label{eq:highorderradiation-1}
   \frac{\partial^{|\alpha|}}{\partial \bm{ x }^{\alpha}} \Phi_{\kappa}(\bm{ x },\bm{ x }_s) &= O \left( R_s^{-1/2} \right), \quad|\alpha|=0,1,2,3,\\
    \label{eq:highorderradiation-2}
   \frac{\partial^{|\alpha|}}{\partial \bm{ x }^{\alpha}}\big(
   \partial_{r_s} \Phi_{\kappa}(\bm{ x },\bm{ x }_s) 
   -\mathrm{ i } \kappa \Phi_{\kappa}(\bm{ x },\bm{ x }_s)\big) &= O \left( R_s^{-3/2} \right), \quad
   |\alpha|=0,1,2,3
 \end{align}
for any $\bm{ x }\in \Omega$ and  $\bm{ x }_s\in \Gamma_s$.

Applying Lemma \ref{lem:asymptotic} could immediately yield the following result.

\begin{lemma}
  \label{lem:asymptotic-2}
  The following relation hold.
     \begin{align}
       \label{eq:asy-3}
    -2 \kappa^{3}\int_{ \Gamma_{r} }\overline{ G(\bm{ x },\bm{ x }_r)}\Phi_{\kappa}(\bm{ z },\bm{ x }_r) \;\mathrm{d}s(\bm{ x }_r)
     & =  \mathrm{Im}\, \Phi_{\kappa} (\bm{ x }, \bm{ z }) + w_{r,2} (\bm{ x }, \bm{ z }) \quad \forall  \bm{ x }, \bm{ z } \in  \Omega,\\
       \label{eq:asy-4}
    -2 \kappa^{2}\int_{ \Gamma_{r} }\overline{ \partial_{\bm{n}(\bm{ x }_r) } G(\bm{ x },\bm{ x }_r)}\Phi_{\kappa}(\bm{ z },\bm{ x }_r) \;\mathrm{d}s(\bm{ x }_r)
     & =  -\mathrm{ i }~ \big( \mathrm{Im}\, \Phi_{\kappa} (\bm{ x }, \bm{ z }) + w_{r,3} (\bm{ x }, \bm{ z }) \big) \quad \forall  \bm{ x }, \bm{ z } \in  \Omega,\\
       \label{eq:asy-5}
    2 \kappa\int_{ \Gamma_{r} }\overline{M_{\bm{ x }_r} G(\bm{ x },\bm{ x }_r)}\Phi_{\kappa}(\bm{ z },\bm{ x }_r) \;\mathrm{d}s(\bm{ x }_r)
     & =  \mathrm{Im}\, \Phi_{\kappa} (\bm{ x }, \bm{ z }) + w_{r,4} (\bm{ x }, \bm{ z }) \quad \forall  \bm{ x }, \bm{ z } \in  \Omega,\\
              \label{eq:asy-6}
    2\int_{ \Gamma_{r} }\overline{N_{\bm{ x }_r} G(\bm{ x },\bm{ x }_r)}\Phi_{\kappa}(\bm{ z },\bm{ x }_r) \;\mathrm{d}s(\bm{ x }_r)
     & =  \mathrm{ i } ~ \big( \mathrm{Im}\, \Phi_{\kappa} (\bm{ x }, \bm{ z }) + w_{r,5} (\bm{ x }, \bm{ z }) \big) \quad \forall  \bm{ x }, \bm{ z } \in  \Omega,
     \end{align}
      where $|\nabla_{\bm{ x }}^3w_{r,j}( \bm{ x }, \bm{ z })|\leq C R_{r}^{-1}$ $(j=2,3,4,5)$  hold uniformly for any $\bm{ x }$, $\bm{ z }\in \Omega$.
\end{lemma}
\begin{proof}
  Note that
  \begin{align*}
    &-2 \kappa^{3}\int_{ \Gamma_{r} }\overline{ G(\bm{ x },\bm{ x }_r)}\Phi_{\kappa}(\bm{ z },\bm{ x }_r) \;\mathrm{d}s(\bm{ x }_r)\\
    = & \kappa \int_{ \Gamma_{r} }\overline{ \Phi_{\kappa}(\bm{ x },\bm{ x }_r)}\Phi_{\kappa}(\bm{ z },\bm{ x }_r) \;\mathrm{d}s(\bm{ x }_r)
    - \kappa \int_{ \Gamma_{r} }\overline{ \Phi_{\mathrm{ i } \kappa}(\bm{ x },\bm{ x }_r)}\Phi_{\kappa}(\bm{ z },\bm{ x }_r) \;\mathrm{d}s(\bm{ x }_r).
  \end{align*}
  Using the following asymptotic behavior of the Hankel
functions
  \begin{equation}
    \label{eq:antiHelmholtz}
    H^{(1)}_m(\mathrm{ i } t)=O \left( \dfrac{e^{-t }}{ t ^{1/2} } \right)\quad
     t \to +\infty, \quad m = 0,1,\cdots
  \end{equation}
and \eqref{eq:asy-1} in  Lemma \ref{lem:asymptotic}, we can prove \eqref{eq:asy-3}. The estimate of $w_{r,2}$ follows from the estimate of $w_{r,1}$, \eqref{eq:highorderradiation-1} and the asymptotic behavior \eqref{eq:antiHelmholtz}.

  By the fact that $\Phi_{\kappa}$ satisfies the Sommerfeld radiation condition, we have that
  \begin{align*}
    & -2 \kappa^{2}\int_{ \Gamma_{r} }\overline{ \partial_{\bm{n}(\bm{ x }_r) } G(\bm{ x },\bm{ x }_r)}\Phi_{\kappa}(\bm{ z },\bm{ x }_r) \;\mathrm{d}s(\bm{ x }_r) \\
    = &  \int_{ \Gamma_{r} }\overline{ \partial_{\bm{n}(\bm{ x }_r) } \Phi_{\kappa}(\bm{ x },\bm{ x }_r)} \Phi_{\kappa}(\bm{ z },\bm{ x }_r) \;\mathrm{d}s(\bm{ x }_r) -  \int_{ \Gamma_{r} }\overline{ \partial_{\bm{n}(\bm{ x }_r) } \Phi_{\mathrm{ i } \kappa}(\bm{ x },\bm{ x }_r)} \Phi_{\kappa}(\bm{ z },\bm{ x }_r) \;\mathrm{d}s(\bm{ x }_r) \\
    = &  -\mathrm{ i }\kappa~  \int_{ \Gamma_{r} }\overline{ \Phi_{\kappa}(\bm{ x },\bm{ x }_r)}\Phi_{\kappa}(\bm{ z },\bm{ x }_r) \;\mathrm{d}s(\bm{ x }_r) +
        \int_{ \Gamma_{r} }\overline{ \partial_{\bm{n}(\bm{ x }_r) } \Phi_{\kappa}(\bm{ x },\bm{ x }_r)-\mathrm{ i } \kappa  \Phi_{\kappa}(\bm{ x },\bm{ x }_r) } \Phi_{\kappa}(\bm{ z },\bm{ x }_r) \;\mathrm{d}s(\bm{ x }_r) \\
    & -  \int_{ \Gamma_{r} }\overline{ \partial_{\bm{n}(\bm{ x }_r) } \Phi_{\mathrm{ i } \kappa}(\bm{ x },\bm{ x }_r)} \Phi_{\kappa}(\bm{ z },\bm{ x }_r) \;\mathrm{d}s(\bm{ x }_r) 
     =  -\mathrm{ i }~ \big( \mathrm{Im}\, \Phi_{\kappa} (\bm{ x }, \bm{ z }) + w_{r,3} (\bm{ x }, \bm{ z }) \big),
   \end{align*}
   where we use \eqref{eq:antiHelmholtz}. The estimate of $w_{r,3}$ can be obtained by the estimate of $w_{r,1}$, \eqref{eq:highorderradiation-2} and
   the asymptotic behavior \eqref{eq:antiHelmholtz}.

   Since $\Phi_{\kappa}$ (or $\Phi_{\mathrm{ i } \kappa}$) is the fundamental solution of Helmholtz equation with wave number $\kappa$ (or $\mathrm{ i } \kappa$), we obtain
   \begin{align}
     \notag
     \Delta_{\bm{ x }_r}  G(\bm{ x },\bm{ x }_r)
     &= \frac{1}{2\kappa^{2}}\left(\Delta_{\bm{ x }_r}\Phi_{\mathrm{i}\kappa}(\bm{x},\bm{y})
       -\Delta_{\bm{ x }_r} \Phi_{\kappa}(\bm{x},\bm{y})\right)\\
     \label{eq:laplace}
     & =  \frac{1}{2\kappa^{2}}\left(\kappa^2\Phi_{\mathrm{i}\kappa}(\bm{x},\bm{y})
           +\kappa^2\Phi_{\kappa}(\bm{x},\bm{y})\right)
      = \frac{1}{2}\left(\Phi_{\mathrm{i}\kappa}(\bm{x},\bm{y})
           +\Phi_{\kappa}(\bm{x},\bm{y})\right).
   \end{align}
   A substitution implies
   \begin{align*}
     & 2 \kappa\int_{ \Gamma_{r} }\overline{M_{\bm{ x }_r} G(\bm{ x },\bm{ x }_r)}\Phi_{\kappa}(\bm{ z },\bm{ x }_r) \;\mathrm{d}s(\bm{ x }_r)\\
     =&  \kappa \int_{ \Gamma_{r} }\left[\overline{ \Phi_{\kappa}(\bm{ x },\bm{ x }_r)}
          + 
        \overline{ \Phi_{\mathrm{ i } \kappa}(\bm{ x },\bm{ x }_r)}
   + 2 \overline{\mathcal{V}_1(\bm{ x },\bm{ x }_r) }\right] \Phi_{\kappa}(\bm{ z },\bm{ x }_r) \;\mathrm{d}s(\bm{ x }_r)  \\
     =&   \mathrm{Im}\, \Phi_{\kappa} (\bm{ x }, \bm{ z }) + w_{r,1} (\bm{ x }, \bm{ z })
        + \kappa \int_{ \Gamma_{r} }\overline{ \Phi_{\mathrm{ i }\kappa}(\bm{ x },\bm{ x }_r)}\Phi_{\kappa}(\bm{ z },\bm{ x }_r) \;\mathrm{d}s(\bm{ x }_r) 
        + 2 \kappa\int_{ \Gamma_{r} }\overline{\mathcal{V}_1(\bm{ x },\bm{ x }_r) }\Phi_{\kappa}(\bm{ z },\bm{ x }_r) \;\mathrm{d}s(\bm{ x }_r)  \\
       =&   \mathrm{Im}\, \Phi_{\kappa} (\bm{ x }, \bm{ z }) + w_{r,4} (\bm{ x }, \bm{ z }),
   \end{align*}
   where $\mathcal{V}_1(\bm{ x },\bm{ x }_r):=M_{\bm{ x }_r} G(\bm{ x },\bm{ x }_r)
        -\Delta_{\bm{ x }_r}G(\bm{ x },\bm{ x }_r)$  and we use \eqref{eq:asy-1} 
   and \eqref{eq:Mpolar}.  The estimate of $w_{r,4}$ follows from the estimate of $w_{r,1}$ and the asymptotic behavior \eqref{eq:antiHelmholtz}.

   Analogously, by \eqref{eq:Npolar}, \eqref{eq:laplace} and \eqref{eq:asy-1}, 
   we have
   \begin{align*}
     & 2 \int_{ \Gamma_{r} }\overline{N_{\bm{ x }_r} G(\bm{ x },\bm{ x }_r)}\Phi_{\kappa}(\bm{ z },\bm{ x }_r) \;\mathrm{d}s(\bm{ x }_r)\\
     =&  - \int_{ \Gamma_r } \left[ \overline{ \partial_{\bm{n}(\bm{ x }_r) }\Phi_{\kappa}(\bm{ x },\bm{ x }_r)}
        - \overline{ \partial_{\bm{n}(\bm{ x }_r) }\Phi_{\mathrm{ i }\kappa}(\bm{ x },\bm{ x }_r)}     + 2 \overline{\mathcal{V}_2(\bm{ x },\bm{ x }_r)} \right]
        \Phi_{\kappa}(\bm{ z },\bm{ x }_r) \;\mathrm{d}s(\bm{ x }_r)\\
     = & \mathrm{ i } \kappa \int_{ \Gamma_{r} }\overline{ \Phi_{\kappa}(\bm{ x },\bm{ x }_r)}\Phi_{\kappa}(\bm{ z },\bm{ x }_r) \;\mathrm{d}s(\bm{ x }_r)
         -   \int_{ \Gamma_{r} }\overline{ \partial_{\bm{n}(\bm{ x }_r) }\Phi_{\kappa}(\bm{ x },\bm{ x }_r) - \mathrm{ i } \kappa\Phi_{\kappa}(\bm{ x },\bm{ x }_r)}\Phi_{\kappa}(\bm{ z },\bm{ x }_r) \;\mathrm{d}s(\bm{ x }_r) \\
     &     -  \int_{ \Gamma_{r} }\overline{ \partial_{\bm{n}(\bm{ x }_r) }\Phi_{\mathrm{ i }\kappa}(\bm{ x },\bm{ x }_r)}\Phi_{\kappa}(\bm{ z },\bm{ x }_r) \;\mathrm{d}s(\bm{ x }_r)
    +  2\int_{ \Gamma_{r} }\overline{ \mathcal{V}_2(\bm{ x },\bm{ x }_r)} \Phi_{\kappa}(\bm{ z },\bm{ x }_r) \;\mathrm{d}s(\bm{ x }_r) \\
     = &   \mathrm{ i } ~ \big(  \mathrm{Im}\, \Phi_{\kappa} (\bm{ x }, \bm{ z }) + w_{r,5} (\bm{ x }, \bm{ z }) \big).
   \end{align*}
   where  $\mathcal{V}_2(\bm{ x },\bm{ x }_r):=N_{\bm{ x }_r} G(\bm{ x },\bm{ x }_r)
        +\partial_{\bm{n}(\bm{ x }_{r})} \Delta_{\bm{ x }_r}G(\bm{ x },\bm{ x }_r)$.
The estimate of $w_{r,5}$ can be obtained by the estimate of $w_{r,1}$, \eqref{eq:highorderradiation-2} and
the asymptotic behavior \eqref{eq:antiHelmholtz}.

   Therefore, the proof is complete.
\end{proof}

We also need the following Funk-Hecke formula (c.f. \cite[Lemma 2.7]{liu_novel_2017}) in our analysis.
\begin{lemma}
  \label{lem:plane}
  For any $\xi, \bm{ z } \in  \mathbb{R}^2$, we have
  \begin{align}
    \label{eq:plane}
    \int_{ \mathbb{S}  } e^{\mathrm{ i } \kappa \hat{\bm x}\cdot(\bm{ z }-\xi)}\;\mathrm{d}s( \hat{\bm x})
    = 8\pi ~ \mathrm{Im}\,\Phi_{\kappa}(\xi,\bm{ z }).
  \end{align}
\end{lemma}
Now we are in a position to  present the mathematical analysis of these imaging functions for  extended obstacles in the following theorems.

\begin{theorem}
  For any $\bm{ z }\in \Omega$, it holds
     \begin{equation*}
       I_{j}(\bm{ z }) =
       \int_{ \mathbb{S}  } | \psi^{\infty}(\cdot, \bm{ z } ) | ^2 \;\mathrm{d}s(\cdot)
       +  w_j(\bm{ z }) \quad  \text{for} \quad j=1,2,\cdots,10,
     \end{equation*}
     where $\psi^{\infty}(\cdot, \bm{ z } )$ is the far-field pattern of $\psi(\cdot, \bm{ z } )$, which is defined as the solution to the  following boundary value problem:
     \begin{subequations}\notag
    \begin{align}
\Delta^{2}_{\bm{ x }}\psi(\bm{ x }, \bm{ z }) - \kappa^4 \psi(\bm{ x }, \bm{ z })  & =   0  \quad \mathrm{ in } ~ D^c, \\
( \mathcal{B}_1  \psi({ x }, \bm{ z }), \mathcal{B}_{2} \psi(\bm{ x }, \bm{ z }))    &  =  - (\mathcal{B}_1 \mathrm{Im}\, \Phi(\bm{ x }, \bm{ z }), \mathcal{B}_2  \mathrm{Im}\, \Phi(\bm{ x }, \bm{ z }))  \quad \mathrm{ on } ~ \Gamma,  \\
      \partial_{r}\psi(\bm{ x }, \bm{ z }) - \mathrm{ i } \kappa \psi(\bm{ x }, \bm{ z })  &  =  o \left( \dfrac{1}{\sqrt{r}} \right) \quad r =| \bm{ x } | \to \infty.
\end{align}
\end{subequations}
     Furthermore, the following inequalities hold
     \begin{align*}
           \| w_j \|_{L^{\infty}(\Omega)} &\leq C(R_r^{-1}+ R_s^{-1}) \quad \text{for}\quad  j=1,\cdots,4,\\
     \| w_j \|_{L^{\infty}(\Omega)} &\leq C R_s^{-1} \quad \text{for}\quad  j=5,\cdots,8, \quad 
     \| w_9 \|_{L^{\infty}(\Omega)} \leq C R_r^{-1} ,\quad w_{10} = 0.
     \end{align*}

   \end{theorem}
   \begin{remark}
     The imaginary part of the fundamental solution $\mathrm{Im}\,\Phi_{\kappa}(\xi ,\bm{ z })$, i.e., the Bessel function of order zero, attains the largest value  when $\bm{ z }$ is equal to  $\xi$ and  has small values  when $\bm{ z }$ is far from $\xi$. However, $\partial_{\bm{n} (\xi)} \mathrm{Im}\,\Phi_{\kappa}(\xi ,\bm{ z })$ can attains larger value when $\bm{ z }$ is near $\bm{ \xi  }$ and has small values when $\bm{ z }$ is far from $\xi$. Note that $\psi(\xi ,\bm{ z })$ is the solution to \eqref{eq:bvp-1}-\eqref{eq:bvp-3}  with boundary data $(f,g)=-(\mathrm{Im}\,\Phi_{\kappa}(\xi ,\bm{ z }),\partial_{\bm{n} (\xi)} \mathrm{Im}\,\Phi_{\kappa}(\xi ,\bm{ z }))$. Therefore, the imaging functions $I_j(\bm{ z })$ $(j=1,2,\cdots,10)$ will exhibit a distinct contrast near the boundary $\Gamma$ of the scatterer $D$ and decay  with distance from the boundary $\Gamma$, which will be confirmed in the numerical examples in Section \ref{sec:numerical}.
   \end{remark}
   \begin{proof}
     We will present the proof of the cases  of boundary conditions  \uppercase\expandafter{\romannumeral 1} and  \uppercase\expandafter{\romannumeral 2}. The cases of boundary conditions  \uppercase\expandafter{\romannumeral 3} and  \uppercase\expandafter{\romannumeral 4} can be treated similarly and to avoid repetition, we omit the proof.
     
\emph{Proof of}  the case of boundary condition \uppercase\expandafter{\romannumeral 1}.
     
       The proof is divided into five cases.

       \textbf{Case 1:} $j=1$ and $5$.

  Applying \eqref{eq:Green-rep} in Lemma \ref{lem:rep} to $w(\xi)=u^{sc}(\xi,\bm{x}_s)$ ($ \bm{ x }_s \in \Gamma_s$) or $u^{sc}(\xi,\bm{d})$ ($\bm{d}\in \mathbb{S} $) could yield
  \begin{align}
    \notag
w(\bm{x}_r) = 
\int_{\partial D}\Big\{ &[G(\xi,\bm{x}_r) \,N_{\xi} w(\xi )
                + \partial_{\bm{n}(\xi)}G(\xi,\bm{x}_r)\, M_{\xi} w(\xi )  ]\\
    \label{eq:eq:Green-rep1}
&-[N_{\xi}G(\xi,\bm{x}_r)\, w(\xi )
+ M_{\xi}G(\xi,\bm{x}_r) \,\partial_{\bm{n}(\xi)} w(\xi ) ] 
  \Big\} {\;\rm d}s(\xi) \quad \bm{x}_r\in\Gamma_r.
  \end{align}
  Then exchanging the order of integration and using \eqref{eq:asy-3} in  Lemma \ref{lem:asymptotic-2}, we obtain
  \begin{align*}
   & -2 \kappa^3 \int_{ \Gamma_{r} } \Phi_{\kappa}(\bm{ z },\bm{ x }_r) \overline{w(\bm{ x }_{r})} \;\mathrm{d}s(\bm{ x }_r) \\
    = & \int_{ \Gamma } \Big\{ [M_{\xi}\overline{w(\xi)} \, \partial_{\bm{n}(\xi )}
        (\mathrm{Im}\, \Phi_{\kappa}(\xi,\bm{z}) + w_{r,2}(\xi,\bm{z}) )
        + N_{\xi } \overline{w(\xi)}\,
         (\mathrm{Im}\, \Phi_{\kappa}(\xi,\bm{z}) + w_{r,2}(\xi,\bm{z}) ) ]  \\
   & \qquad - [ \partial_{\bm{n}(\xi )}\overline{w(\xi)} M_{\xi }\,
     ( \mathrm{Im}\, \Phi_{\kappa}(\xi,\bm{z}) + w_{r,2}(\xi,\bm{z})  )
     +  \overline{w(\xi)} \, N_{\xi }
     ( \mathrm{Im}\, \Phi_{\kappa}(\xi,\bm{z}) + w_{r,2}(\xi,\bm{z}) )
     ]  \Big\} \;\mathrm{d}s(\xi).
  \end{align*}
  Define
  \begin{equation}
    \label{eq:vsd}
    v_{s}(\xi , \bm{ z })  =  \kappa \int_{ \Gamma_{s} } \Phi_{\kappa}(\bm{ z }, \bm{ x }_s)
    \overline{u^{sc}(\xi ,\bm{ x }_s)} \;\mathrm{d}s(\bm{ x }_s)
    \quad \text{and} \quad
    v_{d}(\xi , \bm{ z })  =  \frac{1}{8\pi} \int_{ \Gamma_{s} } e^{\mathrm{ i } \kappa \bm{ z }\cdot \bm{ d }}
    \overline{u^{sc}(\xi ,\bm{ d })} \;\mathrm{d}s(\bm{ d }).
  \end{equation}
  The imaging functions $I_j (\bm{ z })$ ($j=1,5$) can be rewritten as follows:
  \begin{align}
    \notag
    I_j(\bm{ z }) 
    = &\mathrm{Im}\, \int_{ \Gamma } \Big\{ [M_{\cdot }\overline{v(\cdot )}\, \partial_{\bm{n}(\cdot  )}
        (\mathrm{Im}\, \Phi_{\kappa}(\cdot ,\bm{z}) + w_{r,2}(\cdot,\bm{z}) )
        +N_{\cdot  } \overline{v(\cdot)}\,
        (\mathrm{Im}\, \Phi_{\kappa}(\cdot ,\bm{z}) + w_{r,2}(\cdot ,\bm{z}) ) ]  \\
    \label{eq:imaging}
   & \qquad - [ \partial_{\bm{n}(\cdot  )}\overline{v(\cdot )}\,  M_{\cdot  }
     ( \mathrm{Im}\, \Phi_{\kappa}(\cdot ,\bm{z}) + w_{r,2}(\cdot ,\bm{z})  )
     + \overline{v(\cdot )} \, N_{\cdot }
     ( \mathrm{Im}\, \Phi_{\kappa}(\cdot ,\bm{z}) + w_{r,2}(\cdot ,\bm{z}) )
     ]  \Big\} \;\mathrm{d}s(\cdot ),
  \end{align}
  where we set $v(\xi)=v_{s}(\xi , \bm{ z })$ or $ v_{d}(\xi , \bm{ z })$ for brevity.
  
  Next, we claim that
  \begin{equation}
    \label{eq:decomposition}
    v_s(\xi , \bm{ z })=\overline{\psi(\xi ,\bm{ z })} +\overline{ \zeta(\xi ,\bm{ z })}
    \quad \text{and} \quad
    v_d(\xi , \bm{ z })=\overline{\psi(\xi ,\bm{ z })},
  \end{equation}
  where $\zeta(\cdot ,\bm{ z })$ is a  solution of \eqref{eq:bvp-1}-\eqref{eq:bvp-3} with $f=-w_s(\cdot ,\bm{ z })$ and $g=-\partial_{\bm{n}(\cdot )}w_s(\cdot ,\bm{ z }) $. It follows from Lemma \ref{lem:well} that
  \begin{equation}    \label{eq:zeta}
    \|\zeta(\cdot ,\bm{ z } )\|_{H^{3/2}(\Gamma)} + \|\partial_{\bm{n}(\cdot )} \zeta(\cdot ,\bm{ z } )\|_{H^{1/2}(\Gamma)}   
       +\|M_{\cdot} \zeta(\cdot ,\bm{ z })\|_{H^{-1/2}(\Gamma)} + \|N_{\cdot} \zeta(\cdot ,\bm{ z }) \|_{H^{-3/2}(\Gamma)}
        \le C R_s^{-1}
  \end{equation}
  holds  uniformly for any $\bm{ z }\in \Omega$.
  
  We note that the scattered field $u^{sc}$ satisfies \eqref{eq:bi} and \eqref{eq:radiation}, which implies that the complex conjugate $\overline{v(\xi)}$ also satisfies the biharmonic wave equation \eqref{eq:bi}, i.e.,
  \begin{equation*}
\Delta_{\xi }^2 \overline{v(\xi )} -\kappa^4 \overline{v(\xi )} = 0 \quad \mathrm{ in  } ~ D^{c},
  \end{equation*}
  and the Sommerfeld radiation condition \eqref{eq:radiation}. By the boundary conditions \eqref{eq:boundary} and  Lemma \ref{lem:asymptotic}, we have
  \begin{align*}
   \overline{v_{s}(\xi , \bm{ z })} = & \kappa \int_{ \Gamma_{s} } \overline{ \Phi_{\kappa}(\bm{ z },\bm{ x }_s)}u^{sc}(\xi,\bm{ x }_s)  \;\mathrm{d}s(\bm{ x }_s)  \\
    = & -  \kappa \int_{ \Gamma_{s} }\overline{ \Phi_{\kappa}(\bm{ z },\bm{ x }_s)}\Phi_{\kappa}(\xi ,\bm{ x }_s) \;\mathrm{d}s(\bm{ x }_s)
        = -( \mathrm{Im}\, \Phi_{\kappa} (\xi, \bm{ z }) + w_{s,1} (\xi , \bm{ z }) ) \quad \xi \in \Gamma,    \\
    \partial_{\bm{n}(\xi)}    \overline{v_{s}(\xi , \bm{ z })} = & \kappa \int_{ \Gamma_{s} } \overline{ \Phi_{\kappa}(\bm{ z },\bm{ x }_s)} \partial_{\bm{n}(\xi)}  u^{sc}(\xi,\bm{ x }_s)   \;\mathrm{d}s(\bm{ x }_s) \\
    = & -  \kappa \int_{ \Gamma_{s} }\overline{ \Phi_{\kappa}(\bm{ z },\bm{ x }_s)} \partial_{\bm{n}(\xi)} \Phi_{\kappa}(\xi ,\bm{ x }_s) \;\mathrm{d}s(\bm{ x }_s)
      = - \partial_{\bm{n}(\xi)} ( \mathrm{Im}\, \Phi_{\kappa} (\xi, \bm{ z }) + w_{s,1} (\xi , \bm{ z }) ) \quad \xi \in \Gamma.  
  \end{align*}
  Alternatively, from Lemma \ref{lem:plane} we conclude that
    \begin{align*}
   \overline{v_{d}(\xi , \bm{ z })} = &  \frac{1}{8\pi} \int_{ \mathbb{S} } \overline{  e^{\mathrm{ i } \kappa \bm{ z }\cdot \bm{ d }}}u^{sc}(\xi,\bm{ d })  \;\mathrm{d}s(\bm{ d }) 
    =  -  \frac{1}{8\pi} \int_{ \mathbb{S} } e^{\mathrm{ i } \kappa(\xi - \bm{ z })\cdot \bm{ d }} \;\mathrm{d}s(\bm{ d })
        = -\mathrm{Im}\, \Phi_{\kappa} (\xi, \bm{ z }) \quad \xi \in \Gamma,    \\
    \partial_{\bm{n}(\xi)}    \overline{v_{d}(\xi , \bm{ z })} = &  \frac{1}{8\pi} \int_{\mathbb{S} } \overline{ e^{\mathrm{ i } \kappa \bm{ z }\cdot \bm{ d }}} \partial_{\bm{n}(\xi)}  u^{sc}(\xi,\bm{ d })   \;\mathrm{d}s(\bm{ d }) \\
      = & -  \frac{1}{8\pi} \partial_{\bm{n}(\xi)}  \int_{ \mathbb{S} }  e^{\mathrm{ i } \kappa(\xi - \bm{ z })\cdot \bm{ d }}  \;\mathrm{d}s(\bm{ d })
      = - \partial_{\bm{n}(\xi)}  \mathrm{Im}\, \Phi_{\kappa} (\xi, \bm{ z }) \quad \xi \in \Gamma.  
  \end{align*}
  Hence, the well-posedness of the direct problem, i.e., Theorem \ref{lem:well}, shows that  $\overline{v_s(\xi , \bm{ z })}=\psi(\xi ,\bm{ z }) + \zeta(\xi ,\bm{ z })$ and  $\overline{v_d(\xi , \bm{ z })}=\psi(\xi ,\bm{ z })$, which immediately leads to the decomposition \eqref{eq:decomposition}.

  Substituting the decomposition \eqref{eq:decomposition} into \eqref{eq:imaging}, we find that
    \begin{align}
    I_j(\bm{ z }) 
    \notag
    = &\mathrm{Im}\, \int_{ \Gamma } \Big\{ [M_{\xi}\overline{\psi(\xi ,\bm{ z })}\, \partial_{\bm{n}(\xi )}
        (\mathrm{Im}\, \Phi_{\kappa}(\xi,\bm{z})  ) 
        +N_{\xi } \overline{\psi(\xi ,\bm{ z })} \,
        (\mathrm{Im}\, \Phi_{\kappa}(\xi,\bm{z}) ) ]  \\
    \label{eq:imaging-1}
   & \qquad - [ \partial_{\bm{n}(\xi )}\overline{\psi(\xi ,\bm{ z })}\,  M_{\xi }
     ( \mathrm{Im}\, \Phi_{\kappa}(\xi,\bm{z})  )
     + \overline{\psi(\xi ,\bm{ z })}\,  N_{\xi }
     ( \mathrm{Im}\, \Phi_{\kappa}(\xi,\bm{z})  )
     ]  \Big\} \;\mathrm{d}s(\xi)
     +w_j,
    \end{align}
    for $j=1,5$. Here $\| w_{1} \|_{L^{\infty}(\Omega)} \leq C(R_r^{-1}+ R_s^{-1})$ and $\| w_{5} \|_{L^{\infty}(\Omega)}
    \leq C R_s^{-1}$.
    Since  $\psi(\xi ,\bm{ z })$ satisfies the clamped plate boundary condition, i.e.,
\begin{equation*}
  \psi(\xi ,\bm{ z }) = -  \mathrm{Im}\, \Phi_{\kappa}(\xi,\bm{z}) \quad \mathrm{ and } \quad
 \partial_{\bm{n}(\xi )} \psi(\xi ,\bm{ z }) = - \partial_{\bm{n}(\xi )} \mathrm{Im}\, \Phi_{\kappa}(\xi,\bm{z}) \quad
  \xi \in \Gamma,
\end{equation*}
    applying integration by parts yields
    \begin{align*}
      &\mathrm{Im}\, \int_{ \Gamma }[ \partial_{\bm{n}(\xi )}\overline{\psi(\xi ,\bm{ z })} \,
        M_{\xi } ( \mathrm{Im}\, \Phi_{\kappa}(\xi,\bm{z})  )
        + \overline{\psi(\xi ,\bm{ z })}\,
        N_{\xi } ( \mathrm{Im}\, \Phi_{\kappa}(\xi,\bm{z})  )
     ] \;\mathrm{d}s(\xi)\\
      =&  - \mathrm{Im}\, \int_{ \Gamma }[ \partial_{\bm{n}(\xi )} \mathrm{Im}\, \Phi_{\kappa}(\xi,\bm{z}) \,
         M_{\xi } ( \mathrm{Im}\, \Phi_{\kappa}(\xi,\bm{z})  )
        + \mathrm{Im}\, \Phi_{\kappa}(\xi,\bm{z}) \,
       N_{\xi }( \mathrm{Im}\, \Phi_{\kappa}(\xi,\bm{z})  )
        ] \;\mathrm{d}s(\xi)\\
        =& - \mathrm{Im}\, \int_{ D} [|\Delta_{\xi } ( \mathrm{Im}\, \Phi_{\kappa}(\xi,\bm{z})  )|^2
        -\kappa^4 |\mathrm{Im}\, \Phi_{\kappa}(\xi,\bm{z})|^2
     ] \;\mathrm{d}\xi =0,
    \end{align*}
    which leads to
        \begin{align*}
    \notag
     I_j(\bm{ z })
          = &\mathrm{Im}\, \int_{ \Gamma }[M_{\xi}\overline{\psi(\xi ,\bm{ z })}\,
              \partial_{\bm{n}(\xi )} (\mathrm{Im}\, \Phi_{\kappa}(\xi,\bm{z})  )
        +N_{\xi } \overline{\psi(\xi ,\bm{ z })} \,
        (\mathrm{Im}\, \Phi_{\kappa}(\xi,\bm{z}) ) ]   \;\mathrm{d}s(\xi) + w_j\\
     = & -\mathrm{Im}\, \int_{ \Gamma }[M_{\xi}\overline{\psi(\xi ,\bm{ z })} \,
              \partial_{\bm{n}(\xi )} \psi(\xi ,\bm{ z })
        + N_{\xi } \overline{\psi(\xi ,\bm{ z })} \,
              \psi(\xi ,\bm{ z }) ]   \;\mathrm{d}s(\xi) + w_j\\
     = & \frac{\kappa^{2}}{4\pi} \int_{ \mathbb{S}  } | \psi ^{\infty}(\cdot ,\bm{ z }) | ^{2} \;\mathrm{d}s(\cdot ) + w_j \quad \text{for} \quad j=1,5,
        \end{align*}
        where we use Lemma \ref{lem:sign}.

        \textbf{Case 2:} $j=2$ and $6$.

        First, using the definitions of $v_s(\xi ,\bm{ z })$ and $v_d(\xi ,\bm{ d })$  in \eqref{eq:vsd}, we show that for $\bm{ z }\in \Omega$ and $j=2,6$, it holds that
          \begin{align}
    \notag
             I_j(\bm{ z }) 
            = &\mathrm{Im}\, \int_{ \Gamma } \Big\{ [M_{\cdot}\overline{v(\cdot )}\, \partial_{\bm{n}(\cdot  )}
        (\mathrm{Im}\, \Phi_{\kappa}(\cdot ,\bm{z}) + w_{r,3}(\cdot,\bm{z}) )
        + N_{\cdot  } \overline{v(\cdot)}\,
        (\mathrm{Im}\, \Phi_{\kappa}(\cdot ,\bm{z}) + w_{r,3}(\cdot ,\bm{z}) ) ]  \\
    \label{eq:imaging-2}
   & \qquad - [ \partial_{\bm{n}(\cdot  )}\overline{v(\cdot )}\,  M_{\cdot  }
     ( \mathrm{Im}\, \Phi_{\kappa}(\cdot ,\bm{z}) + w_{r,3}(\cdot ,\bm{z})  )
     + \overline{v(\cdot )} \,  N_{\cdot }
     ( \mathrm{Im}\, \Phi_{\kappa}(\cdot ,\bm{z}) + w_{r,3}(\cdot ,\bm{z}) )
     ]  \Big\} \;\mathrm{d}s(\cdot ),
  \end{align}
  where $v(\xi)=v_s(\xi ,\bm{ z }) $ when $j=2$ or $v(\xi )= v_d(\xi ,\bm{ d })$ when $j=6$.

  Recalling the representation \eqref{eq:Green-rep} in Lemma \ref{lem:rep} and considering the normal derivative of $w(\xi)=u^{sc}(\xi,\bm{x}_s)$ ($ \bm{ x }_s \in \Gamma_s$) or $u^{sc}(\xi,\bm{d})$ ($\bm{d}\in \mathbb{S} $), we find that
    \begin{align*}
\partial_{\bm{n}(\bm{ x }_r)} w(\bm{x}_r) = 
\int_{\partial D}\Big\{ &[\partial_{\bm{n}(\bm{ x }_r)}G(\xi,\bm{x}_r)\, N_{\xi} w(\xi )
                + \partial_{\bm{n}(\xi)}\partial_{\bm{n}(\bm{ x }_r)}G(\xi,\bm{x}_r)\, M_{\xi} w(\xi )  ]\\
&-[N_{\xi}\partial_{\bm{n}(\bm{ x }_r)} G(\xi,\bm{x}_r)\, w(\xi )
+ M_{\xi}\partial_{\bm{n}(\bm{ x }_r)}G(\xi,\bm{x}_r)\, \partial_{\bm{n}(\xi)} w(\xi ) ] 
  \Big\} {\;\rm d}s(\xi) \quad \bm{x}_r\in\Gamma_r.
  \end{align*}
  In view of  \eqref{eq:asy-2} in Lemma \ref{lem:asymptotic-2}  one has
    \begin{align*}
   & -2 \kappa^2 \int_{ \Gamma_{r} } \Phi_{\kappa}(\bm{ z },\bm{ x }_r) \overline{ \partial_{\bm{n}(\bm{ x }_r)}w(\bm{ x }_{r})} \;\mathrm{d}s(\bm{ x }_r) \\
    = &-\mathrm{ i }  \int_{ \Gamma } \Big\{ [M_{\xi}\overline{w(\xi)}\, \partial_{\bm{n}(\xi )}
        (\mathrm{Im}\, \Phi_{\kappa}(\xi,\bm{z}) + w_{r,3}(\xi,\bm{z}) )
        +N_{\xi } \overline{w(\xi)} \,
         (\mathrm{Im}\, \Phi_{\kappa}(\xi,\bm{z}) + w_{r,3}(\xi,\bm{z}) ) ]  \\
   & \qquad - [ \partial_{\bm{n}(\xi )}\overline{w(\xi)}\,  M_{\xi }
     ( \mathrm{Im}\, \Phi_{\kappa}(\xi,\bm{z}) + w_{r,3}(\xi,\bm{z})  )
     + \overline{w(\xi)} \, N_{\xi }
     ( \mathrm{Im}\, \Phi_{\kappa}(\xi,\bm{z}) + w_{r,3}(\xi,\bm{z}) )
     ]  \Big\} \;\mathrm{d}s(\xi).
  \end{align*}
  Similarly, combining a substitution and the definitions of  function $v$ and imaging functions $I_j$ ($j=3,7$), we can obtain \eqref{eq:imaging-2}, where we use the fact that $\mathrm{Re}\;(-\mathrm{ i } c )=\mathrm{Im}\, c$ for any complex number $c \in \mathbb{C}$.

  Next, the decomposition of $v$, i.e., \eqref{eq:decomposition}, and  \eqref{eq:imaging-2}  can give \eqref{eq:imaging-1} with $w_1$ (or $w_5$) replaced by $w_2$ (or $w_6$). Proceeding as in case 1 immediately can prove the conclusion.

  \textbf{Case 3:} $j=3$ and $7$.

  It remains to prove   \eqref{eq:imaging-2}  with $w_{r,3}$ replaced by $w_{r,4}$.
  It follows from the representation \eqref{eq:Green-rep} in Lemma \ref{lem:rep} that
      \begin{align*}
M_{\bm{x}_r} w(\bm{x}_r) = 
\int_{\partial D}\Big\{ &[M_{\bm{x}_r}G(\xi,\bm{x}_r)\, N_{\xi} w(\xi )
                + \partial_{\bm{n} (\xi)} M_{\bm{ x }_r} G(\xi,\bm{x}_r)\, M_{\xi} w(\xi )  ]\\
&-[N_{\xi}M_{\bm{x}_r}G(\xi,\bm{x}_r)\, w(\xi )
+ M_{\xi}M_{\bm{x}_r}G(\xi,\bm{x}_r)\, \partial_{\bm{n}(\xi)} w(\xi ) ] 
  \Big\} {\;\rm d}s(\xi) \quad \bm{x}_r\in\Gamma_r,
  \end{align*}
  where $w(\xi)$ is $ u^{sc}(\xi,\bm{x}_s)$  or $u^{sc}(\xi,\bm{d})$.
  Using \eqref{eq:asy-5} in Lemma \ref{lem:asymptotic-2}, we get
      \begin{align*}
   & 2 \kappa \int_{ \Gamma_{r} } \Phi_{\kappa}(\bm{ z },\bm{ x }_r) \overline{ M_{\bm{x}_r}w(\bm{ x }_{r})} \;\mathrm{d}s(\bm{ x }_r) \\
        = &  \int_{ \Gamma } \Big\{ [M_{\xi}\overline{w(\xi)} \,
            \partial_{\bm{n}(\xi )}
        (\mathrm{Im}\, \Phi_{\kappa}(\xi,\bm{z}) + w_{r,4}(\xi,\bm{z}) )
        +N_{\xi } \overline{w(\xi)} \,
         (\mathrm{Im}\, \Phi_{\kappa}(\xi,\bm{z}) + w_{r,4}(\xi,\bm{z}) ) ]  \\
   & \qquad - [ \partial_{\bm{n}(\xi )}\overline{w(\xi)}\,
     M_{\xi }
     ( \mathrm{Im}\, \Phi_{\kappa}(\xi,\bm{z}) + w_{r,4}(\xi,\bm{z})  )
     + \overline{w(\xi)}\,  N_{\xi }
     ( \mathrm{Im}\, \Phi_{\kappa}(\xi,\bm{z}) + w_{r,4}(\xi,\bm{z}) )
     ]  \Big\} \;\mathrm{d}s(\xi), 
  \end{align*}
which proves  \eqref{eq:imaging-2}  with $w_{r,3}$ replaced by $w_{r,4}$   due to  the definitions of  function $v$ and imaging functions $I_j$ ($j=3,7$).

 \textbf{Case 4:} $j=4$ and $8$. 

 It remains to prove   \eqref{eq:imaging-2}  with $w_{r,3}$ replaced by $w_{r,5}$.
 Using the representation \eqref{eq:Green-rep} in Lemma \ref{lem:rep}, we see that
   \begin{align*}
N_{\bm{x}_r} w(\bm{x}_r) = 
\int_{\partial D}\Big\{ &[N_{\bm{x}_r}G(\xi,\bm{x}_r) \, N_{\xi} w(\xi )
                + \partial_{\bm{n}(\xi)}N_{\bm{x}_r}G(\xi,\bm{x}_r) \, M_{\xi} w(\xi )  ]\\
&-[N_{\xi}N_{\bm{x}_r}G(\xi,\bm{x}_r)\, w(\xi )
+ M_{\xi}N_{\bm{x}_r}G(\xi,\bm{x}_r)\,  \partial_{\bm{n}(\xi)} w(\xi ) ] 
  \Big\} {\;\rm d}s(\xi) ,
  \end{align*} 
  where $\bm{x}_r\in\Gamma_r$ and  $w(\xi)$ is $ u^{sc}(\xi,\bm{x}_s)$  or $u^{sc}(\xi,\bm{d})$.
Then we employ \eqref{eq:asy-6} in  Lemma \ref{lem:asymptotic-2} to get
    \begin{align*}
   & -2 \kappa \int_{ \Gamma_{r} } \Phi_{\kappa}(\bm{ z },\bm{ x }_r) \overline{ N_{\bm{x}_r}w(\bm{ x }_{r})} \;\mathrm{d}s(\bm{ x }_r) \\
    =-\mathrm{ i }  &  \int_{ \Gamma } \Big\{ [M_{\xi}\overline{w(\xi)} \, \partial_{\bm{n}(\xi )}
        (\mathrm{Im}\, \Phi_{\kappa}(\xi,\bm{z}) + w_{r,5}(\xi,\bm{z}) )
        +N_{\xi } \overline{w(\xi)} \,
         (\mathrm{Im}\, \Phi_{\kappa}(\xi,\bm{z}) + w_{r,5}(\xi,\bm{z}) ) ]  \\
   & \qquad - [ \partial_{\bm{n}(\xi )}\overline{w(\xi)}\, M_{\xi }
     ( \mathrm{Im}\, \Phi_{\kappa}(\xi,\bm{z}) + w_{r,5}(\xi,\bm{z})  )
     - \overline{w(\xi)}  \, N_{\xi }
     ( \mathrm{Im}\, \Phi_{\kappa}(\xi,\bm{z}) + w_{r,5}(\xi,\bm{z}) )
     ]  \Big\} \;\mathrm{d}s(\xi). 
  \end{align*}
By means of this, \eqref{eq:vsd} and the fact that $\mathrm{Re}\;( -\mathrm{ i } c)=\mathrm{Im}\,c$ for $c \in \mathbb{C}$, the imaging functions $I_j$ ($j=4,8$) can be transformed into  \eqref{eq:imaging-2}  with $w_{r,3}$ replaced by $w_{r,5}$.

   \textbf{Case 5:} $j=5$ and $10$.

    It remains to prove   \eqref{eq:imaging-2}  with $w_{r,3}$ replaced by $0$.
    By \eqref{eq:far-rep} in Lemma \ref{lem:rep}, the far-field pattern $w^{\infty}$ of  $w$ is given by
    \begin{align*}
        w^{\infty}(\hat{\bm x}) = & -\frac{1}{2\kappa^2}
                \int_{\Gamma}\Big\{
                  [  e^{-\mathrm{i}\kappa\hat{\bm x}\cdot \xi}\, N_{\xi} w(\xi)
                           + \partial_{\bm{n}(\xi)}e^{-\mathrm{i}\kappa\hat{\bm x}\cdot \xi} \, M_{\xi} w(\xi) ] \\
             & \qquad\qquad\quad   - [ N_{\xi}e^{-\mathrm{i}\kappa\hat{\bm x}\cdot \xi} \,   w(\xi)
                    + M_{\xi}e^{-\mathrm{i}\kappa\hat{\bm x}\cdot \xi}\,  \partial_{\bm{n}(\xi)} w(\xi) ]
                \Big\} \;\mathrm{d}s(\xi) \quad \hat{\bm x}\in \mathbb{S},
    \end{align*}
    where $w(\xi)$ denotes  $ u^{sc}(\xi,\bm{x}_s)$  or $u^{sc}(\xi,\bm{d})$. From Lemma \ref{lem:plane}, we observe hat
    \begin{align*}
      & \frac{-2 \kappa^2}{8\pi}\int_{ \mathbb{S}  }  e^{-\mathrm{i}\kappa\hat{\bm x}\cdot \bm{ z }} \overline{w^{\infty}(\hat{\bm{ x }})} \;\mathrm{d}s (\hat{\bm x}) \\
       =&  \int_{ \Gamma } \Big\{ [M_{\xi}\overline{w(\xi)}\,  \partial_{\bm{n}(\xi )}
        (\frac{1}{8 \pi}\int_{ \mathbb{S}  }e^{\mathrm{i}\kappa\hat{\bm x}\cdot (\xi-\bm{ z })}  \;\mathrm{d}s (\hat{\bm x}) )
        +N_{\xi } \overline{w(\xi)} \,
         (\frac{1}{8 \pi}\int_{ \mathbb{S}  }e^{\mathrm{i}\kappa\hat{\bm x}\cdot (\xi-\bm{ z })} \;\mathrm{d}s (\hat{\bm x}) ) ]  \\
   & \qquad - [ \partial_{\bm{n}(\xi )}\overline{w(\xi)} \, M_{\xi }
   (\frac{1}{8 \pi}\int_{ \mathbb{S}  }e^{\mathrm{i}\kappa\hat{\bm x}\cdot (\xi-\bm{ z })} \;\mathrm{d}s (\hat{\bm x})  )
     + \overline{w(\xi)}\,  N_{\xi }
   (\frac{1}{8 \pi}\int_{ \mathbb{S}  }e^{\mathrm{i}\kappa\hat{\bm x}\cdot (\xi-\bm{ z })} \;\mathrm{d}s (\hat{\bm x})  )
     ]  \Big\} \;\mathrm{d}s(\xi)  \\
      =&  \int_{ \Gamma } \Big\{ [M_{\xi}\overline{w(\xi)} \partial_{\bm{n}(\xi )}
        (\mathrm{Im}\, \Phi_{\kappa}(\xi,\bm{z}))
        + N_{\xi } \overline{w(\xi)}
         (\mathrm{Im}\, \Phi_{\kappa}(\xi,\bm{z})) ]  \\
   & \qquad - [ \partial_{\bm{n}(\xi )}\overline{w(\xi)}\, M_{\xi }
     ( \mathrm{Im}\, \Phi_{\kappa}(\xi,\bm{z})  )
     + \overline{w(\xi)}\,  N_{\xi }
     ( \mathrm{Im}\, \Phi_{\kappa}(\xi,\bm{z}) )
     ]  \Big\} \;\mathrm{d}s(\xi). 
    \end{align*}
    Then combinations of \eqref{eq:vsd}, \eqref{fun-9} and \eqref{fun-10} leads to  \eqref{eq:imaging-2}  with $w_{r,3}$ replaced by $0$.

    \emph{Proof of} the case of boundary condition  \uppercase\expandafter{\romannumeral 2}.

    By the arguments in the proof of the case of the boundary condition  \uppercase\expandafter{\romannumeral 1,} we obtain
        \begin{align*}
    I_j(\bm{ z }) 
    \notag
    = &\mathrm{Im}\, \int_{ \Gamma } \Big\{ [M_{\xi}\overline{\psi(\xi ,\bm{ z })}\, \partial_{\bm{n}(\xi )}
        (\mathrm{Im}\, \Phi_{\kappa}(\xi,\bm{z})  ) 
        +N_{\xi } \overline{\psi(\xi ,\bm{ z })} \,
        (\mathrm{Im}\, \Phi_{\kappa}(\xi,\bm{z}) ) ]  \\
   & \qquad - [ \partial_{\bm{n}(\xi )}\overline{\psi(\xi ,\bm{ z })}\,  M_{\xi }
     ( \mathrm{Im}\, \Phi_{\kappa}(\xi,\bm{z})  )
     + \overline{\psi(\xi ,\bm{ z })}\,  N_{\xi }
     ( \mathrm{Im}\, \Phi_{\kappa}(\xi,\bm{z})  )
     ]  \Big\} \;\mathrm{d}s(\xi)
     +w_j \quad j=1,2,\cdots 10.
        \end{align*}
        Obviously, the function $\psi(\xi ,\bm{ z })$ in the case satisfies the simply supported boundary condition, i.e.,
\begin{equation*}
  \psi(\xi ,\bm{ z }) = -  \mathrm{Im}\, \Phi_{\kappa}(\xi,\bm{z}) \quad \mathrm{ and } \quad
M_{\bm{n}(\xi )} \psi(\xi ,\bm{ z }) = - M_{\xi} \mathrm{Im}\, \Phi_{\kappa}(\xi,\bm{z}) \quad
  \xi \in \Gamma,
\end{equation*}
which leads to
    \begin{align*}
      &\mathrm{Im}\, \int_{ \Gamma }[ M_{\xi } \overline{\psi(\xi ,\bm{ z })} \,
       \partial_{\bm{n}(\xi )} ( \mathrm{Im}\, \Phi_{\kappa}(\xi,\bm{z})  )
        - \overline{\psi(\xi ,\bm{ z })}\,
        N_{\xi } ( \mathrm{Im}\, \Phi_{\kappa}(\xi,\bm{z})  )
     ] \;\mathrm{d}s(\xi)\\
      =&  - \mathrm{Im}\, \int_{ \Gamma }[ M_{\xi } \mathrm{Im}\, \Phi_{\kappa}(\xi,\bm{z}) \,
          \partial_{\bm{n}(\xi )}( \mathrm{Im}\, \Phi_{\kappa}(\xi,\bm{z})  )
        - \mathrm{Im}\, \Phi_{\kappa}(\xi,\bm{z}) \,
       N_{\xi }( \mathrm{Im}\, \Phi_{\kappa}(\xi,\bm{z})  )
        ] \;\mathrm{d}s(\xi)=0. 
    \end{align*}
    Hence we arrive at
            \begin{align*}
    I_j(\bm{ z }) 
    \notag
    = &\mathrm{Im}\, \int_{ \Gamma } \Big\{
        N_{\xi } \overline{\psi(\xi ,\bm{ z })} \,
        (\mathrm{Im}\, \Phi_{\kappa}(\xi,\bm{z}) ) ]  
   -  \partial_{\bm{n}(\xi )}\overline{\psi(\xi ,\bm{ z })}\,  M_{\xi }
     ( \mathrm{Im}\, \Phi_{\kappa}(\xi,\bm{z})  )
     \Big\} \;\mathrm{d}s(\xi)
        +w_j\\
       = &-\mathrm{Im}\, \int_{ \Gamma } \Big\{
       \psi(\xi ,\bm{ z })  \,
         N_{\xi } \overline{\psi(\xi ,\bm{ z })}   
          - \partial_{\bm{n}(\xi )}\overline{\psi(\xi ,\bm{ z })} \,
            M_{\xi } ( \psi(\xi ,\bm{ z }) )
     \Big\} \;\mathrm{d}s(\xi)
           +w_j \\
      = & \frac{\kappa^{2}}{4\pi} \int_{ \mathbb{S}  } | \psi ^{\infty}(\cdot ,\bm{ z }) | ^{2} \;\mathrm{d}s(\cdot )
+w_j \quad \text{for} \quad j=1,2,\cdots 10,
            \end{align*}
            where we use Lemma \ref{lem:sign}. This allows us to conclude the proof.
      \end{proof}

We end this section by the following reconstruction algorithm:
\vspace{0.5\baselineskip}

\noindent\textbf{Reconstruction algorithm for extended obstacles with phased data}
\
   \begin{itemize}
    \item  Choose a sampling domain $\Omega$ and generate sampling points.
     \item For each sampling point $\bm{ z }\in \Omega$, calculate the imaging function $I_j(\bm{ z })$ for $j=1,2,\cdots,10$.
     \item Plot the contours for the imaging functions
       $$\hat{I}_j(\bm{ z })=\dfrac{I_j(\bm{ z })}{\max\limits_{\bm{ z }\in \Omega}I_j(\bm{ z })}\quad j=1,2,\cdots,10.$$
     \end{itemize}
     
\section{The direct imaging method with phaseless data}
\label{sec:phaseless}
In this section, we develop the direct imaging method with phaseless data. To this end, we introduce  the following imaging function for \textbf{IP-11}.
\begin{equation*}
I_{11}(\bm{ z }) =  -2 \kappa^4 ~\mathrm{Im}\, \int_{ \Gamma_{r} } \int_{ \Gamma_{s} } \Phi_{\kappa}(\bm{ z },\bm{ x }_r)  \Phi_{\kappa}(\bm{ z },\bm{ x }_s)
\dfrac{| u(\bm{ x }_r,\bm{ x }_s) |^2 - |\Phi_{\kappa}(\bm{ x }_r,\bm{ x }_s)  |^2 }{\Phi_{\kappa}(\bm{ x }_r,\bm{ x }_s)}
\;\mathrm{d}s(\bm{ x }_s)\mathrm{d}s(\bm{ x }_r) \quad
\bm{ z }\in \Omega.
\end{equation*}
Without loss of generality, assume $R_r=\tau R_s$ with $\tau \geq 1$. First, we start with some lemmas.
\begin{lemma}
  There exists a constant $C>0$ such that
  \begin{align*}
    | u^{sc}(\bm{ x }_r, \bm{ x }_s) |  \leq  C R^{-1/2}_r R^{-1/2}_s \quad \mathrm{for} \quad \bm{ x }_r \in \Gamma_r, \bm{ x }_s \in \Gamma_s,
  \end{align*}
  where the constant $C$ depends on the wave number $\kappa$ and the domain $D$.
\end{lemma}
\begin{proof}
  The proof is a direct consequence of the representation \eqref{eq:Green-rep} in Lemma \ref{lem:rep}, Lemma \ref{lem:well} and the following estimate of Hankel functions (see \cite[(3.8)]{chen_phaseless_2017})
  \begin{equation*}
    | H^{(1)}_0 (t)| \leq \left( \dfrac{2}{\pi t} \right)^{1/2}, \quad
    | H^{(1)}_1 (t)| \leq \left( \dfrac{2}{\pi t} \right)^{1/2} + \dfrac{2}{\pi t}
    \quad \text{for} \quad t>0.
    \end{equation*}  
\end{proof}

\begin{lemma}[see \cite{dong_uniqueness_2024}]
  It holds that
  $$
  u^{\infty}(- \bm{ d },\bm{ x })= \dfrac{e^{\mathrm{ i } \pi/4}}{\sqrt{8\pi \kappa}} u^{pr}(\bm{ x }, \bm{ d }) \quad
  \bm{ x } \in D^{c}, \bm{ d } \in \mathbb{S} ,
  $$
  where $u^{pr}= -(\Delta u^{sc}-\kappa^2 u^{sc})/(2 \kappa^2)$.
\end{lemma}

\begin{lemma}
  The following relation holds
\begin{equation*}
  | u^{\infty}(\hat{\bm x},-\bm{ d } ) |  + | \partial_{\theta_{\hat{\bm x}} } u^{\infty}(\hat{\bm x},-\bm{ d } ) |
  \leq C \quad \hat{\bm x}, \bm{ d } \in \mathbb{S}, 
\end{equation*}
where $C$ is a constant which depends on the wave number $\kappa$ and the domain $D$.
\end{lemma}
\begin{proof}
  The proof can be easily derived by Lemma \ref{lem:well}.
\end{proof}

The above lemmas shows that the asymptotic behaviors of the scattered field $u^{sc}(\bm{ x }_r,\bm{ x }_s) $ are very similar to the those of the acoustic scattered field with the same point sources as the incident field. Besides, the imaging function $I_{11}$ is the same as that in the case of phaseless inverse acoustic scattering problem in the form. 
Thus we can use the same method in  \cite{chen_phaseless_2017} to get the resolution analysis of the imaging function $I_{11}$. Here we omit the proof for brevity.
\begin{theorem}
  For any $\bm{ z }\in \Omega$, it holds that
  $$I_{11}(\bm{ z }) = I_1(\bm{ z }) + w_{phaseless}(\bm{ z }), $$
  where $\|w_{phaseless}\|_{L^{\infty}(\Omega)}\leq CR_s^{-1/2}$.
\end{theorem}

We end this section by the following reconstruction algorithm:

\vspace{0.5\baselineskip}

\noindent\textbf{Reconstruction algorithm for extended obstacles with phaseless data}
\
\begin{itemize}
    \item  Choose a sampling domain $\Omega$ and generate sampling points.
     \item For each sampling point $\bm{ z }\in \Omega$, calculate the imaging function $I_{11}(\bm{ z })$.
     \item Plot the contours for the imaging function
       $$\hat{I}_{ 11}(\bm{ z })=\dfrac{I_{11}(\bm{ z })}{\max\limits_{\bm{ z }\in \Omega}I_{11}(\bm{ z })}.$$      
\end{itemize}

\section{Numerical experiments}
\label{sec:numerical}
In this section, several numerical examples will be presented to show the feasibility of our algorithm. To get the synthetic data, we adopt the boundary integral equation method in \cite{dong_novel_2024}. For the noisy data, the following formula is used
\begin{equation}
\notag
u^{sc}_{\delta}(\bm{ x }_r, \bm{ x }_s) = u^{sc}(\bm{ x }_r, \bm{ x }_s)
+ \delta \frac{| u^{sc}(\bm{ x }_r, \bm{ x }_s)  | }{|\xi_1+\mathrm{ i } \xi_2|} (\xi_1+\mathrm{ i } \xi_2),
\end{equation}
where $\delta$ is the noise level, and $\xi_1$ and $\xi_2$ are random numbers subject to the standard normal distribution.

{\bf Example 1.}
We consider the case where the obstacle is a unit circle centered at the origin. The sampling domain $\Omega$ is $[-6,6]\times [-6,6]$ and the wave number is $2\pi$. We use algorithm to reconstruct the obstacle; see figure \ref{fig1:circle}. It can be seen from figure \ref{fig1:circle}  that  our algorithm based on these imaging functions can locate the obstacle effectively. 
\begin{figure}[H]
     \centering
     \begin{subfigure}[b]{0.23\textwidth}
         \centering
         \includegraphics[width=\textwidth]{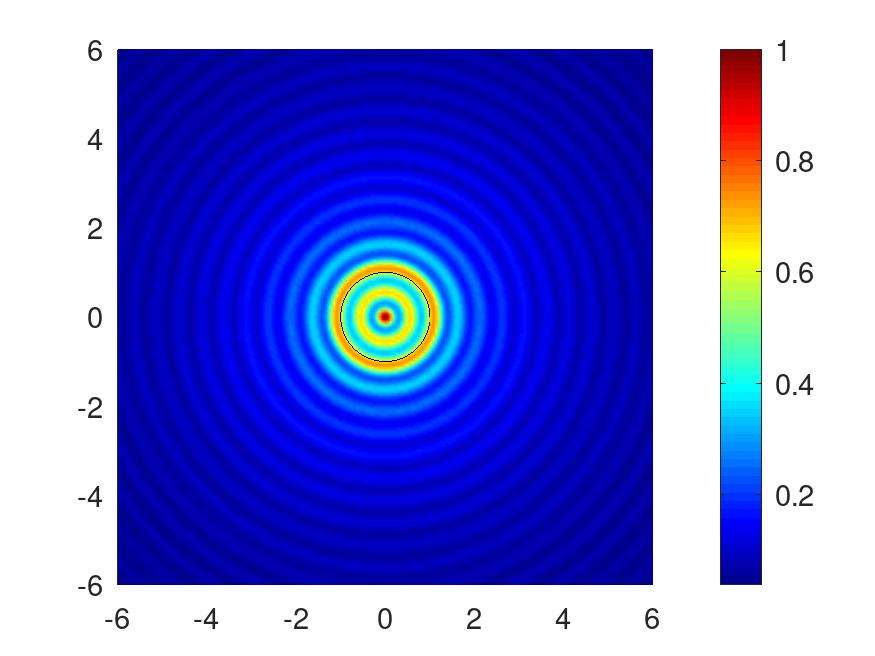}
         \caption{$I_1$}
         \label{fig1_1_1}
     \end{subfigure}
     \hfill
     \begin{subfigure}[b]{0.23\textwidth}
         \centering
         \includegraphics[width=\textwidth]{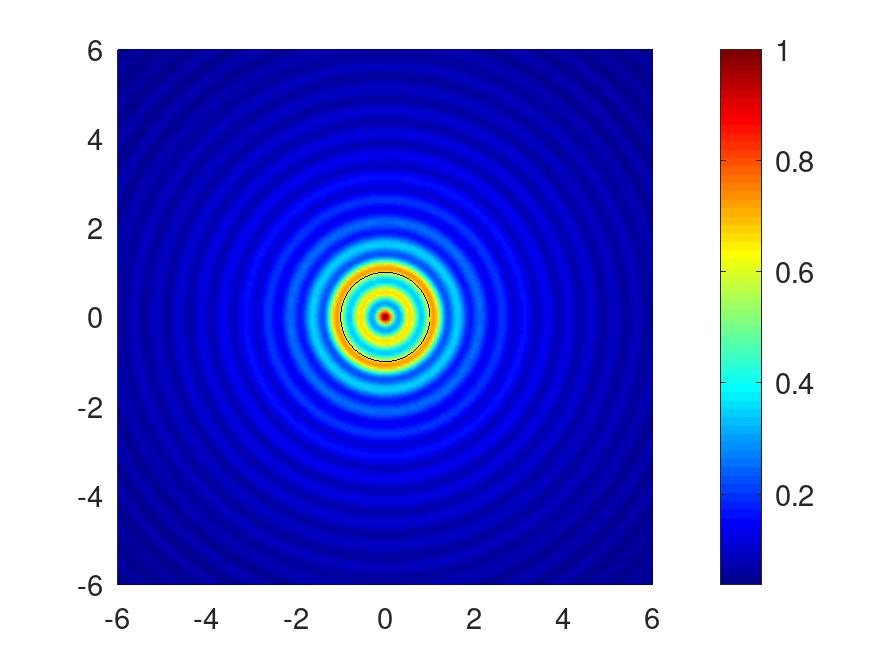}
         \caption{$I_2$}
         \label{fig1_2_1}
     \end{subfigure}
     \hfill
     \begin{subfigure}[b]{0.23\textwidth}
         \centering
         \includegraphics[width=\textwidth]{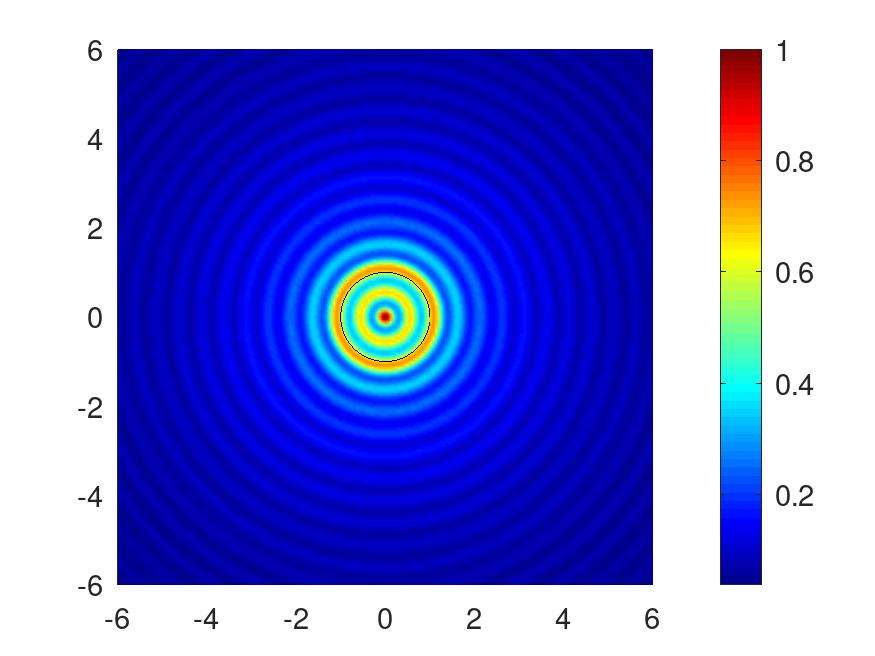}
         \caption{$I_3$}
         \label{fig1_3_1}
     \end{subfigure}
 \hfill
          \begin{subfigure}[b]{0.23\textwidth}
         \centering
         \includegraphics[width=\textwidth]{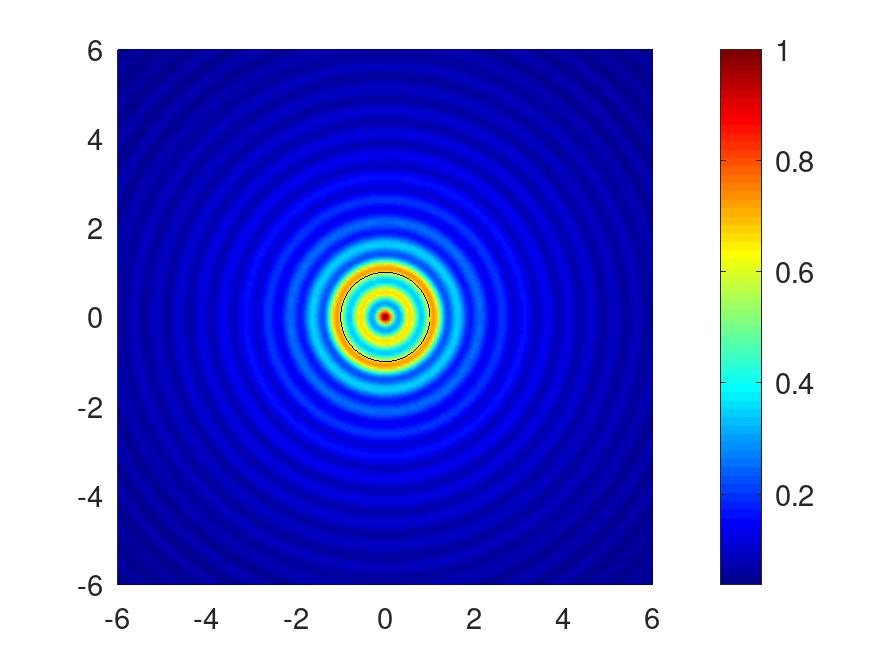}
         \caption{$I_4$}
         \label{fig1_4_1}
       \end{subfigure}

     \begin{subfigure}[b]{0.23\textwidth}
         \centering
         \includegraphics[width=\textwidth]{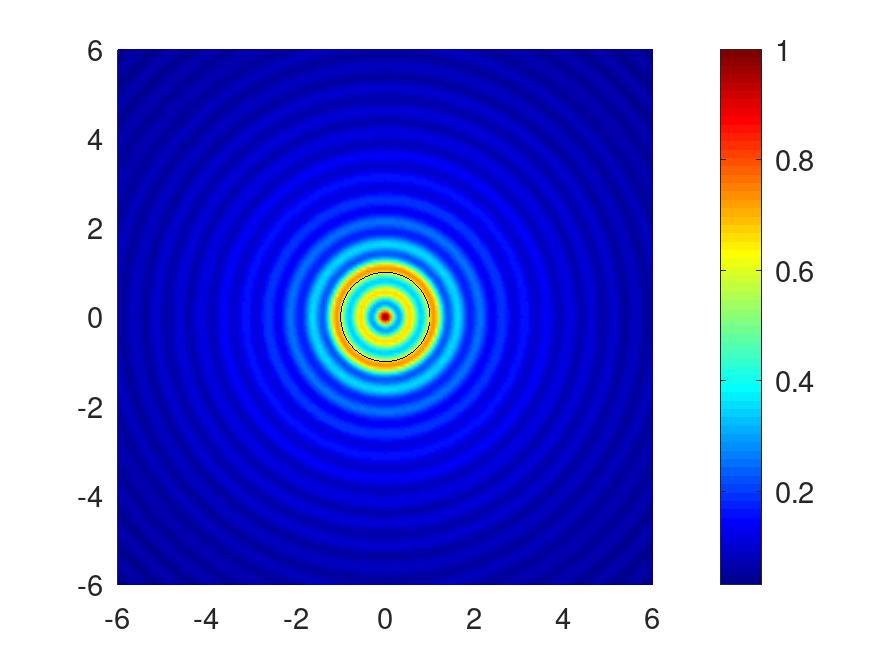}
         \caption{$I_5$}
         \label{fig1_5_1}
     \end{subfigure}
     \hfill
     \begin{subfigure}[b]{0.23\textwidth}
         \centering
         \includegraphics[width=\textwidth]{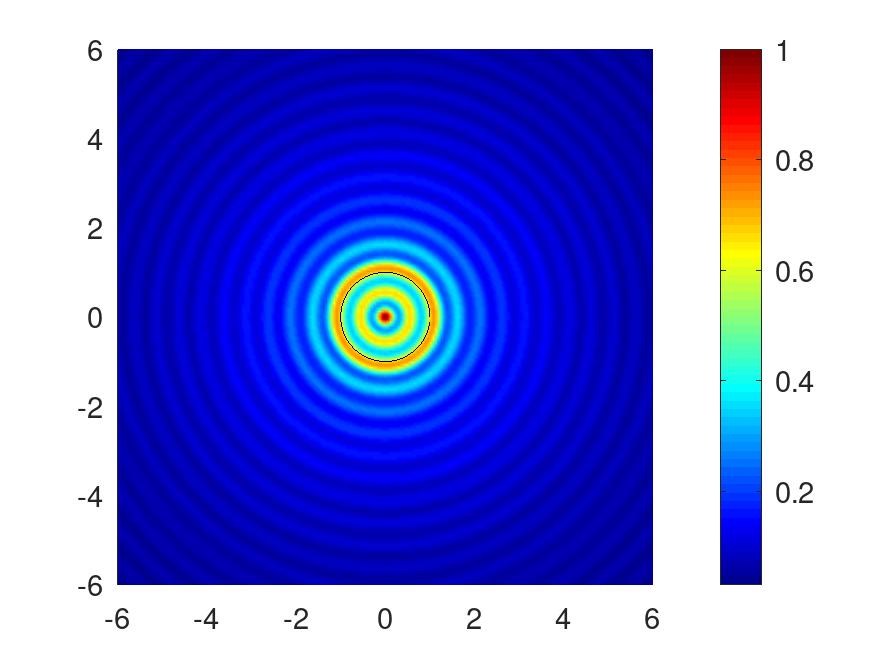}
         \caption{$I_6$}
         \label{fig1_6_1}
       \end{subfigure}
        \hfill
            \begin{subfigure}[b]{0.23\textwidth}
         \centering
         \includegraphics[width=\textwidth]{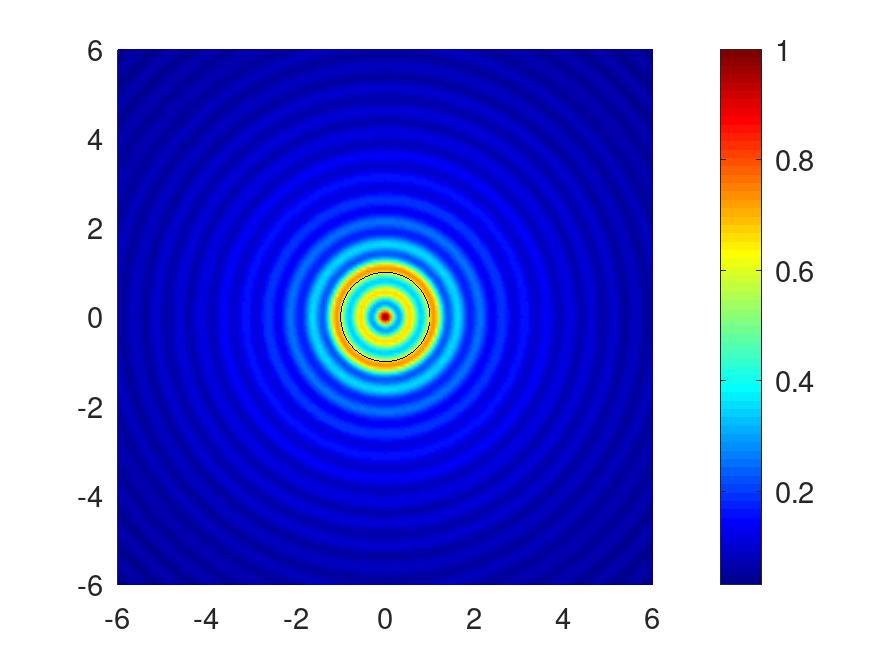}
         \caption{$I_7$}
         \label{fig1_7_1}
     \end{subfigure}
     \hfill
     \begin{subfigure}[b]{0.23\textwidth}
         \centering
         \includegraphics[width=\textwidth]{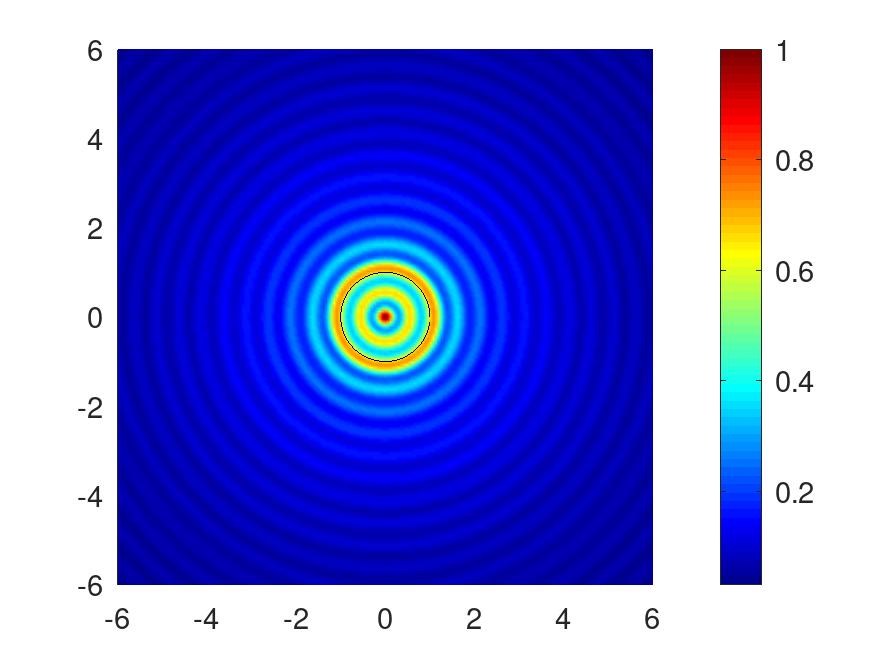}
         \caption{$I_8$}
         \label{fig1_8_1}
       \end{subfigure}
       
     \begin{subfigure}[b]{0.23\textwidth}
         \centering
         \includegraphics[width=\textwidth]{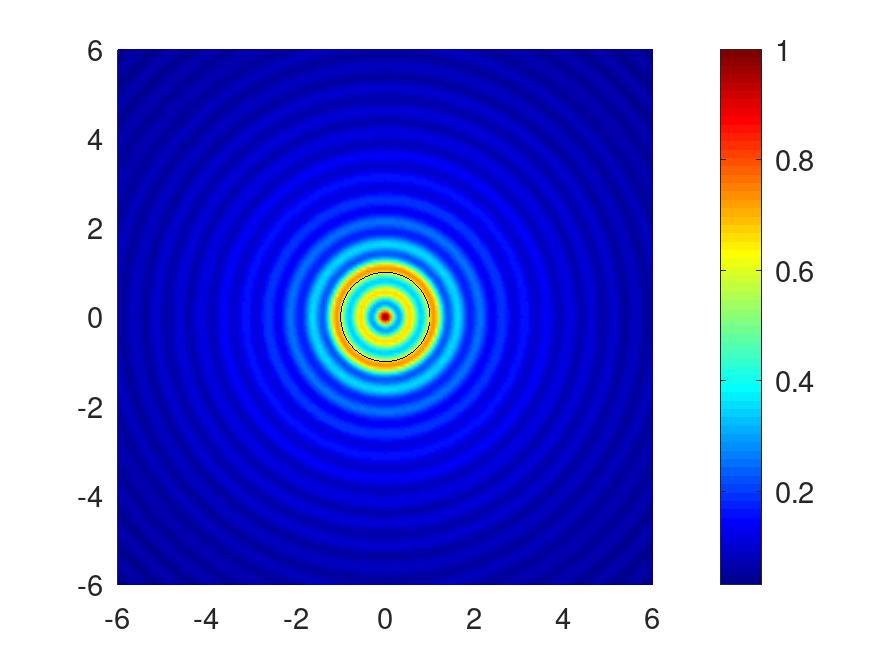}
         \caption{$I_9$}
         \label{fig1_9_1}
       \end{subfigure}
     \hfill
            \begin{subfigure}[b]{0.23\textwidth}
         \centering
         \includegraphics[width=\textwidth]{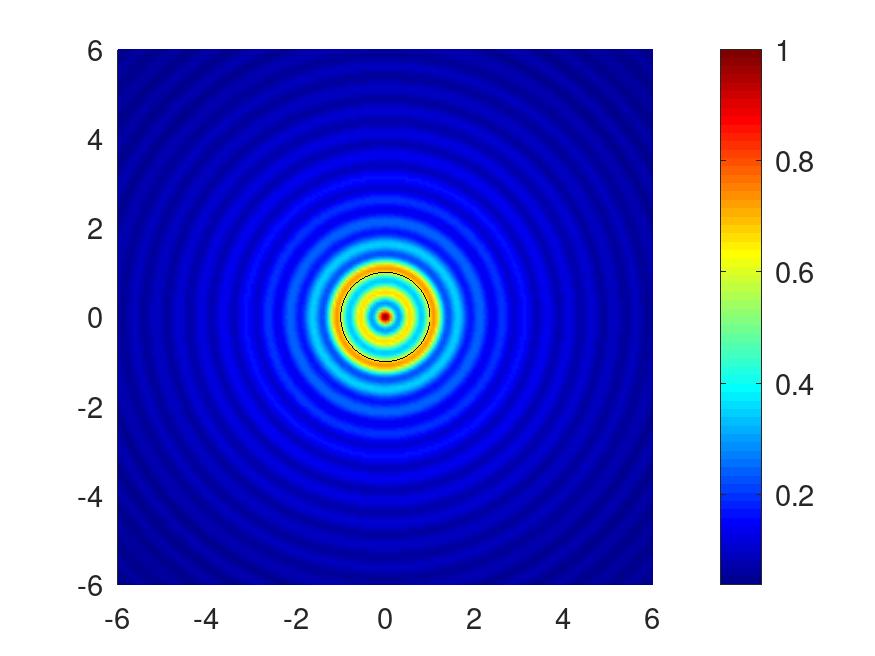}
         \caption{$I_{10}$}
         \label{fig1_10_1}
     \end{subfigure}
     \hfill
     \begin{subfigure}[b]{0.23\textwidth}
         \centering
         \includegraphics[width=\textwidth]{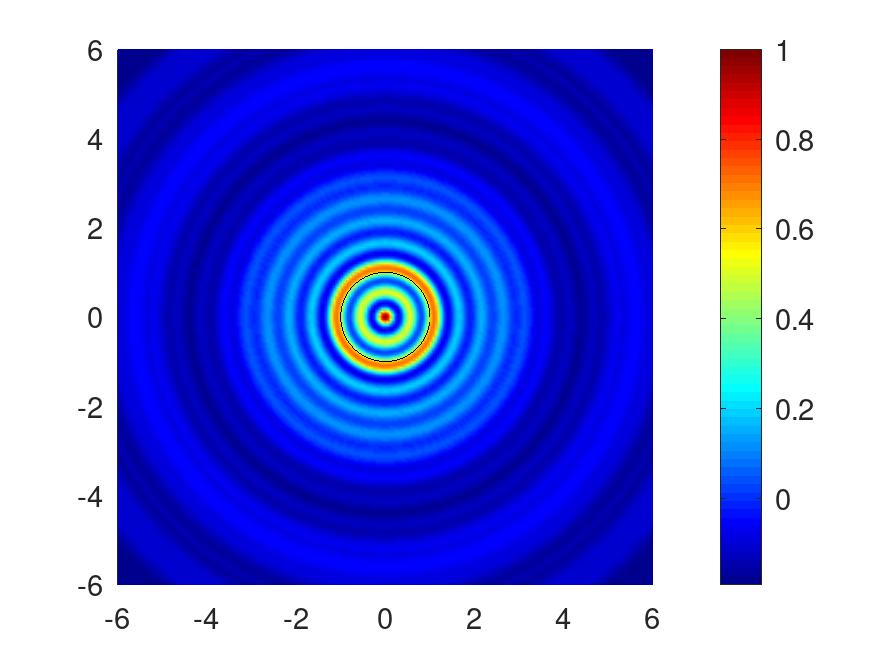}
         \caption{$I_{11}$}
         \label{fig1_11_1}
          \end{subfigure}
         \hfill
         \caption{Example 1: The reconstruction results with no noisy data}
         \label{fig1:circle}
 \end{figure}

{\bf Example 2.}
  We consider the influence of noisy data. Assume that the obstacle $D$ is a kite whose parameterized form is
  \begin{equation*}
    (0.65 \cos(2t)+\cos t-0.65, 1.5\sin t) \quad t \in [0,2\pi].
  \end{equation*}
  The sampling domain is $[-6,6]\times [-6,6]$ and the wave number $\kappa$ is $2\pi$. The results with noise level $0\%$, $5\%$ and $10\%$  are shown in figure \ref{fig2:kite} , which demonstrate that our algorithm is robust to the noise.
   \begin{figure}[H]
     \centering
     \begin{subfigure}[b]{0.3\textwidth}
         \centering
         \includegraphics[width=\textwidth]{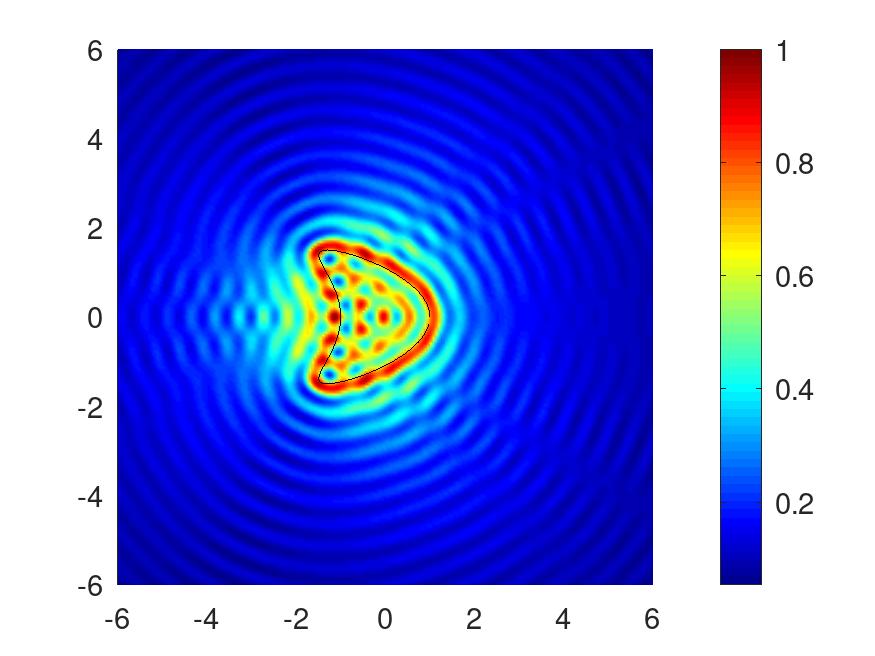}
         \caption{No noise}
         \label{fig2_1_1}
     \end{subfigure}
     \hfill
     \begin{subfigure}[b]{0.3\textwidth}
         \centering
         \includegraphics[width=\textwidth]{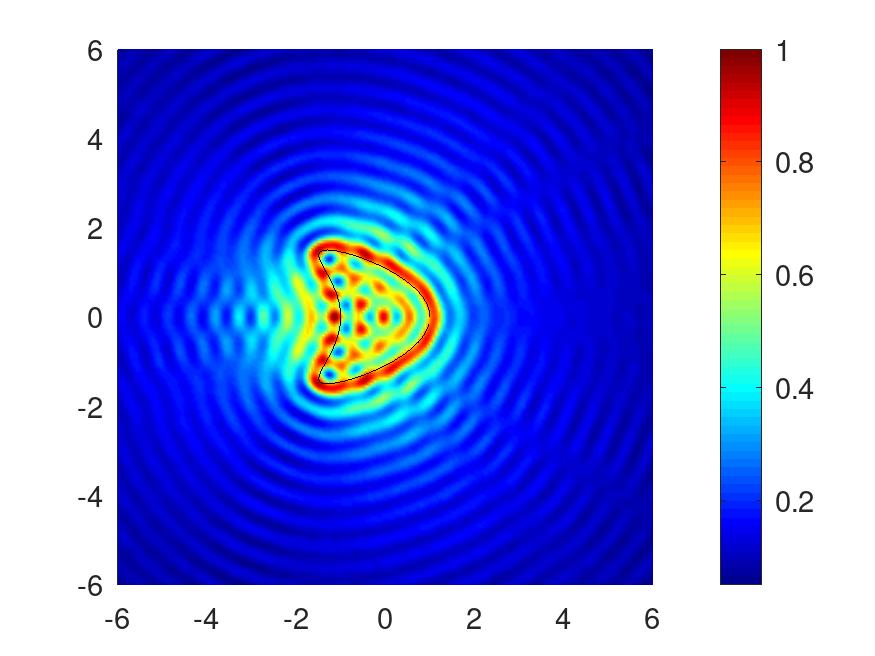}
         \caption{$5\%$ noise}
         \label{fig2_1_2}
     \end{subfigure}
     \hfill
     \begin{subfigure}[b]{0.3\textwidth}
         \centering
         \includegraphics[width=\textwidth]{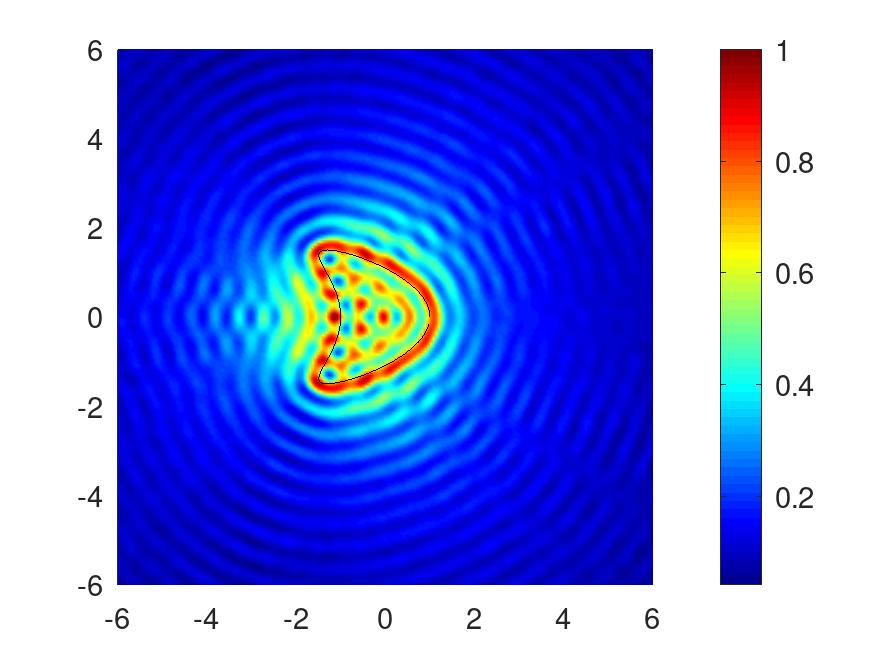}
         \caption{$10\%$ noise}
         \label{fig2_1_3}
     \end{subfigure}
      \end{figure}
     
     \begin{figure}[H]
        \centering
       \ContinuedFloat
          \begin{subfigure}[b]{0.3\textwidth}
         \centering
         \includegraphics[width=\textwidth]{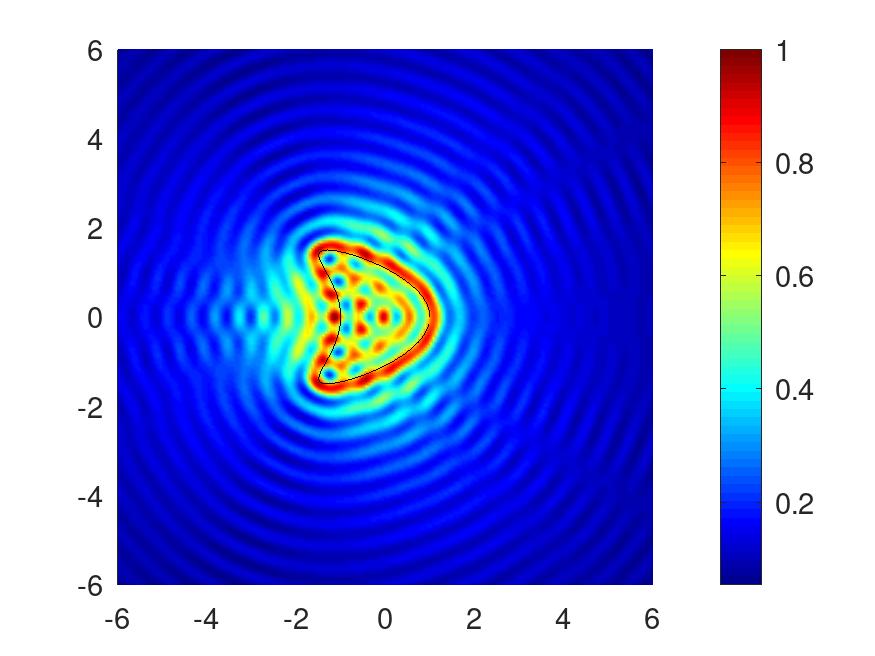}
         \caption{No noise}
         \label{fig2_2_1}
     \end{subfigure}
     \hfill
     \begin{subfigure}[b]{0.3\textwidth}
         \centering
         \includegraphics[width=\textwidth]{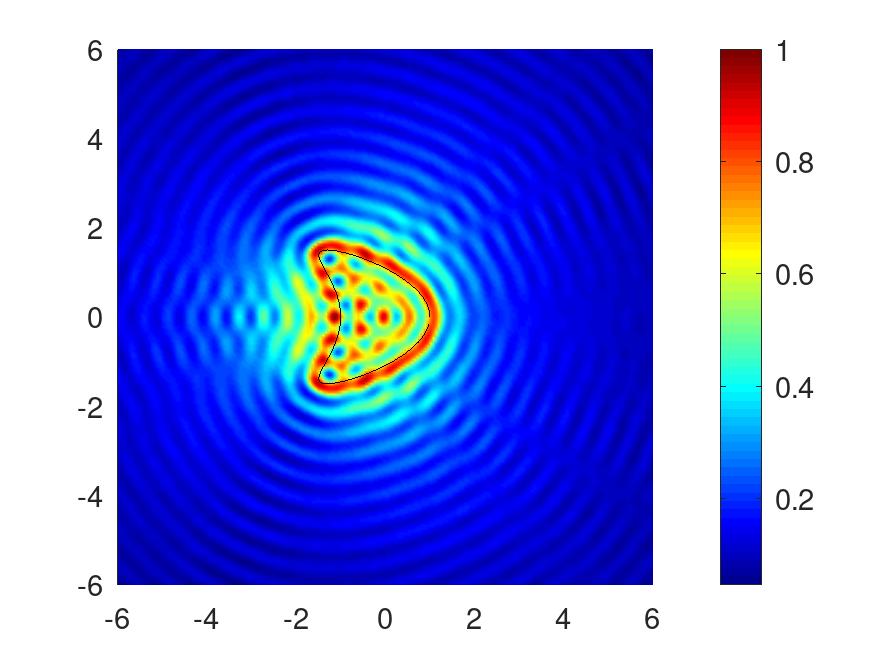}
         \caption{$5\%$ noise}
         \label{fig2_2_2}
     \end{subfigure}
     \hfill
     \begin{subfigure}[b]{0.3\textwidth}
         \centering
         \includegraphics[width=\textwidth]{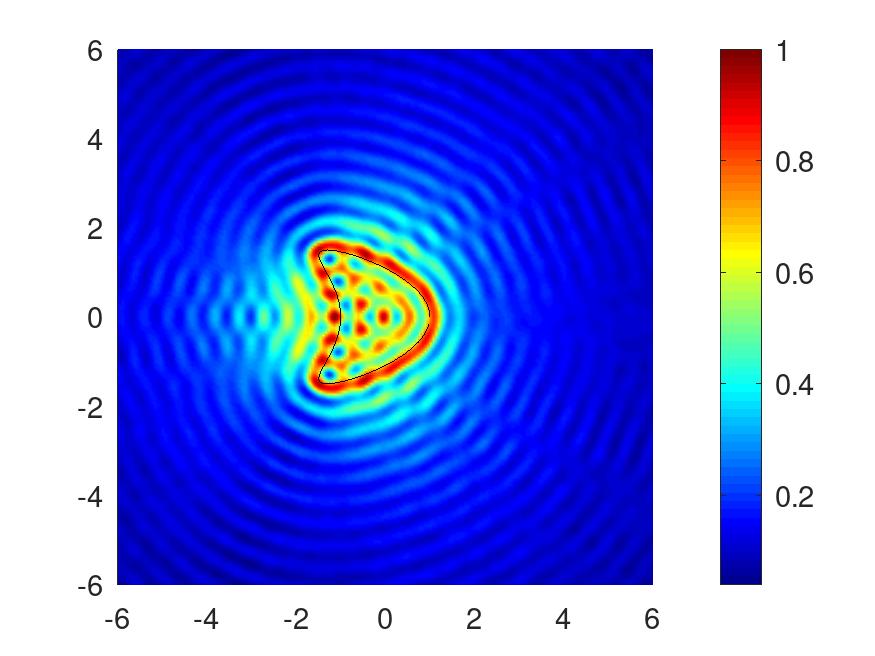}
         \caption{$10\%$ noise}
         \label{fig2_2_3}
       \end{subfigure}

            \begin{subfigure}[b]{0.3\textwidth}
         \centering
         \includegraphics[width=\textwidth]{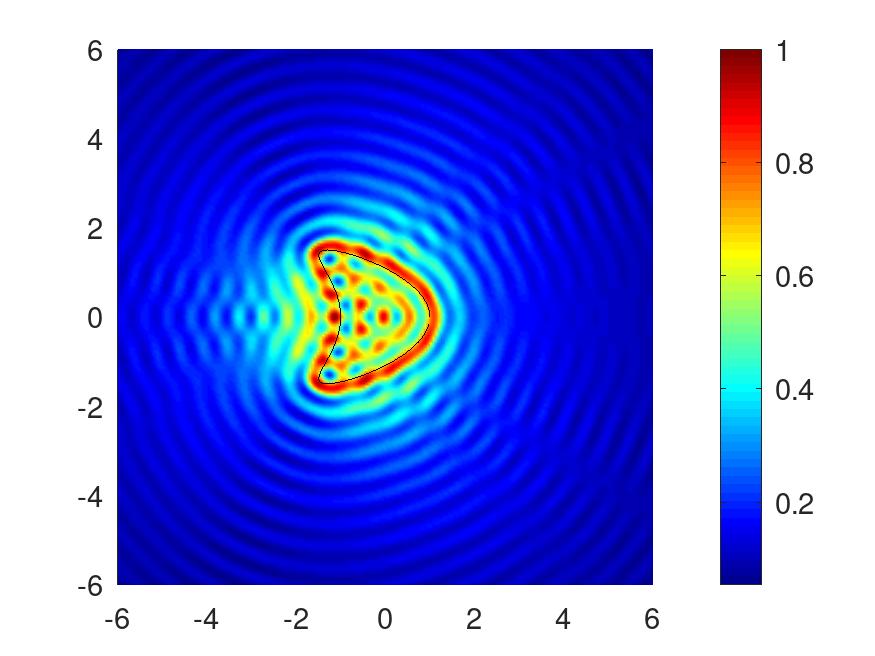}
         \caption{No noise}
         \label{fig2_3_1}
     \end{subfigure}
     \hfill
     \begin{subfigure}[b]{0.3\textwidth}
         \centering
         \includegraphics[width=\textwidth]{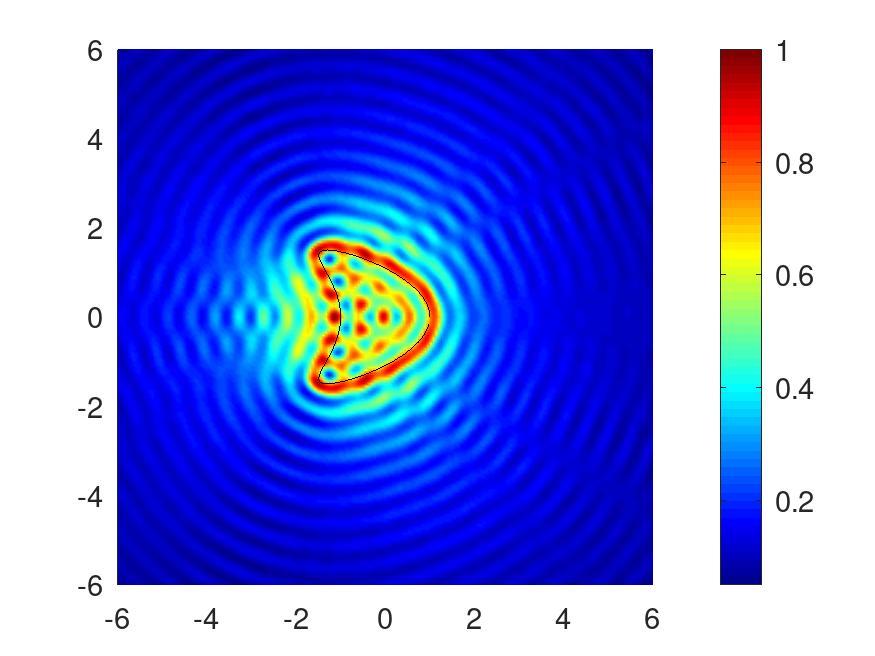}
         \caption{$5\%$ noise}
         \label{fig2_3_2}
     \end{subfigure}
     \hfill
     \begin{subfigure}[b]{0.3\textwidth}
         \centering
         \includegraphics[width=\textwidth]{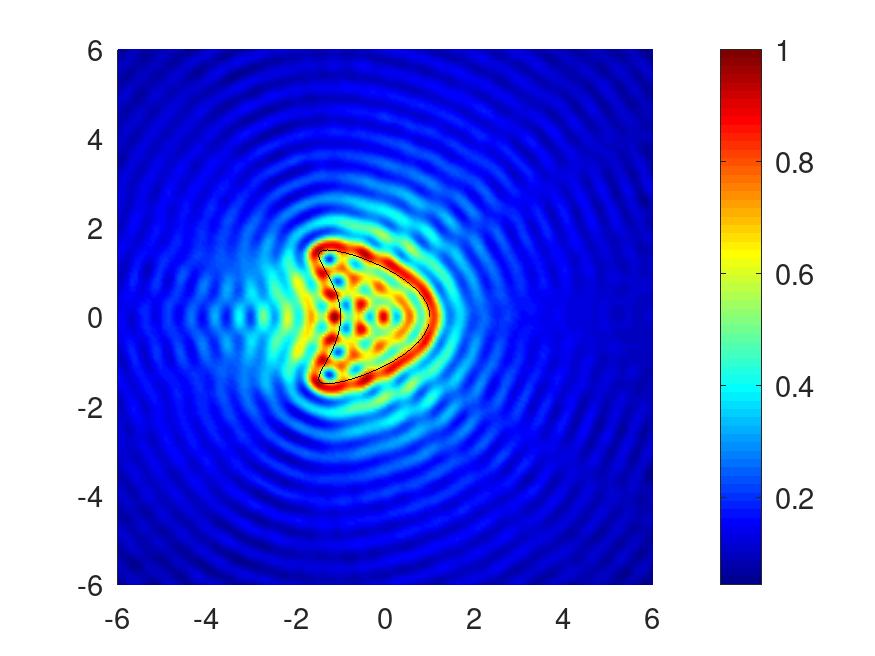}
         \caption{$10\%$ noise}
         \label{fig2_3_3}
       \end{subfigure}

            \begin{subfigure}[b]{0.3\textwidth}
         \centering
         \includegraphics[width=\textwidth]{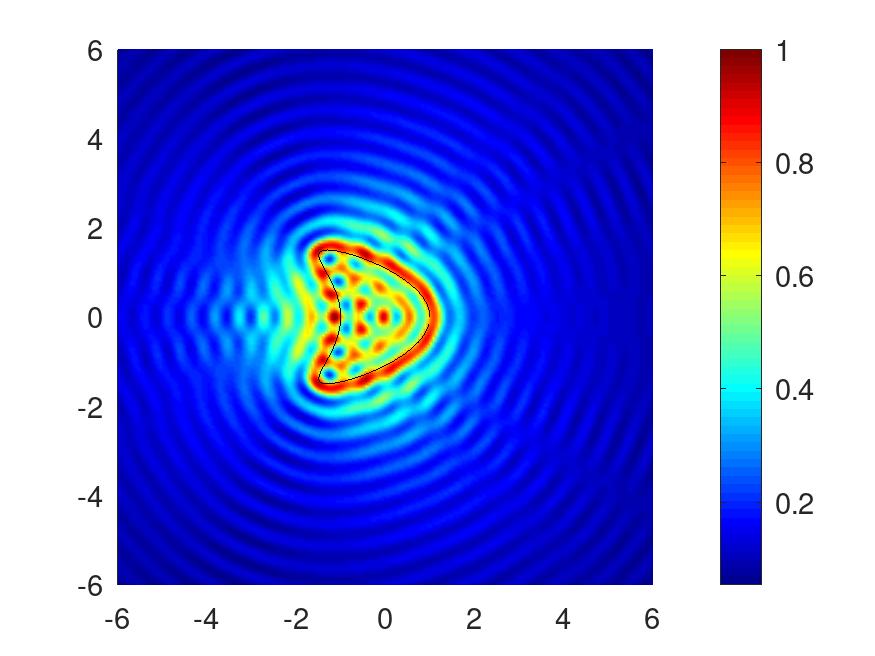}
         \caption{No noise}
         \label{fig2_4_1}
     \end{subfigure}
     \hfill
     \begin{subfigure}[b]{0.3\textwidth}
         \centering
         \includegraphics[width=\textwidth]{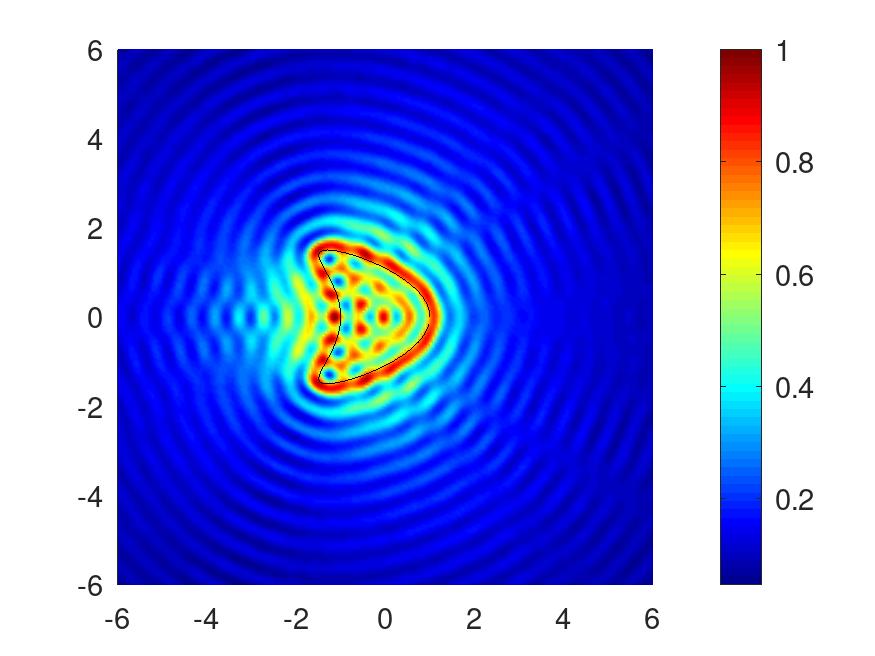}
         \caption{$5\%$ noise}
         \label{fig2_4_2}
     \end{subfigure}
     \hfill
     \begin{subfigure}[b]{0.3\textwidth}
         \centering
         \includegraphics[width=\textwidth]{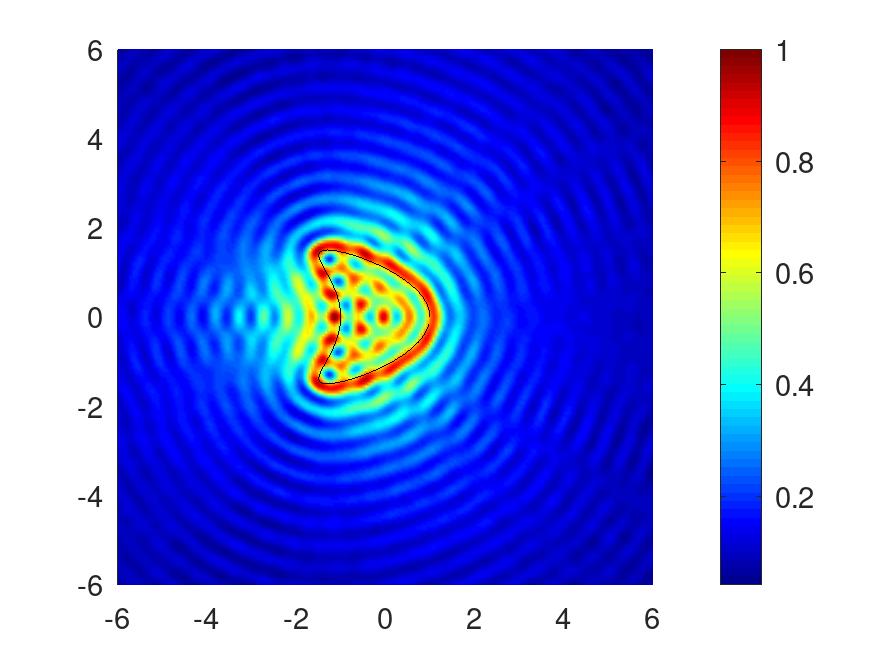}
         \caption{$10\%$ noise}
         \label{fig2_4_3}
       \end{subfigure}

            \begin{subfigure}[b]{0.3\textwidth}
         \centering
         \includegraphics[width=\textwidth]{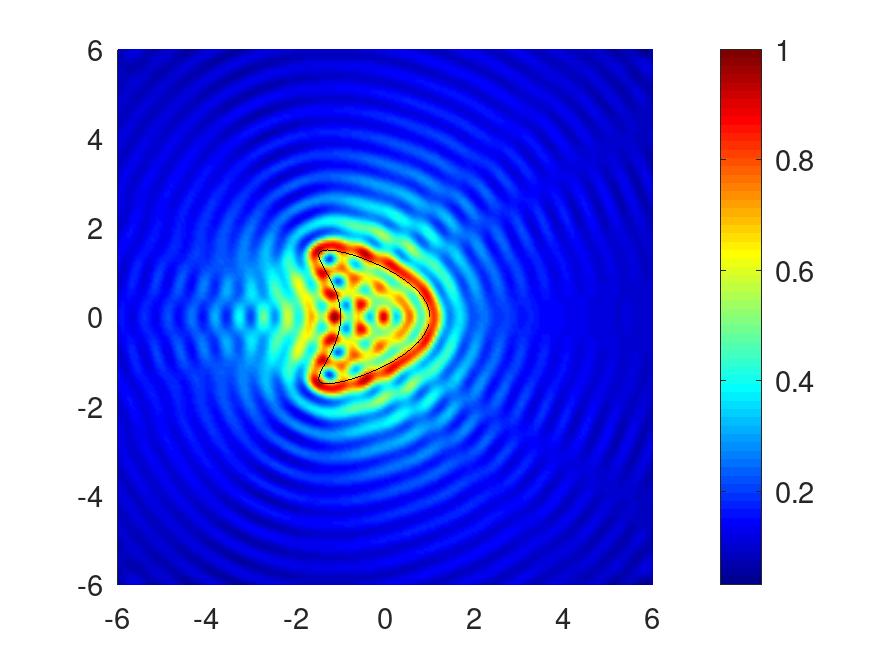}
         \caption{No noise}
         \label{fig2_5_1}
     \end{subfigure}
     \hfill
     \begin{subfigure}[b]{0.3\textwidth}
         \centering
         \includegraphics[width=\textwidth]{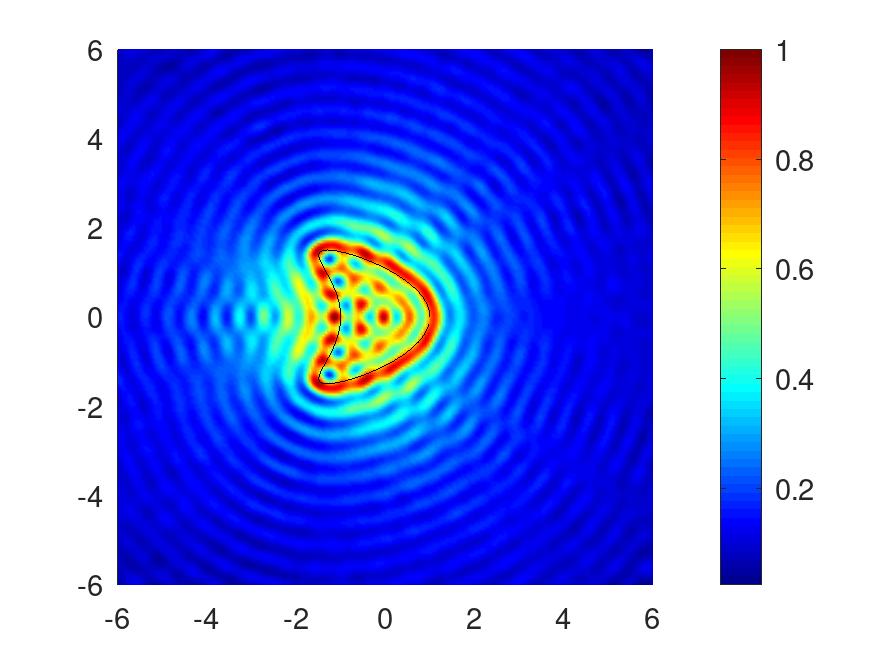}
         \caption{$5\%$ noise}
         \label{fig2_5_2}
     \end{subfigure}
     \hfill
     \begin{subfigure}[b]{0.3\textwidth}
         \centering
         \includegraphics[width=\textwidth]{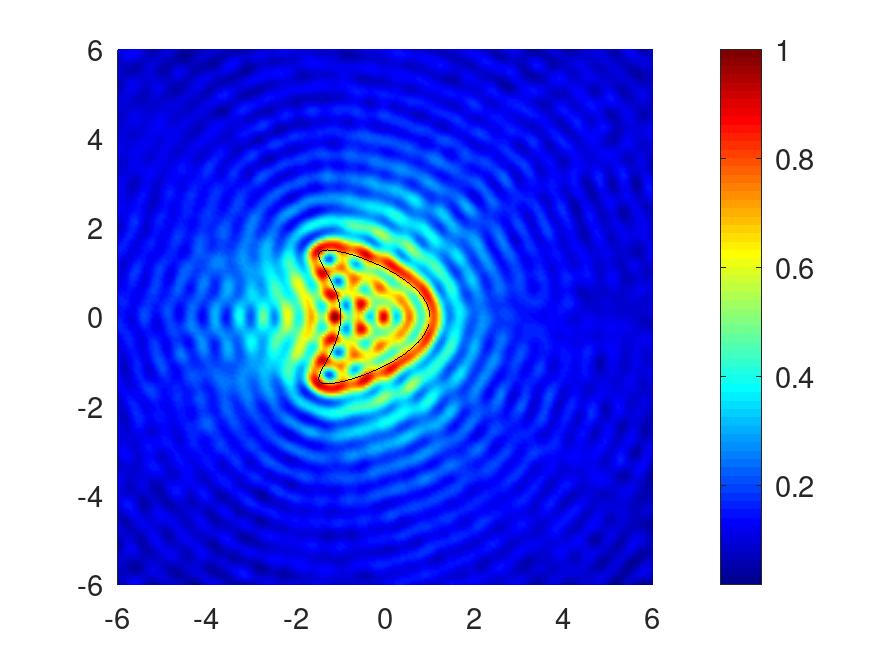}
         \caption{$10\%$ noise}
         \label{fig2_5_3}
       \end{subfigure}

            \begin{subfigure}[b]{0.3\textwidth}
         \centering
         \includegraphics[width=\textwidth]{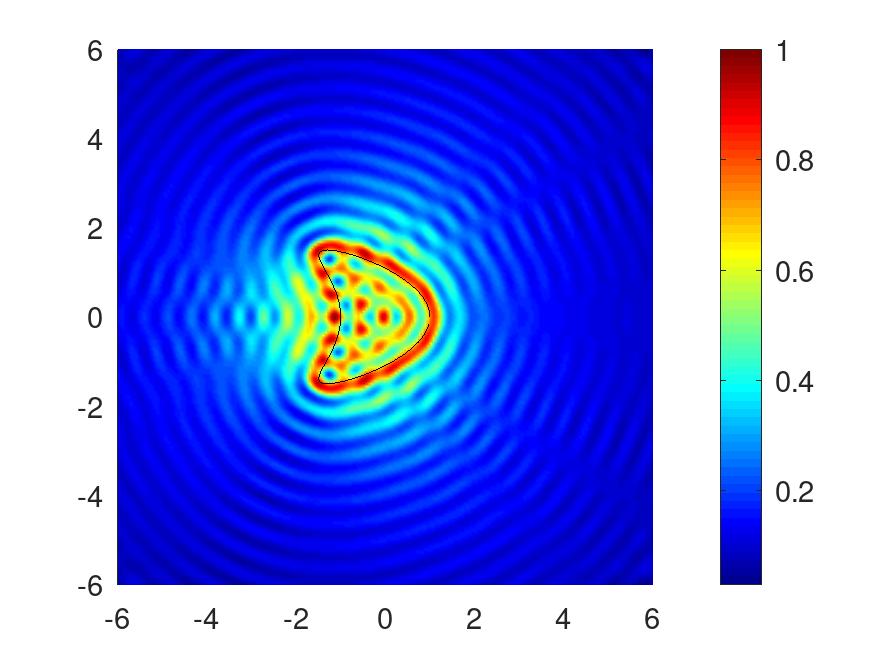}
         \caption{No noise}
         \label{fig2_6_1}
     \end{subfigure}
     \hfill
     \begin{subfigure}[b]{0.3\textwidth}
         \centering
         \includegraphics[width=\textwidth]{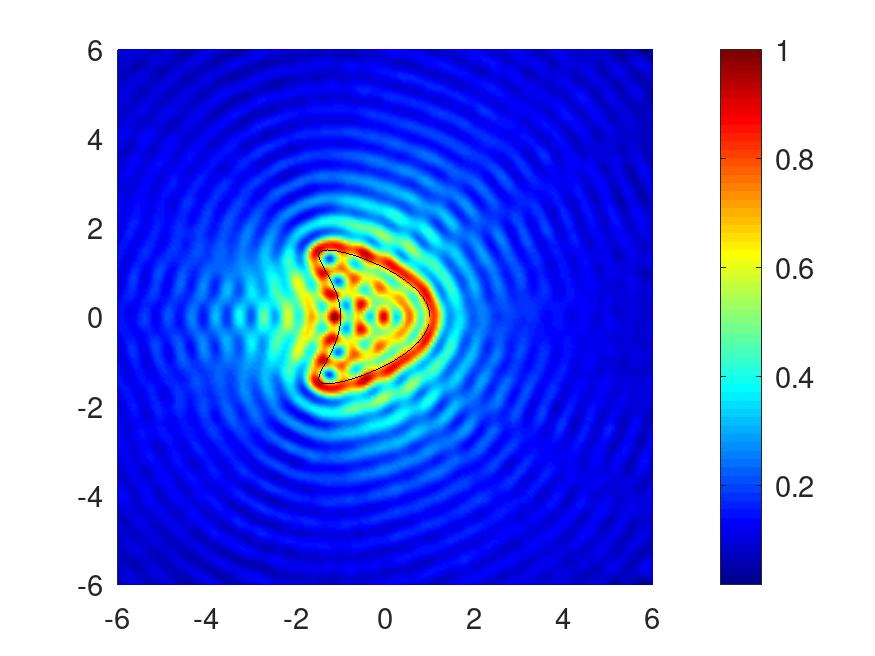}
         \caption{$5\%$ noise}
         \label{fig2_6_2}
     \end{subfigure}
     \hfill
     \begin{subfigure}[b]{0.3\textwidth}
         \centering
         \includegraphics[width=\textwidth]{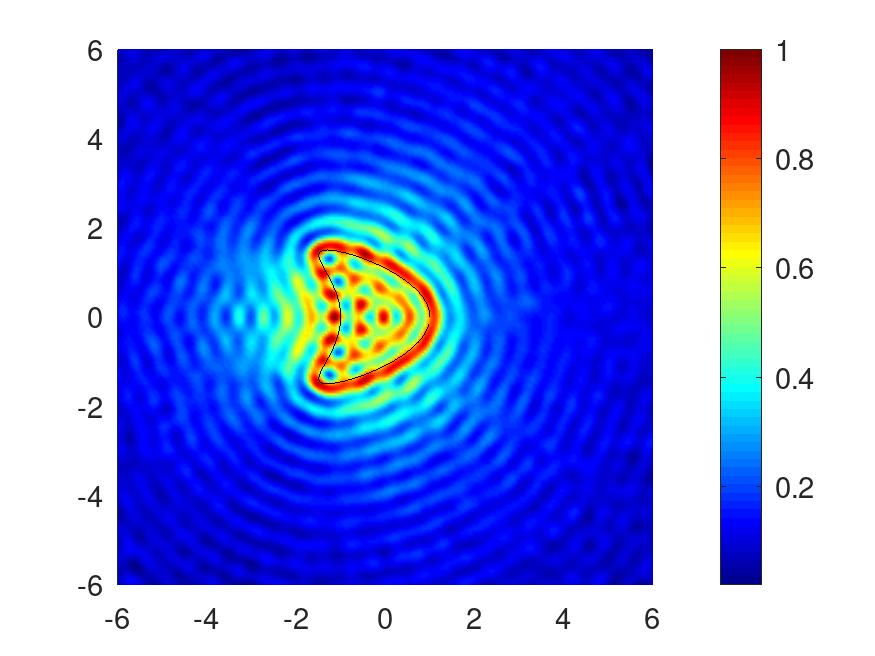}
         \caption{$10\%$ noise}
         \label{fig2_6_3}
       \end{subfigure}
       \end{figure}

       \begin{figure}[H]
         \centering
         \ContinuedFloat
            \begin{subfigure}[b]{0.3\textwidth}
         \centering
         \includegraphics[width=\textwidth]{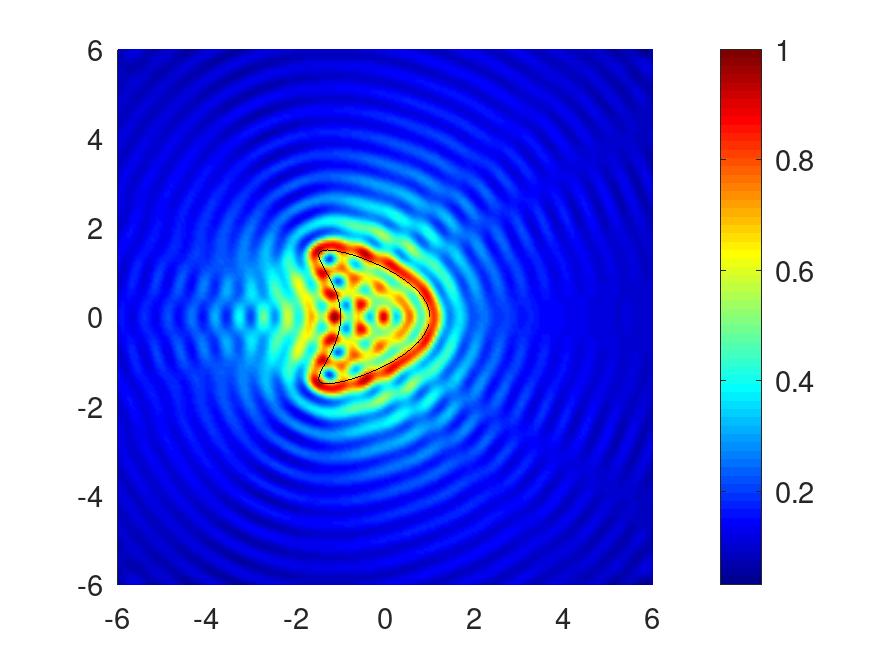}
         \caption{No noise}
         \label{fig2_7_1}
     \end{subfigure}
     \hfill
     \begin{subfigure}[b]{0.3\textwidth}
         \centering
         \includegraphics[width=\textwidth]{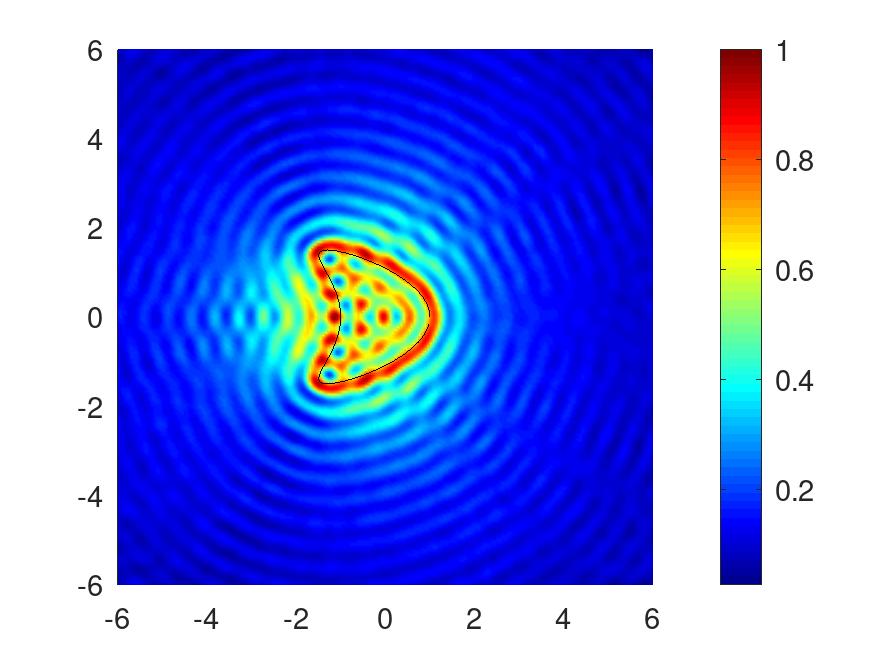}
         \caption{$5\%$ noise}
         \label{fig2_7_2}
     \end{subfigure}
     \hfill
     \begin{subfigure}[b]{0.3\textwidth}
         \centering
         \includegraphics[width=\textwidth]{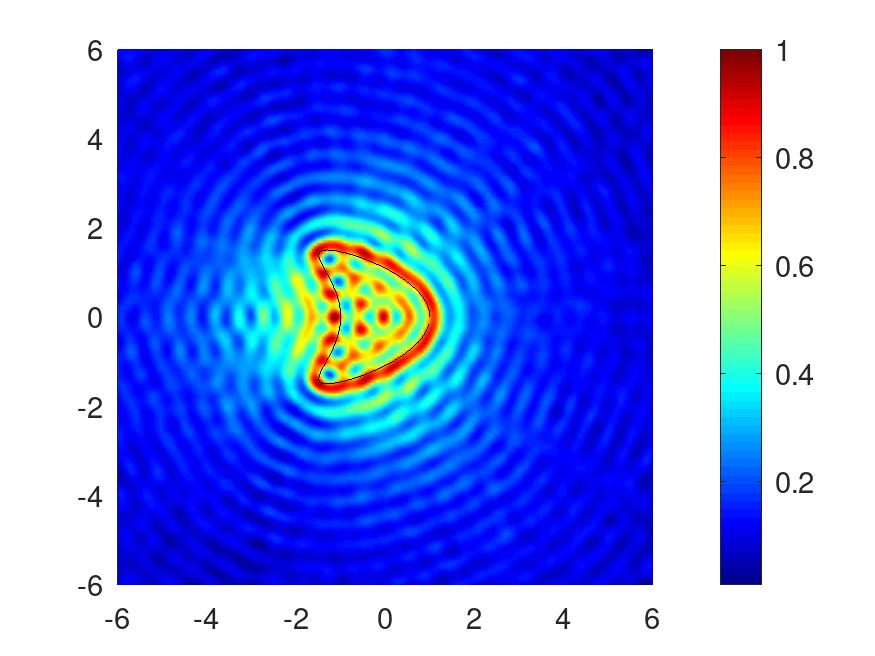}
         \caption{$10\%$ noise}
         \label{fig2_7_3}
     \end{subfigure}

        \begin{subfigure}[b]{0.3\textwidth}
         \centering
         \includegraphics[width=\textwidth]{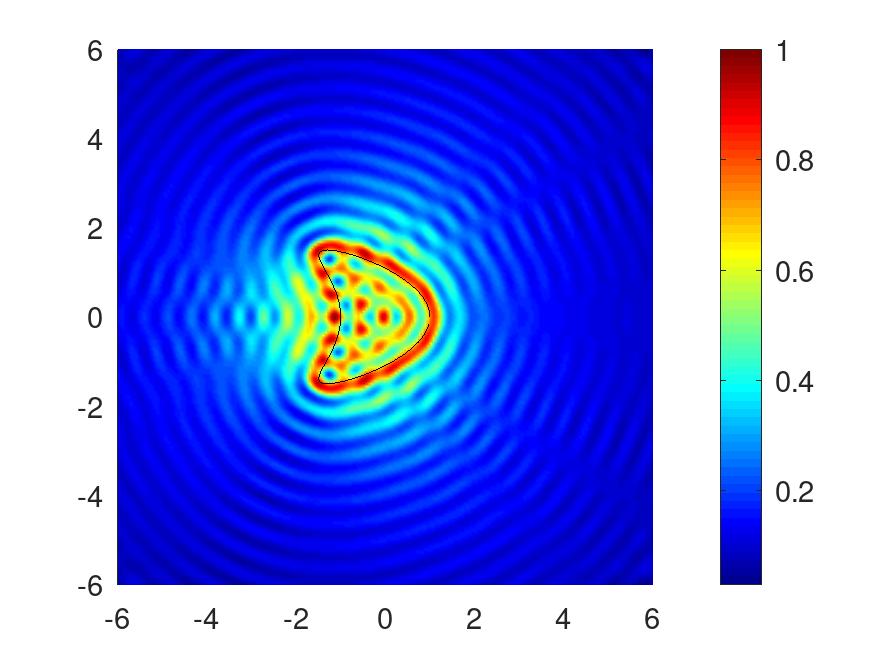}
         \caption{No noise}
         \label{fig2_8_1}
     \end{subfigure}
     \hfill
     \begin{subfigure}[b]{0.3\textwidth}
         \centering
         \includegraphics[width=\textwidth]{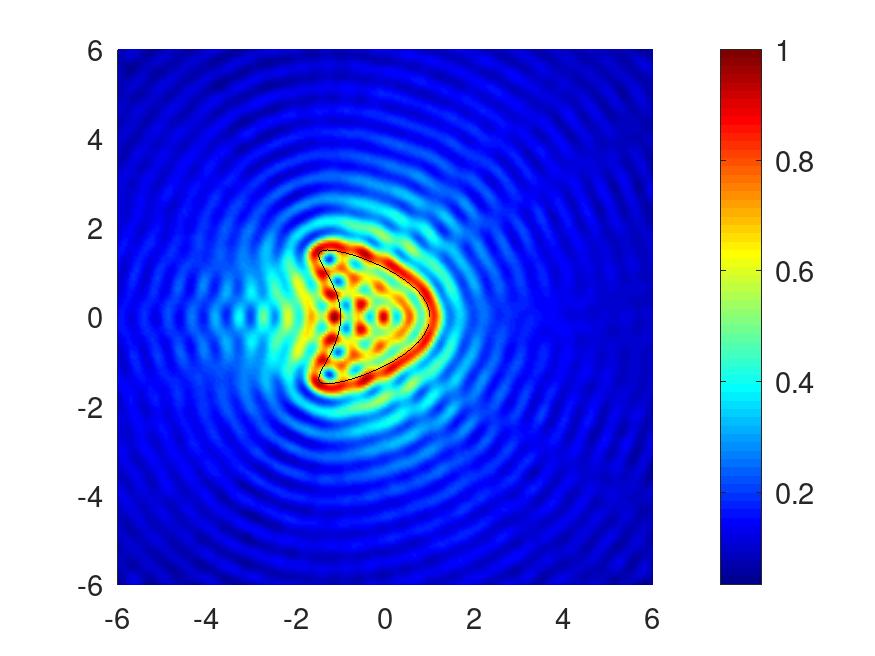}
         \caption{$5\%$ noise}
         \label{fig2_8_2}
     \end{subfigure}
     \hfill
     \begin{subfigure}[b]{0.3\textwidth}
         \centering
         \includegraphics[width=\textwidth]{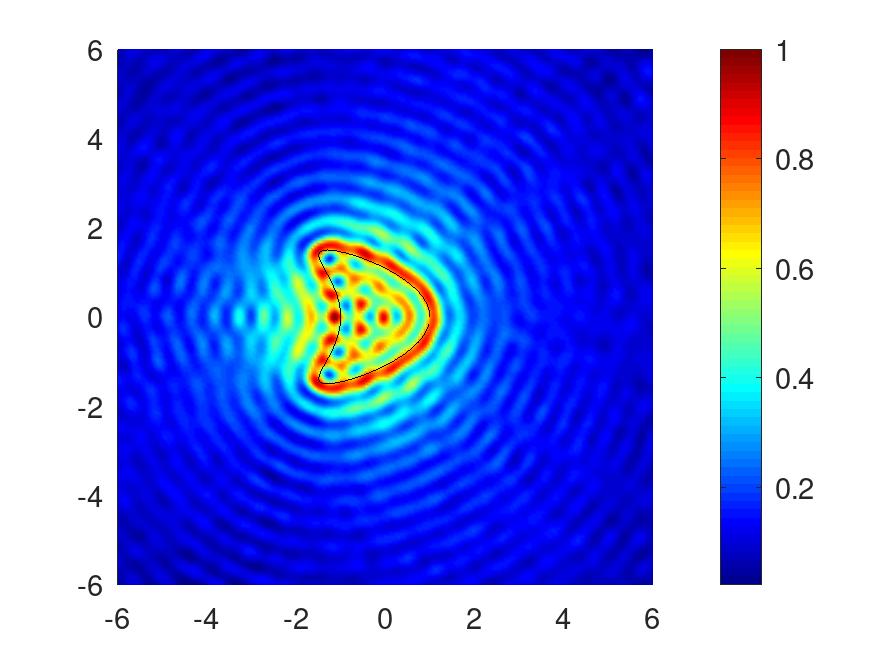}
         \caption{$10\%$ noise}
         \label{fig2_8_3}
        \end{subfigure}

         \begin{subfigure}[b]{0.3\textwidth}
         \centering
        \includegraphics[width=\textwidth]{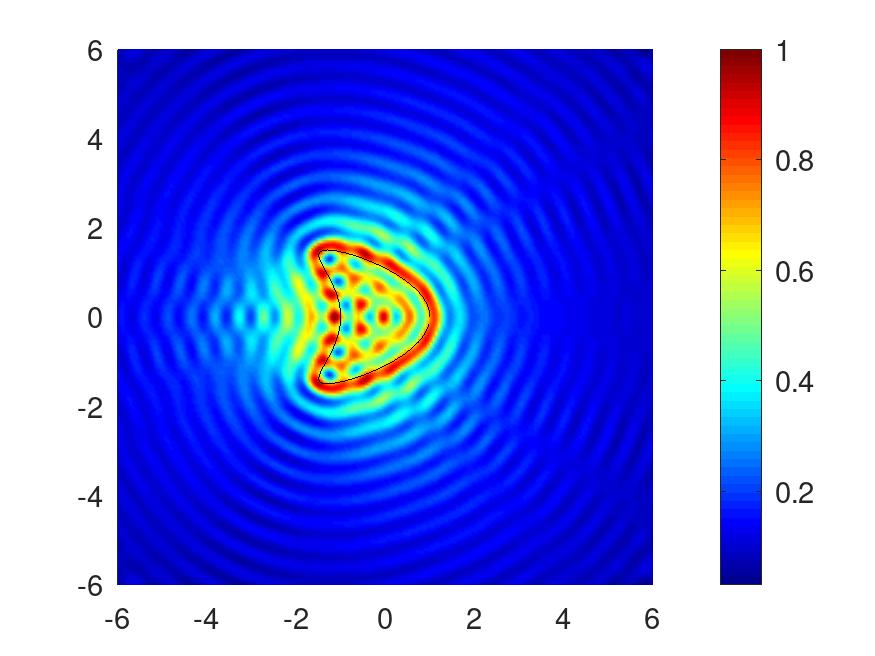}
        \caption{No noise}
         \label{fig2_9_1}
     \end{subfigure}
     \hfill
     \begin{subfigure}[b]{0.3\textwidth}
         \centering
         \includegraphics[width=\textwidth]{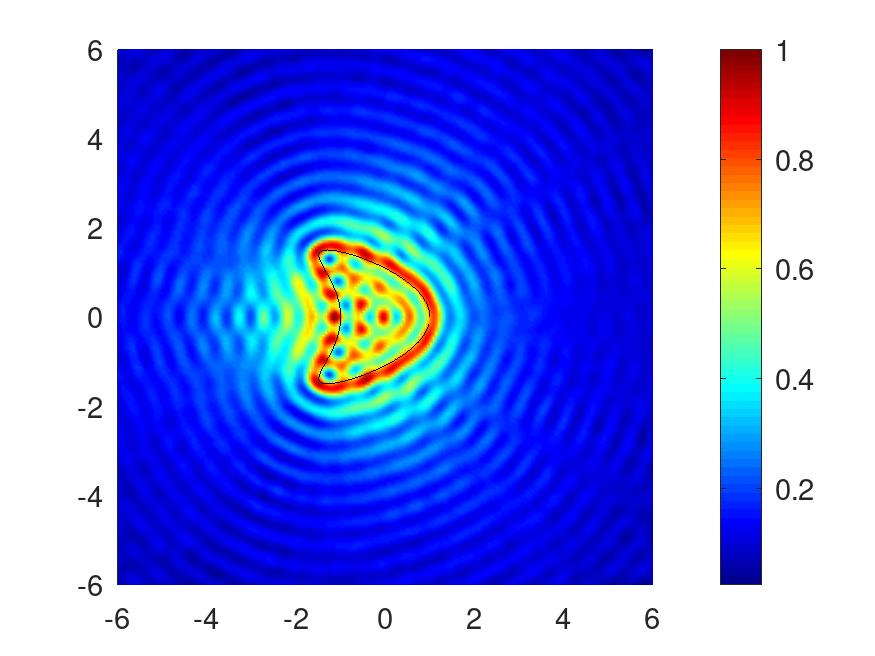}
         \caption{$5\%$ noise}
         \label{fig2_9_2}
     \end{subfigure}
     \hfill
     \begin{subfigure}[b]{0.3\textwidth}
         \centering
         \includegraphics[width=\textwidth]{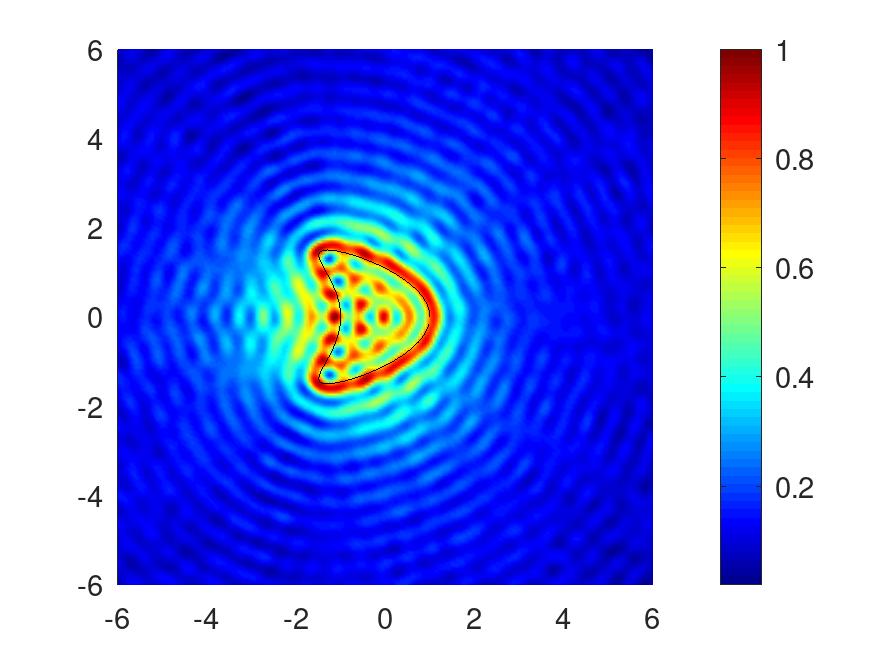}
        \caption{$10\%$ noise}
        \label{fig2_9_3}
      \end{subfigure}

            \begin{subfigure}[b]{0.3\textwidth}
         \centering
         \includegraphics[width=\textwidth]{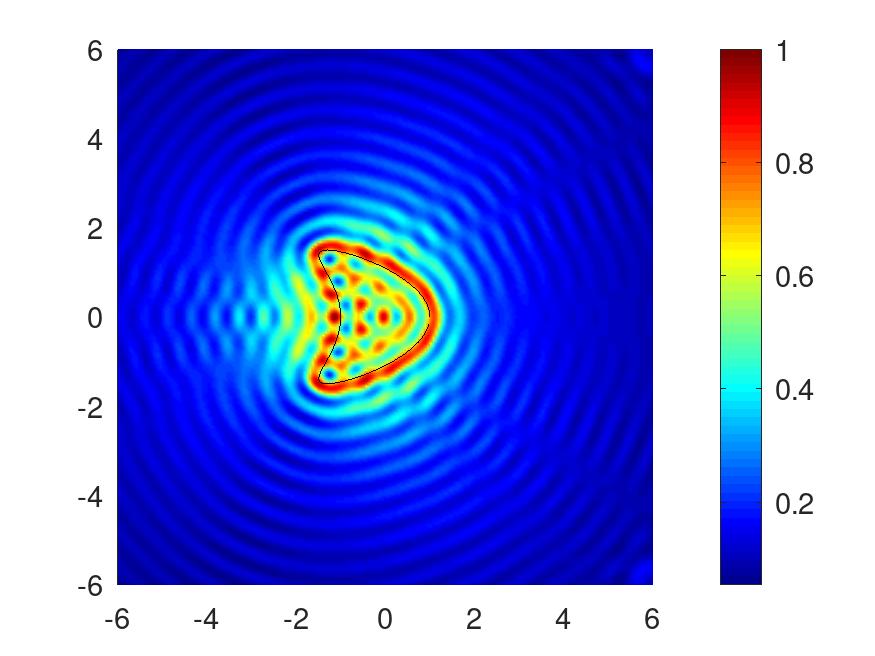}
         \caption{No noise}
         \label{fig2_10_1}
     \end{subfigure}
     \hfill
     \begin{subfigure}[b]{0.3\textwidth}
         \centering
         \includegraphics[width=\textwidth]{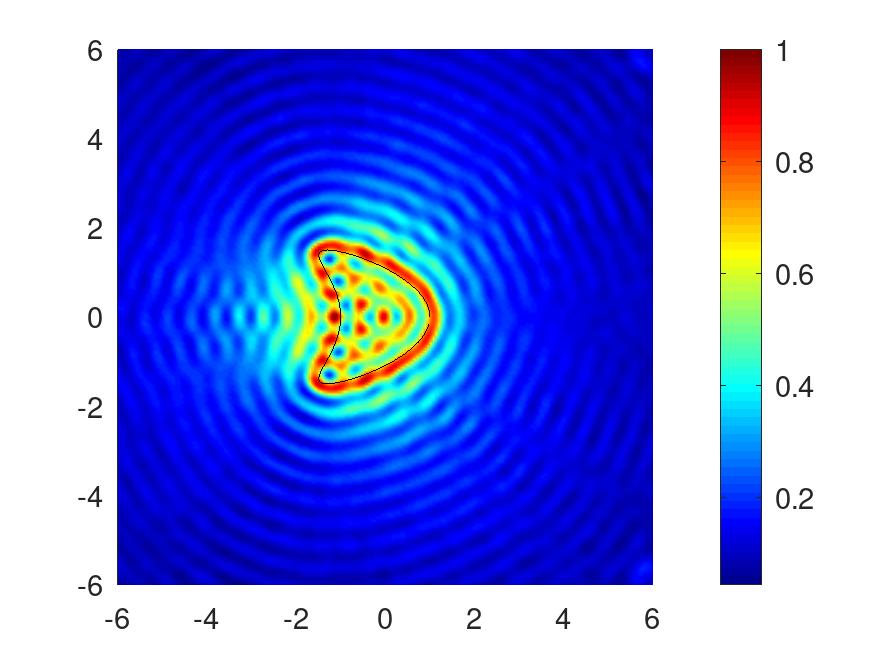}
         \caption{$5\%$ noise}
         \label{fig2_10_2}
     \end{subfigure}
     \hfill
     \begin{subfigure}[b]{0.3\textwidth}
         \centering
         \includegraphics[width=\textwidth]{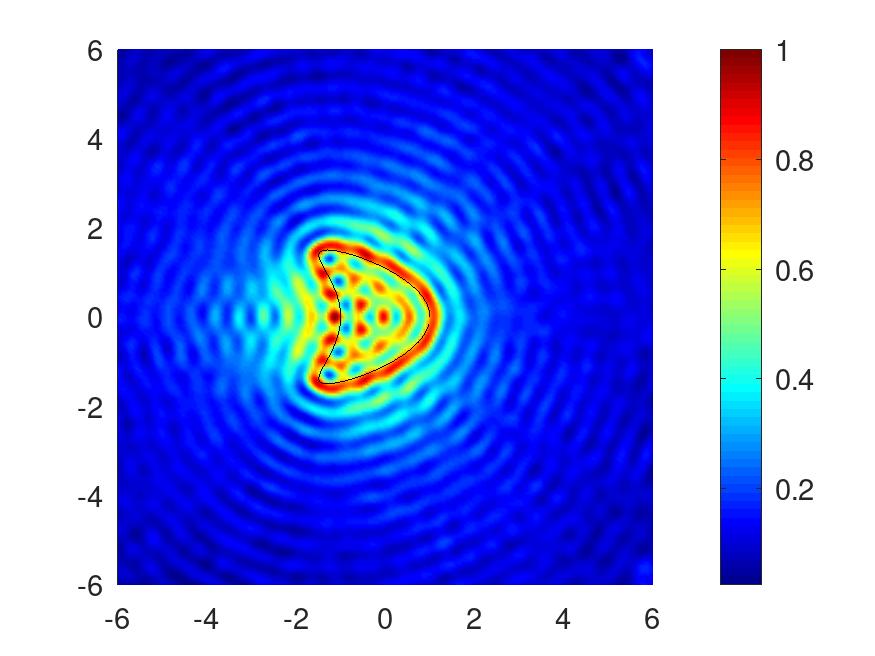}
         \caption{$10\%$ noise}
         \label{fig2_10_3}
       \end{subfigure}

           \begin{subfigure}[b]{0.3\textwidth}
        \centering
        \includegraphics[width=\textwidth]{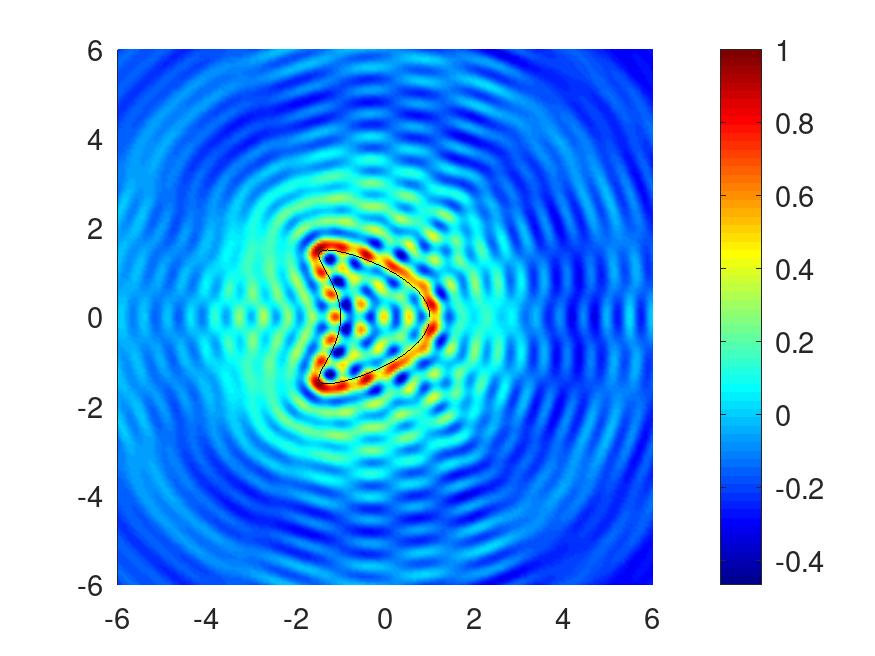}
        \caption{No noise}
        \label{fig2_11_1}
    \end{subfigure}
    \hfill
     \begin{subfigure}[b]{0.3\textwidth}
         \centering
         \includegraphics[width=\textwidth]{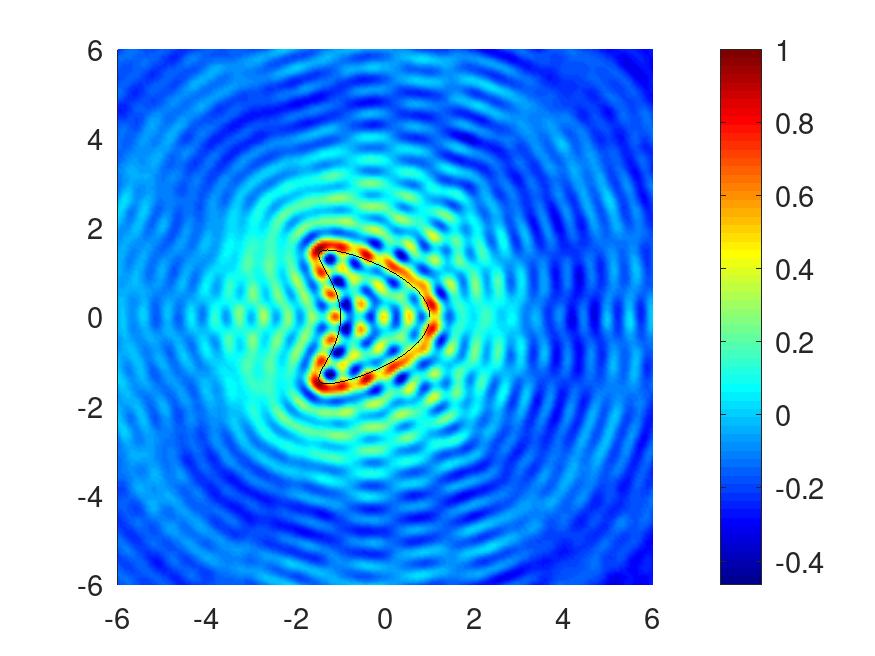}
         \caption{$5\%$ noise}
         \label{fig2_11_2}
     \end{subfigure}
     \hfill
    \begin{subfigure}[b]{0.3\textwidth}
        \centering
         \includegraphics[width=\textwidth]{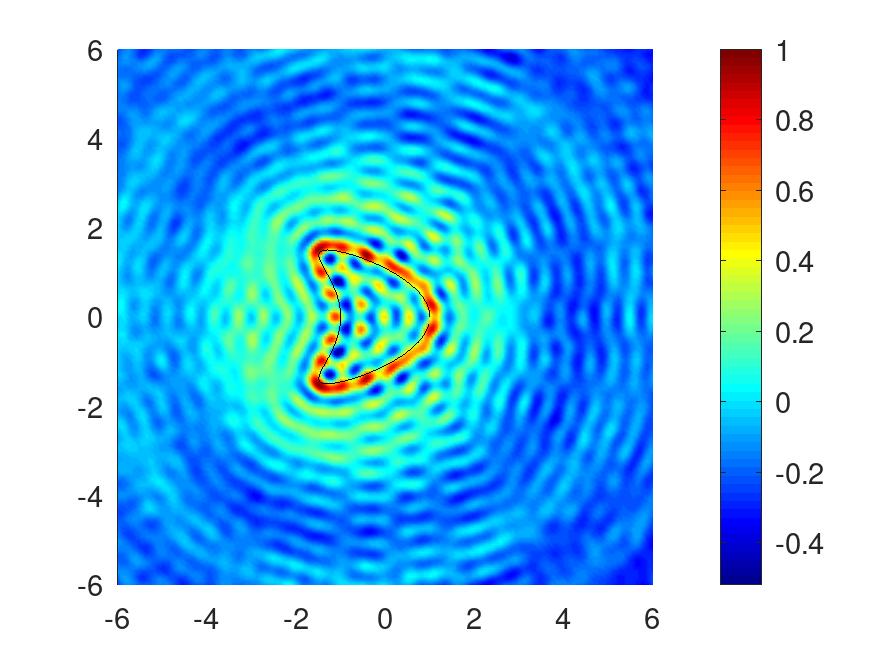}
         \caption{$10\%$ noise}
         \label{fig2_11_3}
       \end{subfigure}

               \caption{Example 2: The $j-$th row presents the reconstruction results of the indicator function $I_j$ $(j=1,\cdots,11)$.}
        \label{fig2:kite}
\end{figure}

{\bf Example 3.} We consider the case of two obstacles, i.e., a circle and a kite.
\begin{align*}
  (\cos t -2, \sin t- 2)\quad t \in [0,2\pi],\\
    (0.65 \cos(2t)+\cos t+1.35, 1.5\sin t+2) \quad t \in [0,2\pi].
  \end{align*}
   The sampling domain is  $[-6,6]\times [-6,6]$ and the wave number $\kappa$ is $2\pi$. The reconstruction results are presented in figure \ref{fig3:obstalces}. Clearly, our algorithm also works well for multiple obstacles.
     \begin{figure}[H]
     \centering
     \begin{subfigure}[b]{0.3\textwidth}
         \centering
         \includegraphics[width=\textwidth]{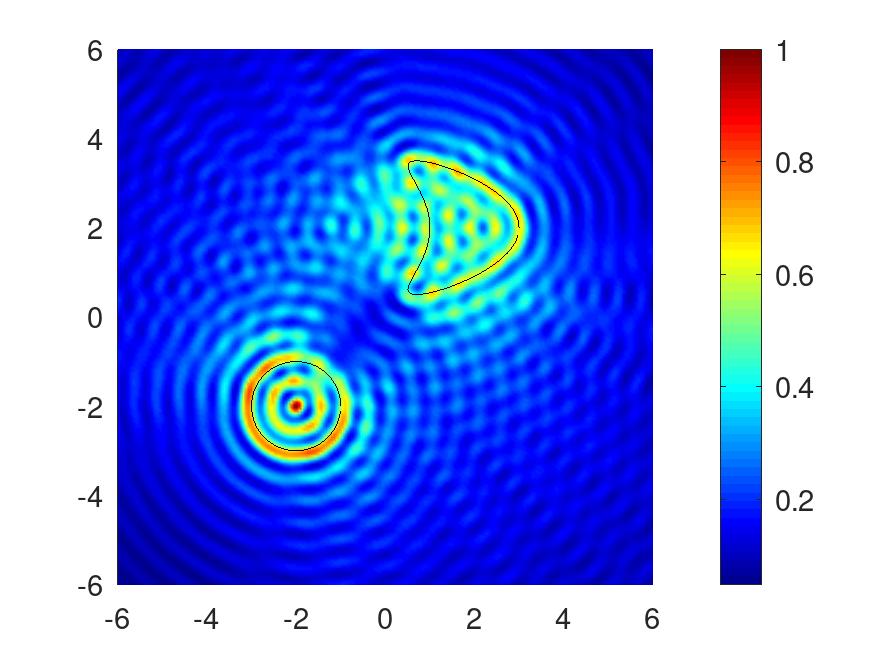}
         \caption{No noise}
         \label{fig3_1_1}
     \end{subfigure}
     \hfill
     \begin{subfigure}[b]{0.3\textwidth}
         \centering
         \includegraphics[width=\textwidth]{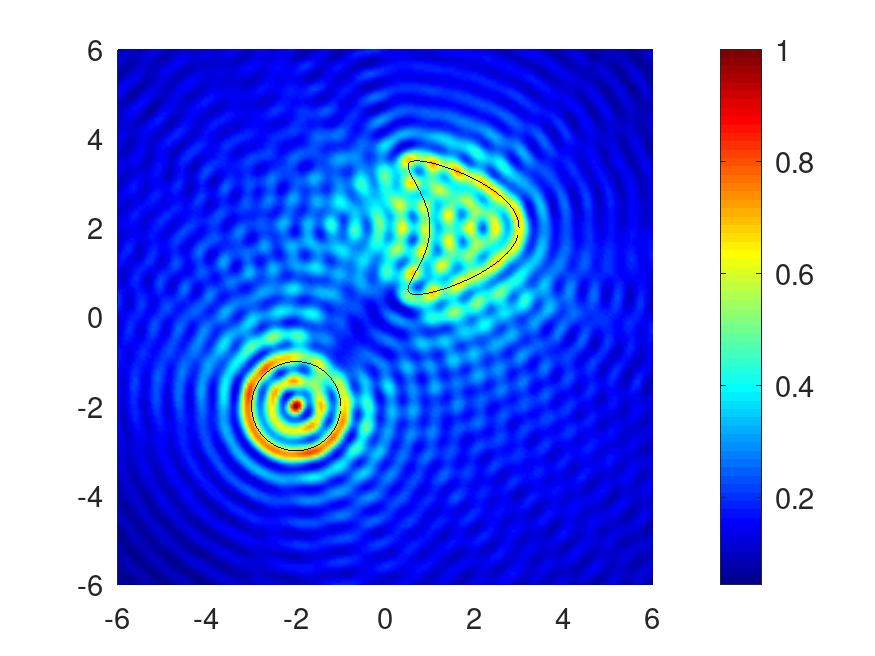}
         \caption{$5\%$ noise}
         \label{fig3_1_2}
     \end{subfigure}
     \hfill
     \begin{subfigure}[b]{0.3\textwidth}
         \centering
         \includegraphics[width=\textwidth]{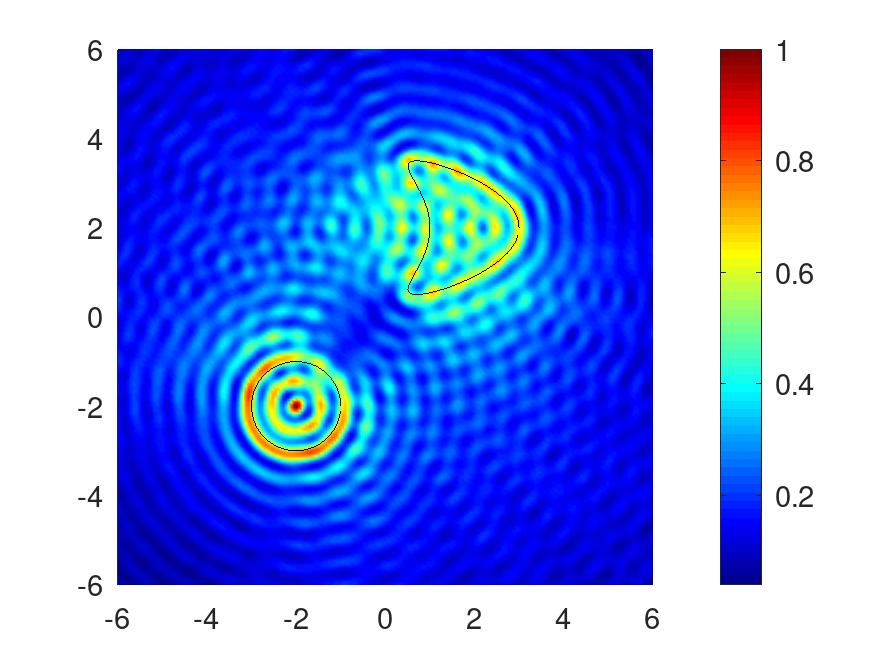}
         \caption{$10\%$ noise}
         \label{fig3_1_3}
     \end{subfigure}

          \begin{subfigure}[b]{0.3\textwidth}
         \centering
         \includegraphics[width=\textwidth]{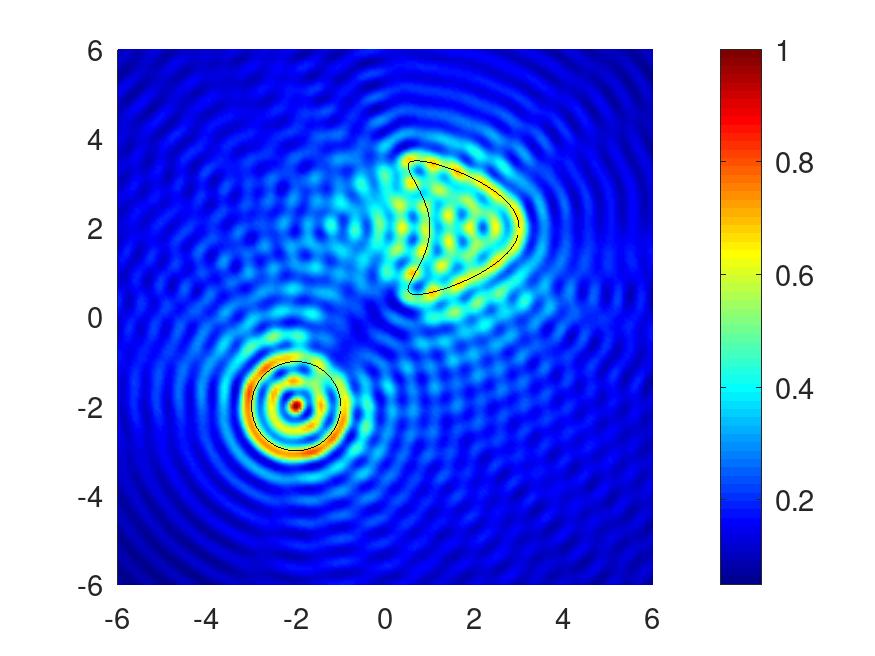}
         \caption{No noise}
         \label{fig3_2_1}
     \end{subfigure}
     \hfill
     \begin{subfigure}[b]{0.3\textwidth}
         \centering
         \includegraphics[width=\textwidth]{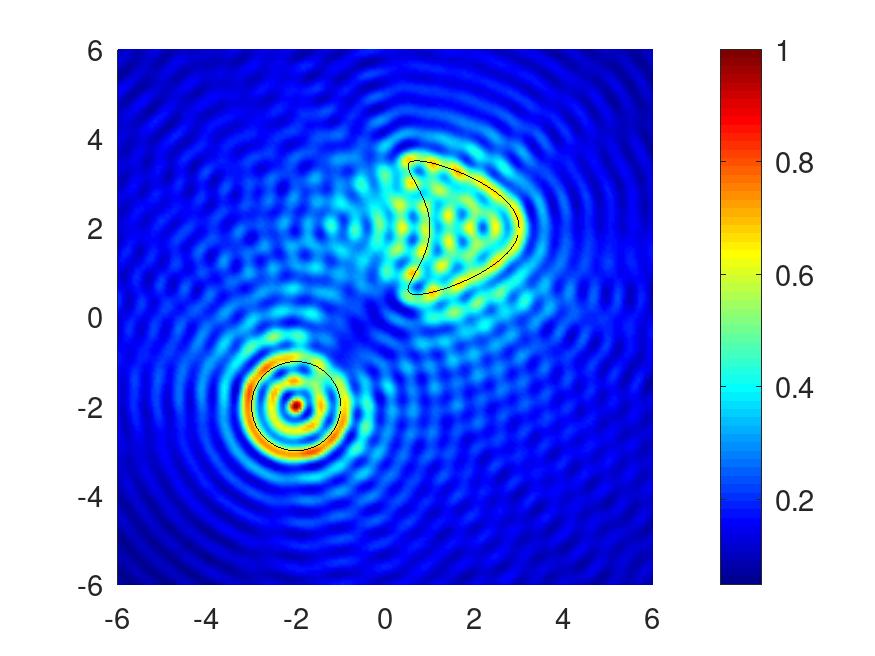}
         \caption{$5\%$ noise}
         \label{fig3_2_2}
     \end{subfigure}
     \hfill
     \begin{subfigure}[b]{0.3\textwidth}
         \centering
         \includegraphics[width=\textwidth]{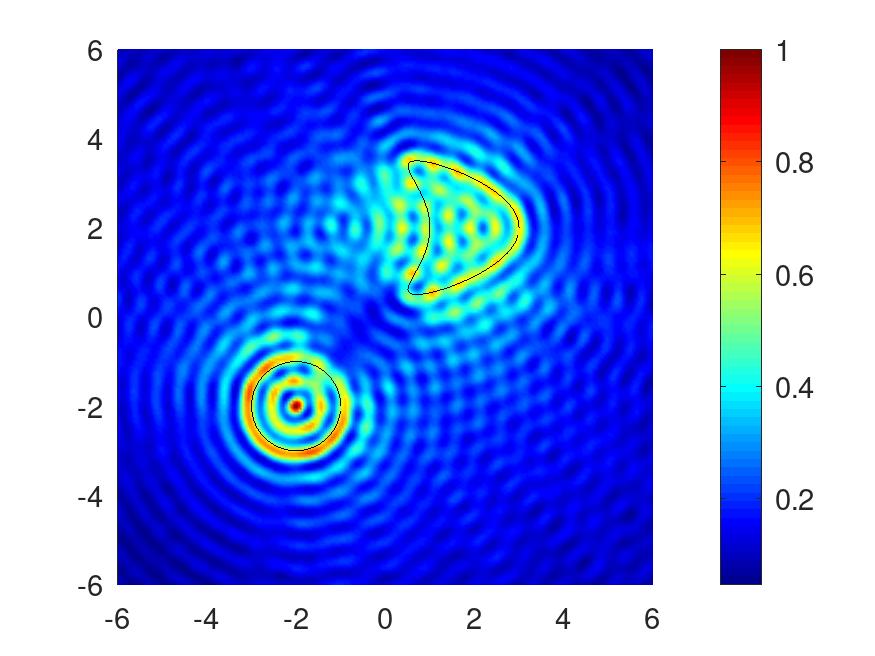}
         \caption{$10\%$ noise}
         \label{fig3_2_3}
       \end{subfigure}
       
            \begin{subfigure}[b]{0.3\textwidth}
         \centering
         \includegraphics[width=\textwidth]{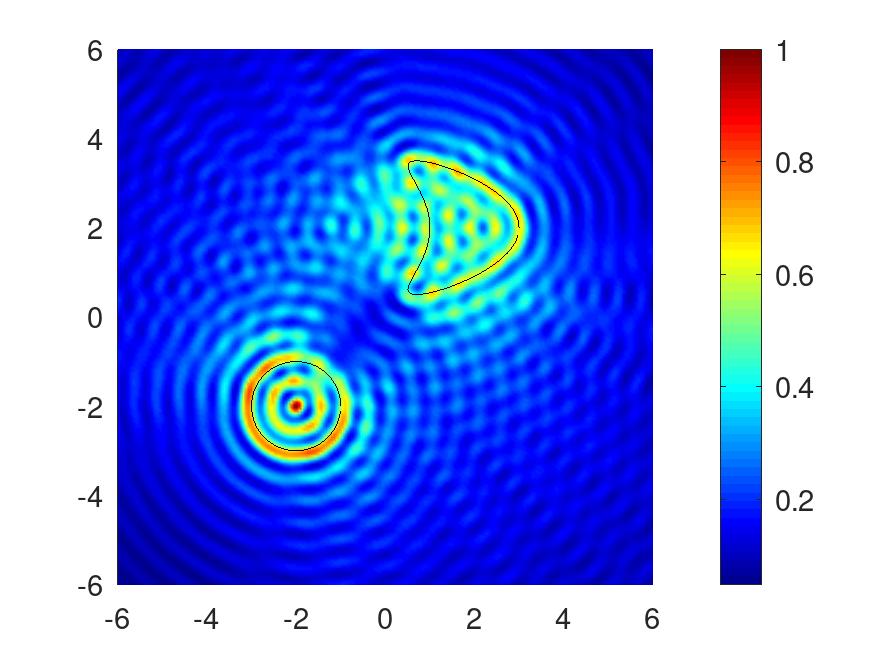}
         \caption{No noise}
         \label{fig3_3_1}
     \end{subfigure}
     \hfill
     \begin{subfigure}[b]{0.3\textwidth}
         \centering
         \includegraphics[width=\textwidth]{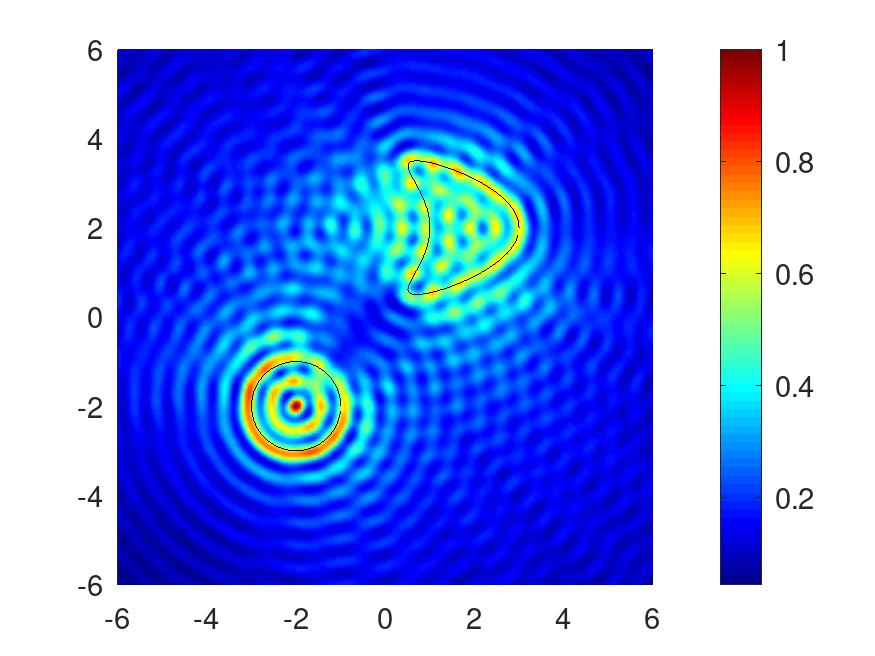}
         \caption{$5\%$ noise}
         \label{fig3_3_2}
     \end{subfigure}
     \hfill
     \begin{subfigure}[b]{0.3\textwidth}
         \centering
         \includegraphics[width=\textwidth]{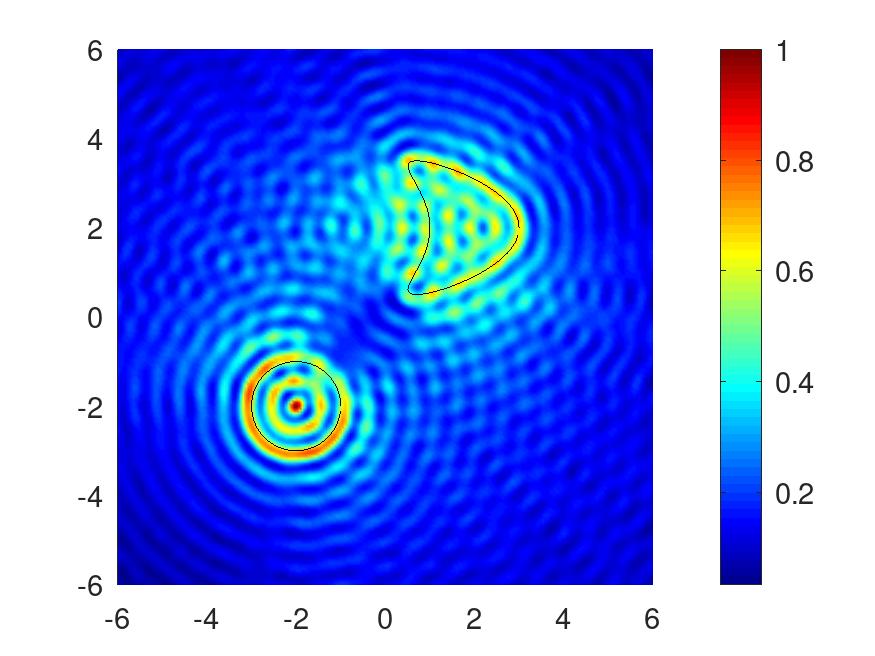}
         \caption{$10\%$ noise}
         \label{fig3_3_3}
       \end{subfigure}

            \begin{subfigure}[b]{0.3\textwidth}
         \centering
         \includegraphics[width=\textwidth]{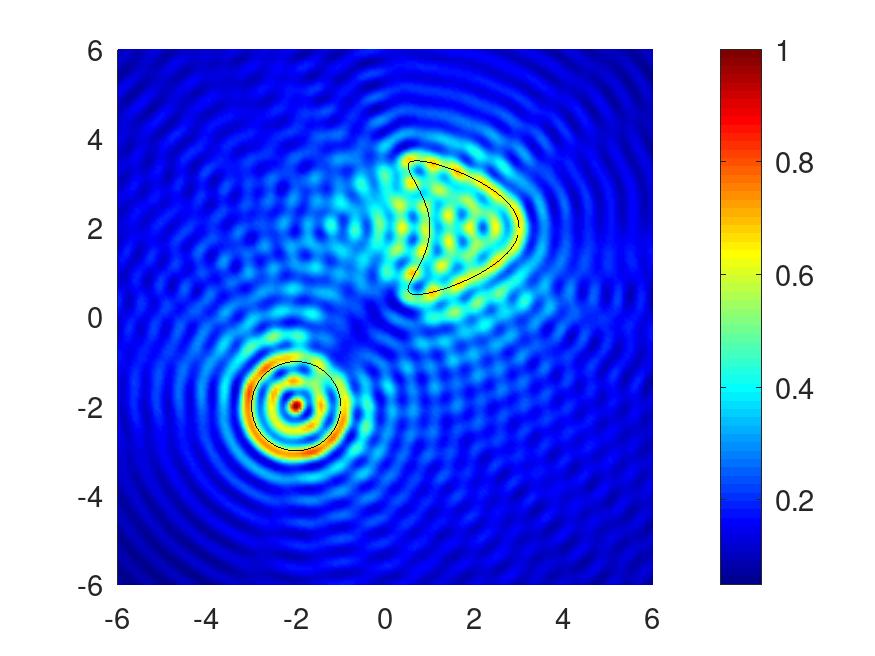}
         \caption{No noise}
         \label{fig3_4_1}
     \end{subfigure}
     \hfill
     \begin{subfigure}[b]{0.3\textwidth}
         \centering
         \includegraphics[width=\textwidth]{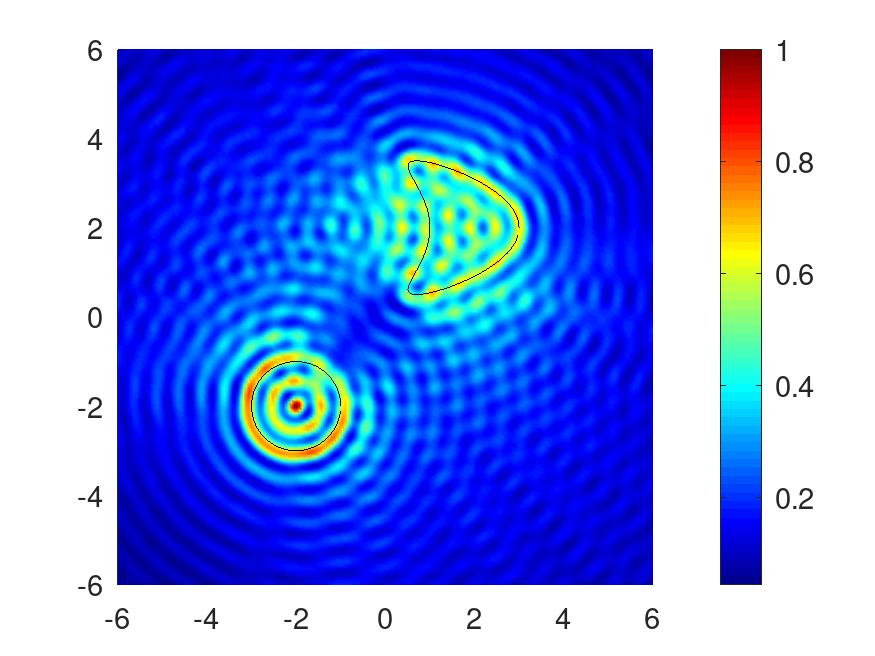}
         \caption{$5\%$ noise}
         \label{fig3_4_2}
     \end{subfigure}
     \hfill
     \begin{subfigure}[b]{0.3\textwidth}
         \centering
         \includegraphics[width=\textwidth]{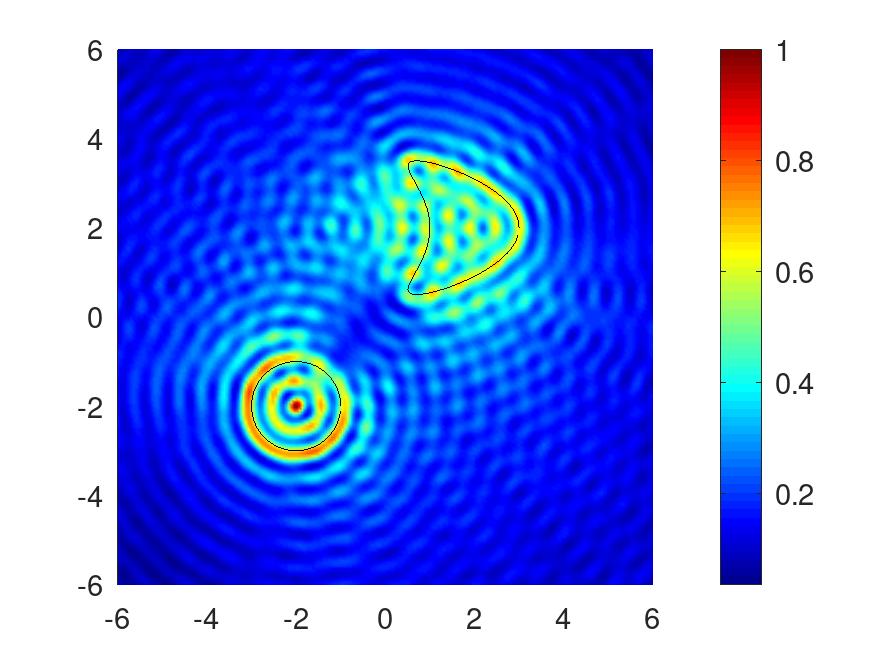}
         \caption{$10\%$ noise}
         \label{fig3_4_3}
       \end{subfigure}
     \end{figure}

     \begin{figure}[H]
                \centering
         \ContinuedFloat
            \begin{subfigure}[b]{0.3\textwidth}
         \centering
         \includegraphics[width=\textwidth]{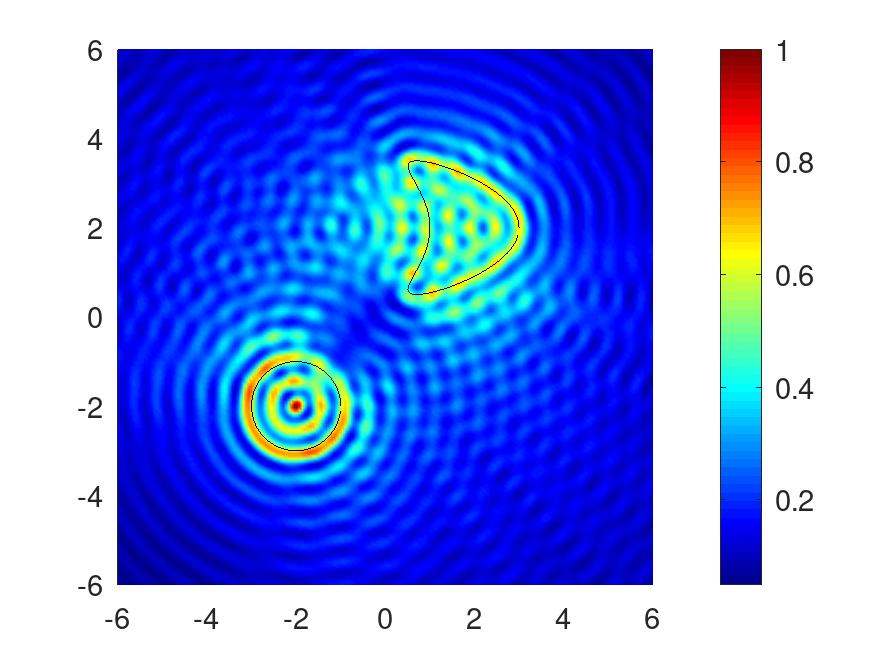}
         \caption{No noise}
         \label{fig3_5_1}
     \end{subfigure}
     \hfill
     \begin{subfigure}[b]{0.3\textwidth}
         \centering
         \includegraphics[width=\textwidth]{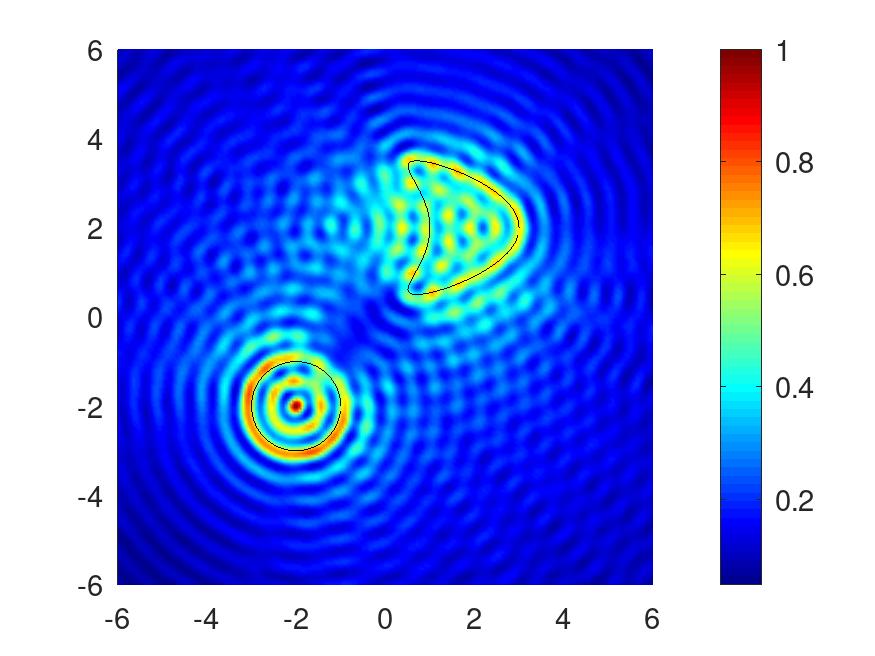}
         \caption{$5\%$ noise}
         \label{fig3_5_2}
     \end{subfigure}
     \hfill
     \begin{subfigure}[b]{0.3\textwidth}
         \centering
         \includegraphics[width=\textwidth]{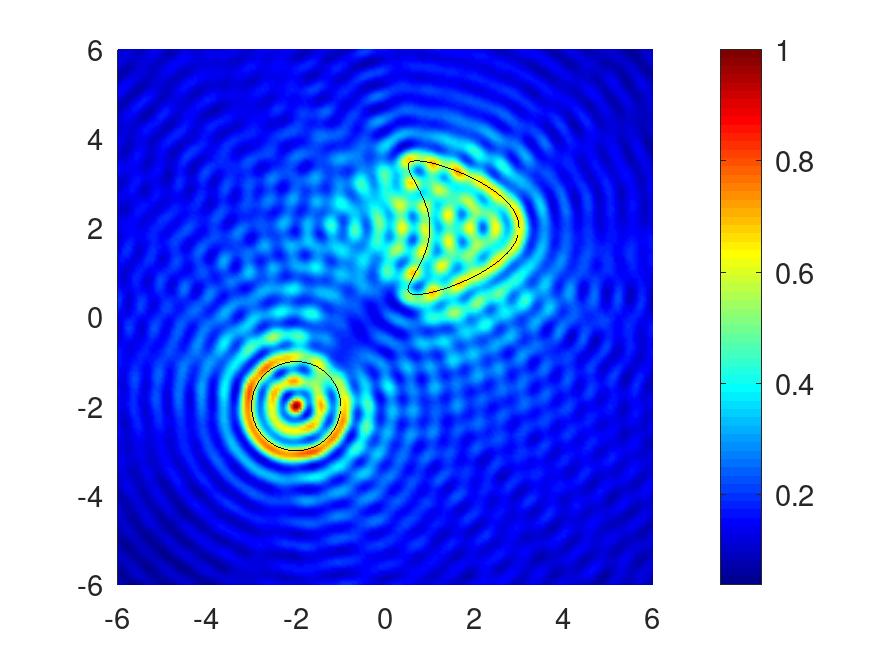}
         \caption{$10\%$ noise}
         \label{fig3_5_3}
       \end{subfigure}

            \begin{subfigure}[b]{0.3\textwidth}
         \centering
         \includegraphics[width=\textwidth]{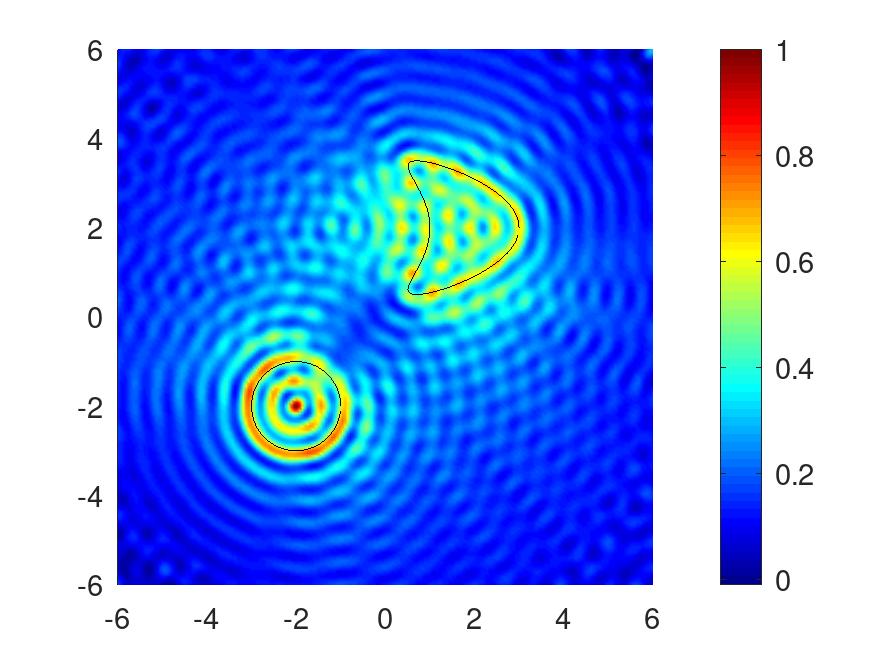}
         \caption{No noise}
         \label{fig3_6_1}
     \end{subfigure}
     \hfill
     \begin{subfigure}[b]{0.3\textwidth}
         \centering
         \includegraphics[width=\textwidth]{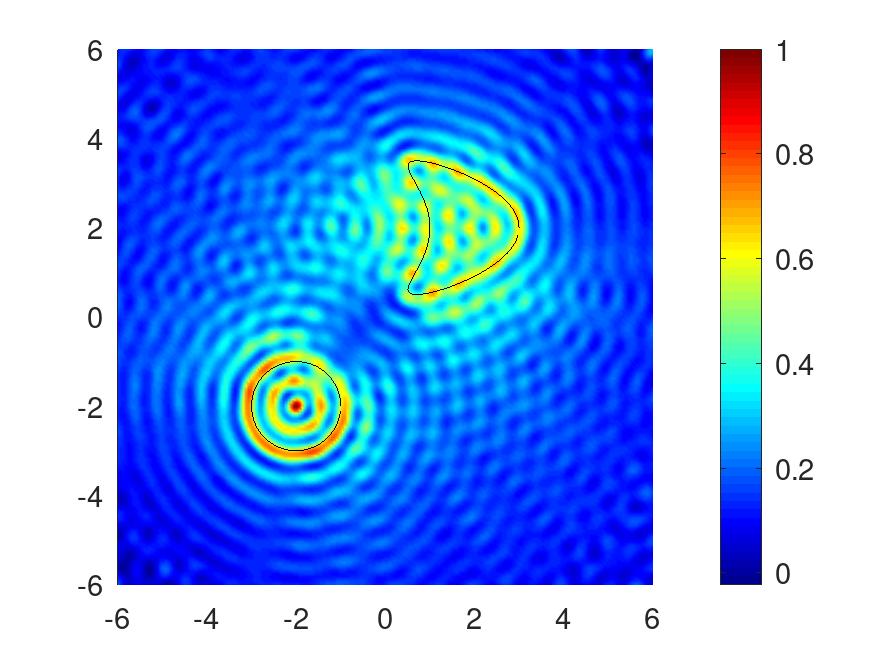}
         \caption{$5\%$ noise}
         \label{fig3_6_2}
     \end{subfigure}
     \hfill
     \begin{subfigure}[b]{0.3\textwidth}
         \centering
         \includegraphics[width=\textwidth]{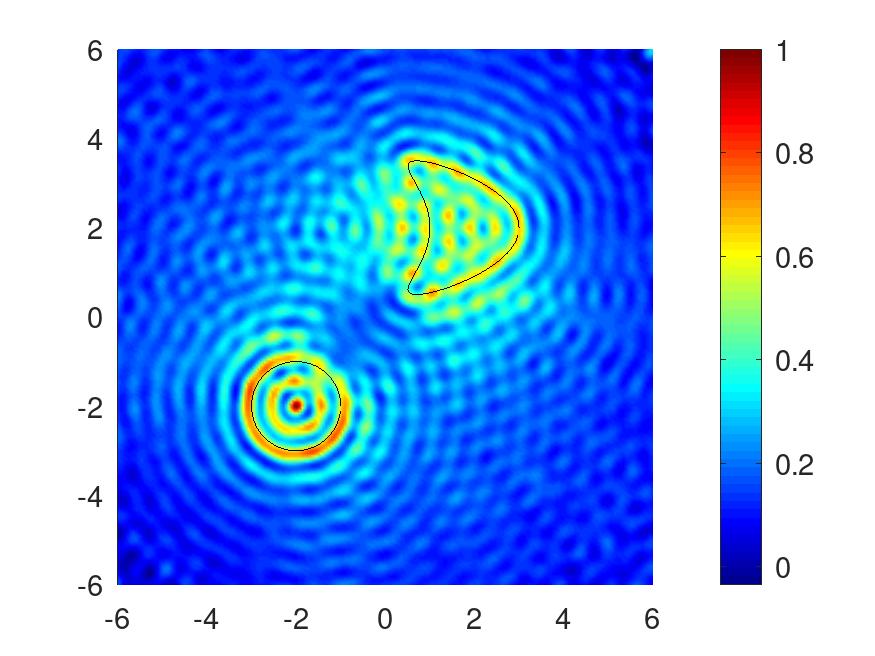}
         \caption{$10\%$ noise}
         \label{fig3_6_3}
       \end{subfigure}

            \begin{subfigure}[b]{0.3\textwidth}
         \centering
         \includegraphics[width=\textwidth]{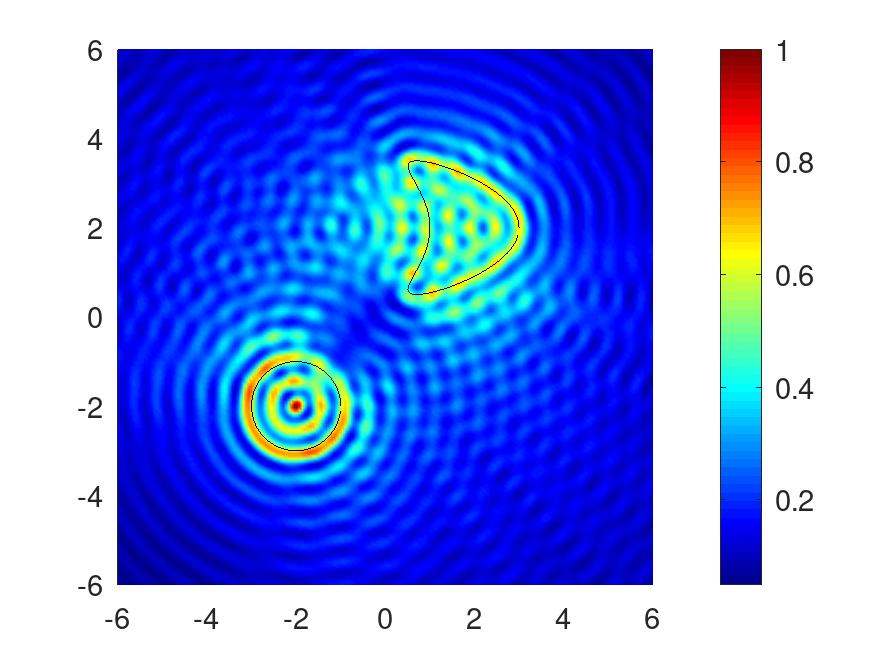}
         \caption{No noise}
         \label{fig3_7_1}
     \end{subfigure}
     \hfill
     \begin{subfigure}[b]{0.3\textwidth}
         \centering
         \includegraphics[width=\textwidth]{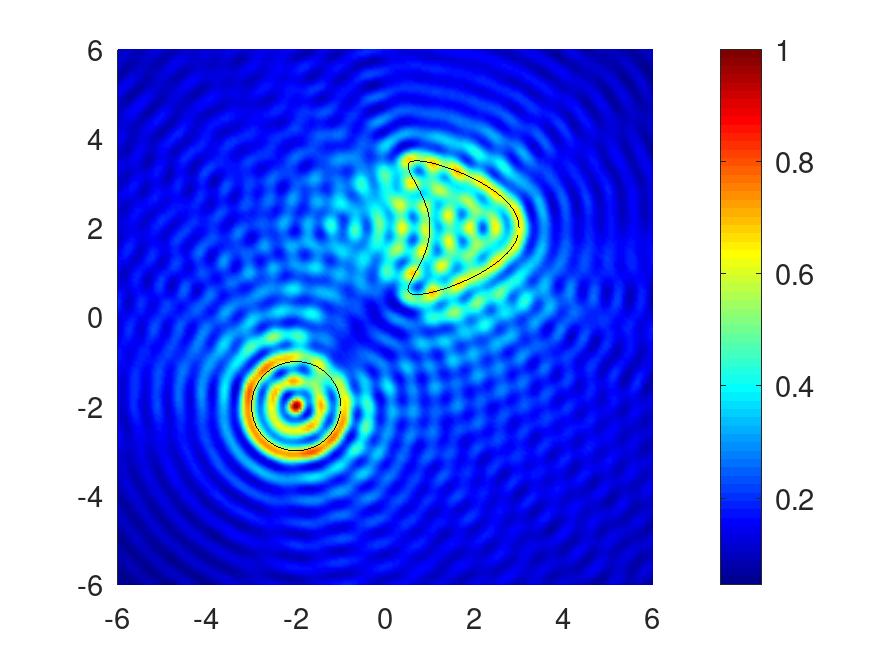}
         \caption{$5\%$ noise}
         \label{fig3_7_2}
     \end{subfigure}
     \hfill
     \begin{subfigure}[b]{0.3\textwidth}
         \centering
         \includegraphics[width=\textwidth]{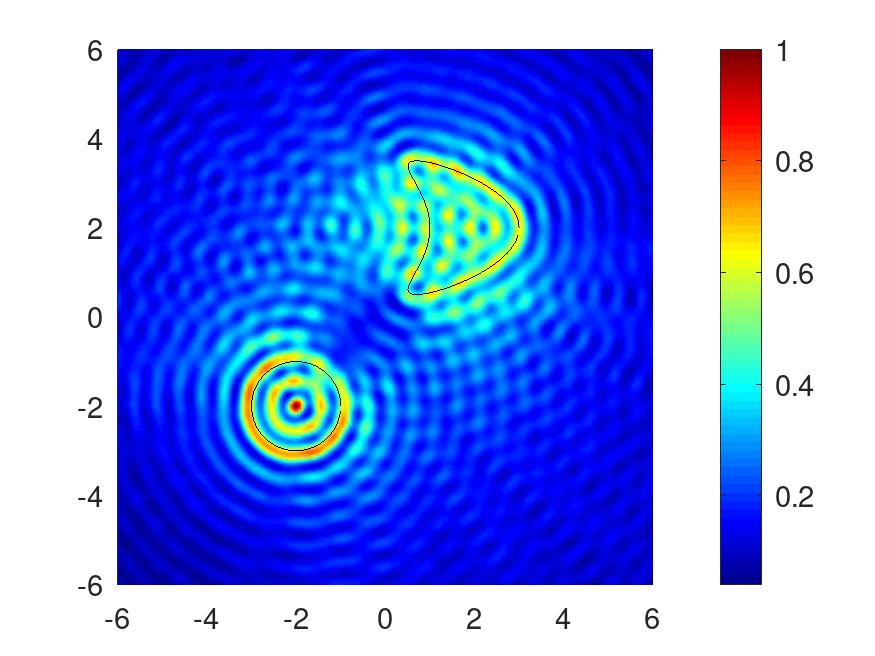}
         \caption{$10\%$ noise}
         \label{fig3_7_3}
     \end{subfigure}
     \newline
          \begin{subfigure}[b]{0.3\textwidth}
         \centering
         \includegraphics[width=\textwidth]{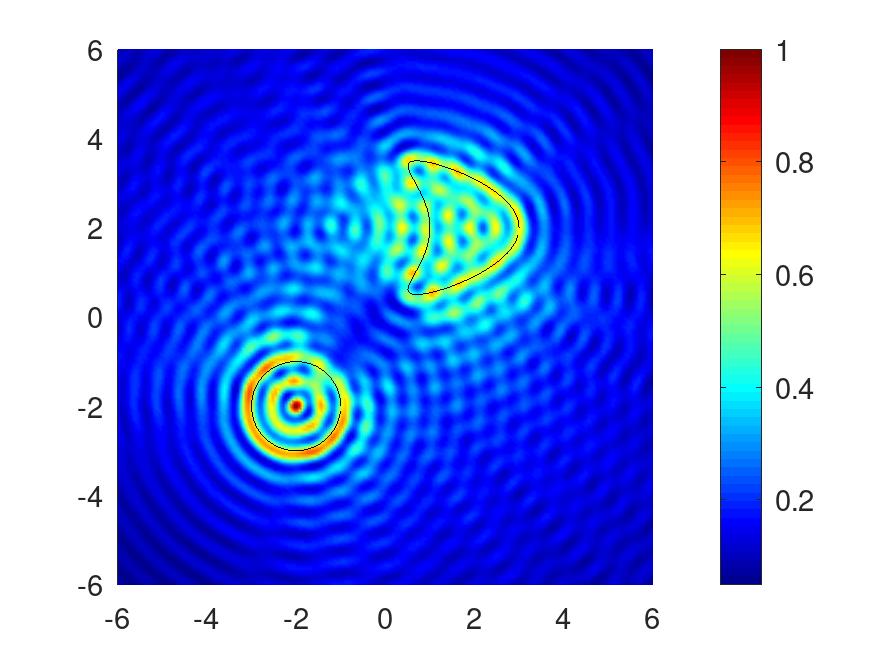}
         \caption{No noise}
         \label{fig3_8_1}
     \end{subfigure}
     \hfill
     \begin{subfigure}[b]{0.3\textwidth}
         \centering
         \includegraphics[width=\textwidth]{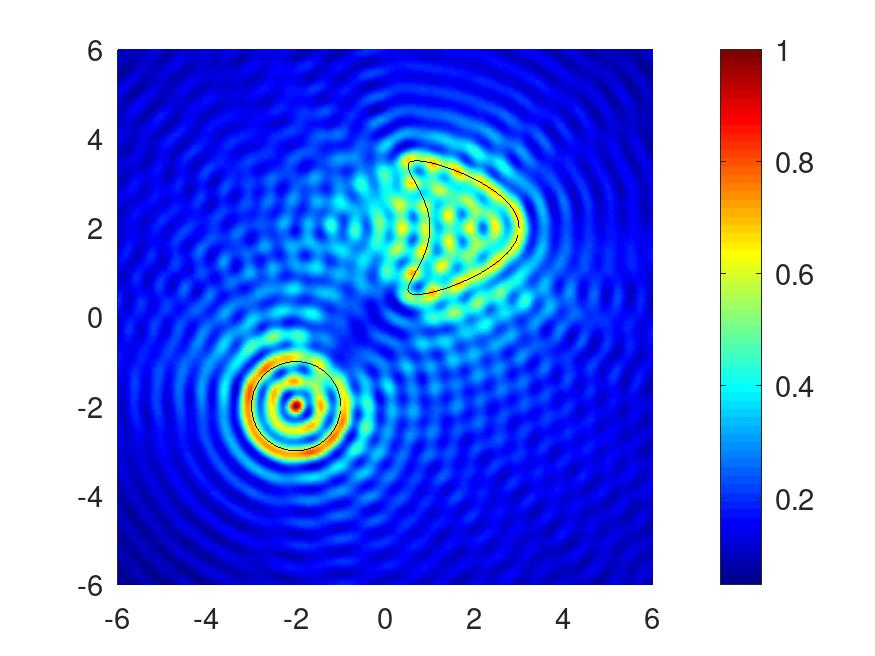}
         \caption{$5\%$ noise}
         \label{fig3_8_2}
     \end{subfigure}
     \hfill
     \begin{subfigure}[b]{0.3\textwidth}
         \centering
         \includegraphics[width=\textwidth]{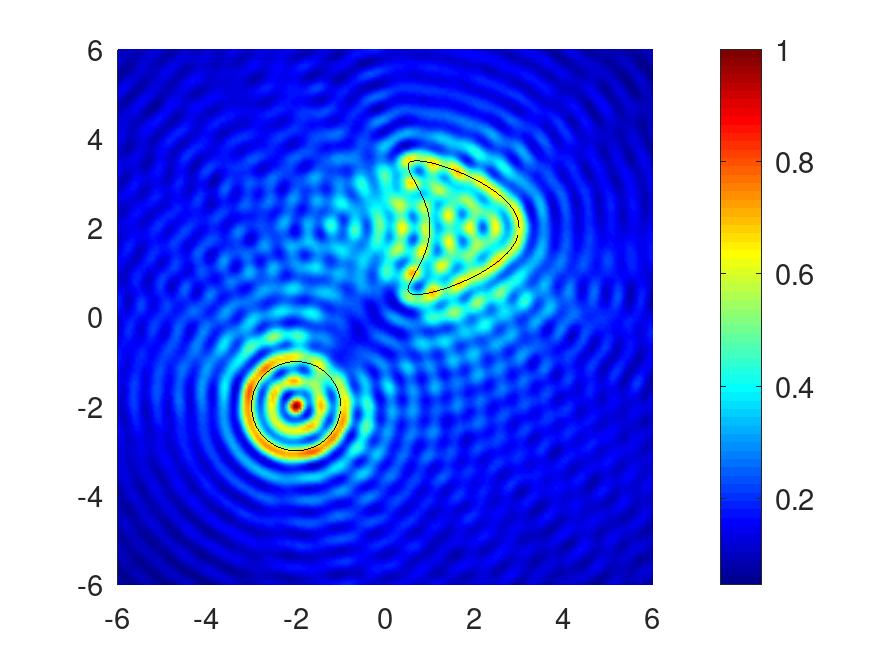}
         \caption{$10\%$ noise}
         \label{fig3_8_3}
       \end{subfigure}

           \begin{subfigure}[b]{0.3\textwidth}
        \centering
        \includegraphics[width=\textwidth]{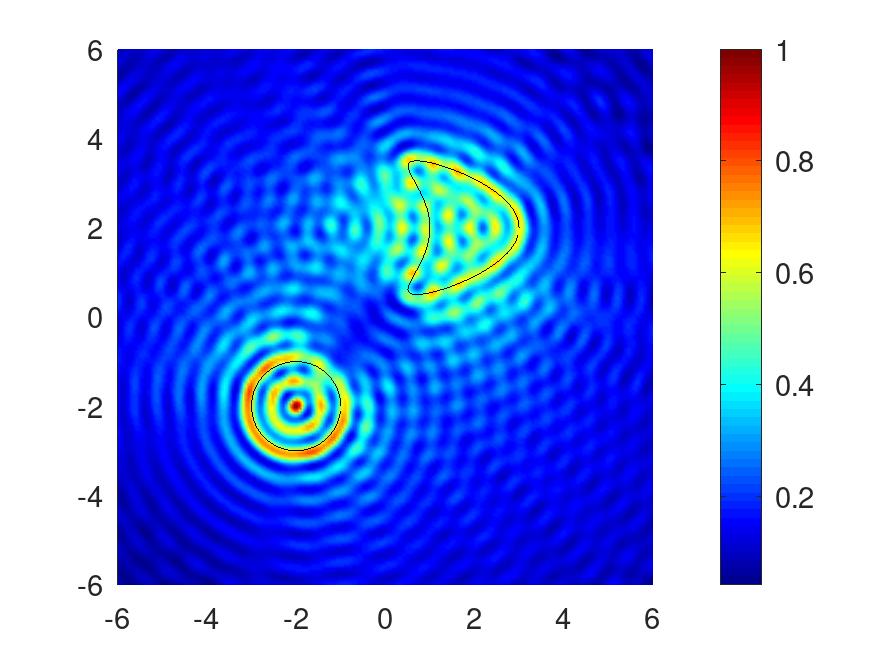}
        \caption{No noise}
         \label{fig3_9_1}
     \end{subfigure}
     \hfill
     \begin{subfigure}[b]{0.3\textwidth}
         \centering
         \includegraphics[width=\textwidth]{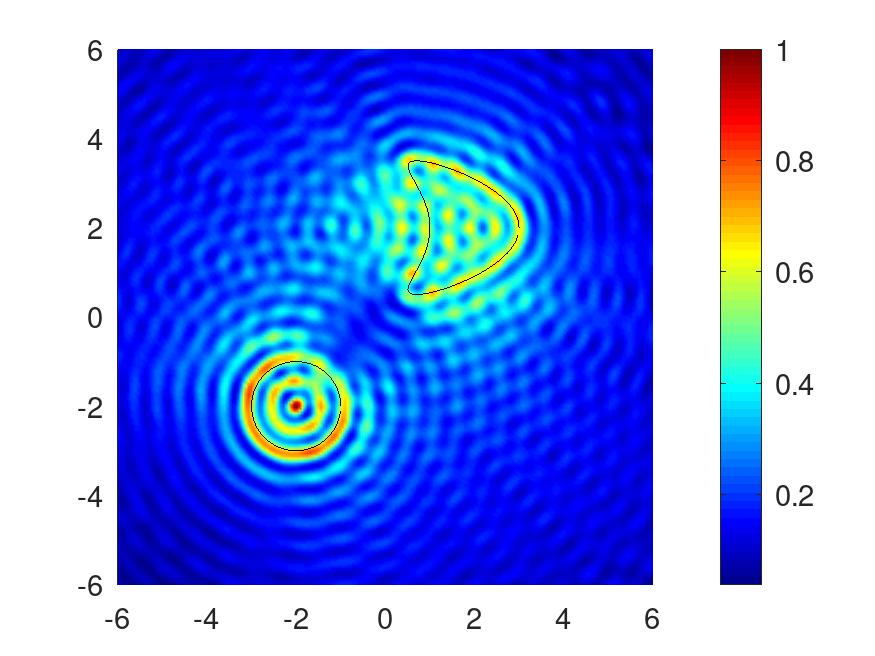}
         \caption{$5\%$ noise}
         \label{fig3_9_2}
     \end{subfigure}
     \hfill
     \begin{subfigure}[b]{0.3\textwidth}
         \centering
         \includegraphics[width=\textwidth]{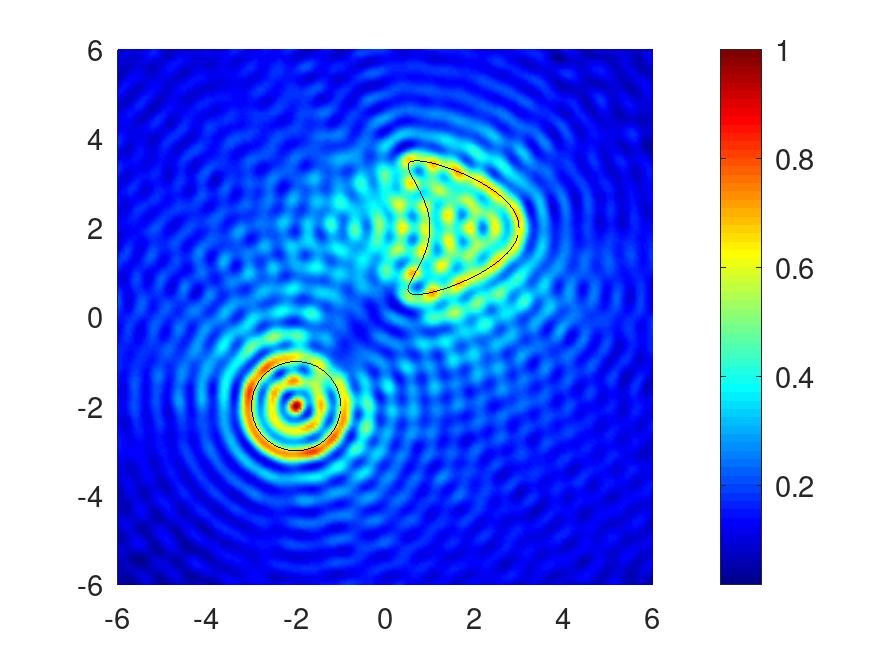}
        \caption{$10\%$ noise}
        \label{fig3_9_3}
      \end{subfigure}
     \end{figure}
     
     \begin{figure}[H]
        \centering
       \ContinuedFloat
            \begin{subfigure}[b]{0.3\textwidth}
         \centering
         \includegraphics[width=\textwidth]{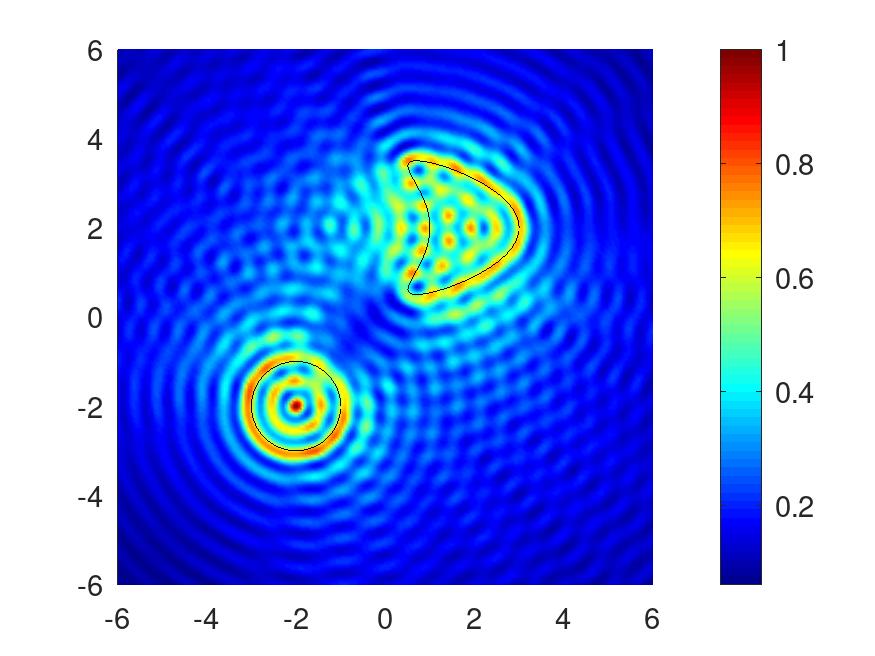}
         \caption{No noise}
         \label{fig3_10_1}
     \end{subfigure}
     \hfill
     \begin{subfigure}[b]{0.3\textwidth}
         \centering
         \includegraphics[width=\textwidth]{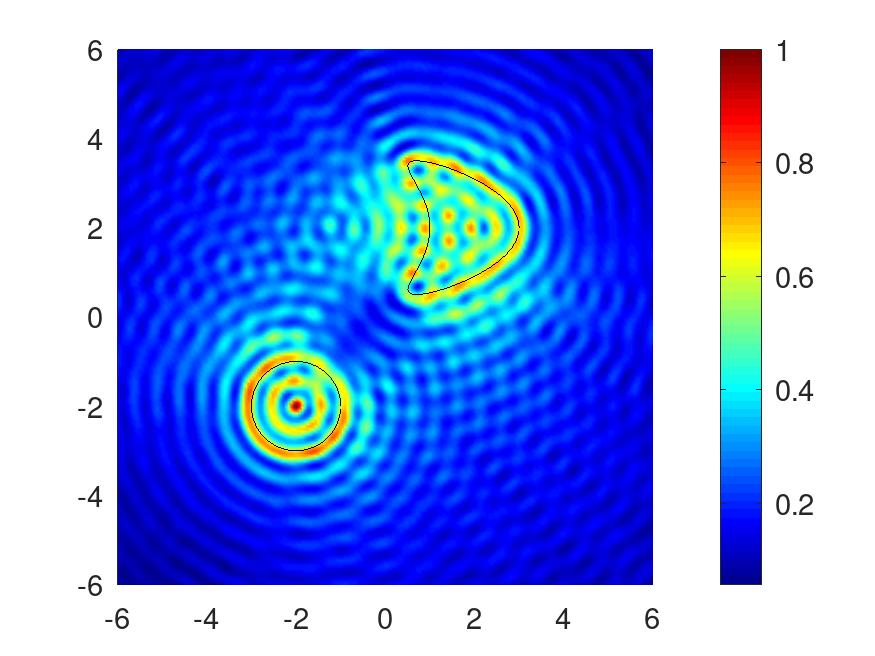}
         \caption{$5\%$ noise}
         \label{fig3_10_2}
     \end{subfigure}
     \hfill
     \begin{subfigure}[b]{0.3\textwidth}
         \centering
         \includegraphics[width=\textwidth]{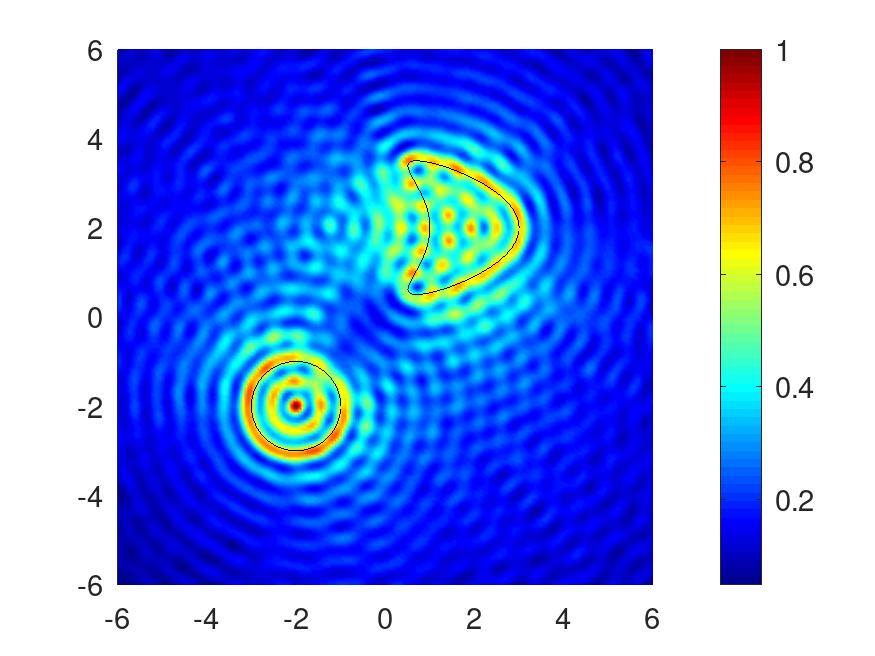}
         \caption{$10\%$ noise}
         \label{fig3_10_3}
       \end{subfigure}

           \begin{subfigure}[b]{0.3\textwidth}
        \centering
        \includegraphics[width=\textwidth]{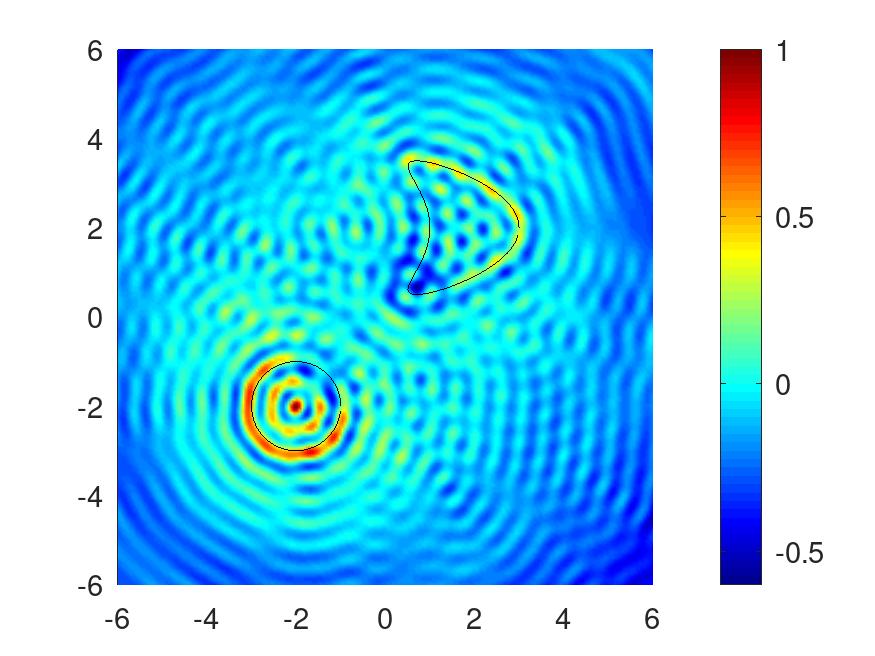}
        \caption{No noise}
        \label{fig3_11_1}
    \end{subfigure}
    \hfill
     \begin{subfigure}[b]{0.3\textwidth}
         \centering
         \includegraphics[width=\textwidth]{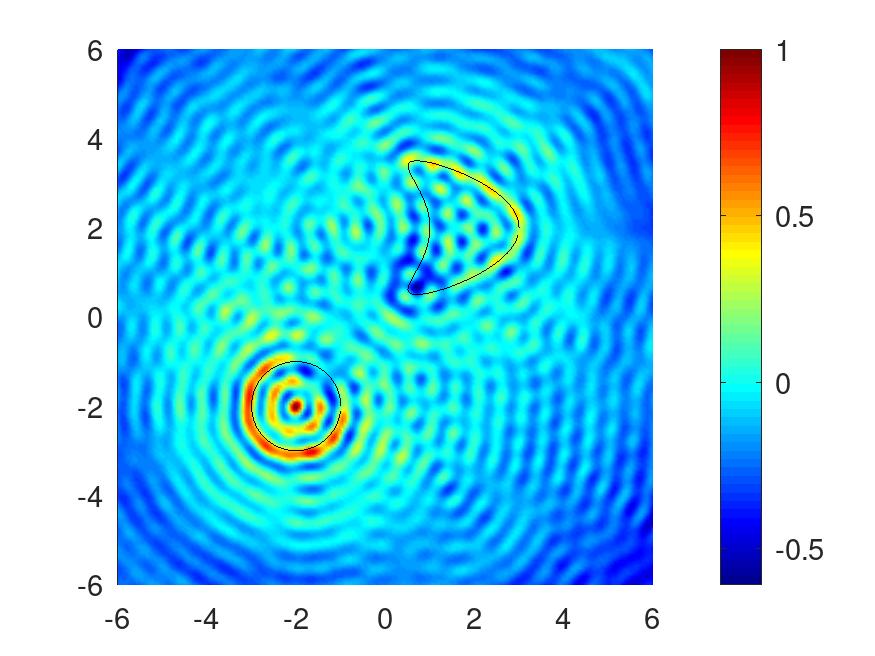}
         \caption{$5\%$ noise}
         \label{fig3_11_2}
     \end{subfigure}
     \hfill
    \begin{subfigure}[b]{0.3\textwidth}
        \centering
         \includegraphics[width=\textwidth]{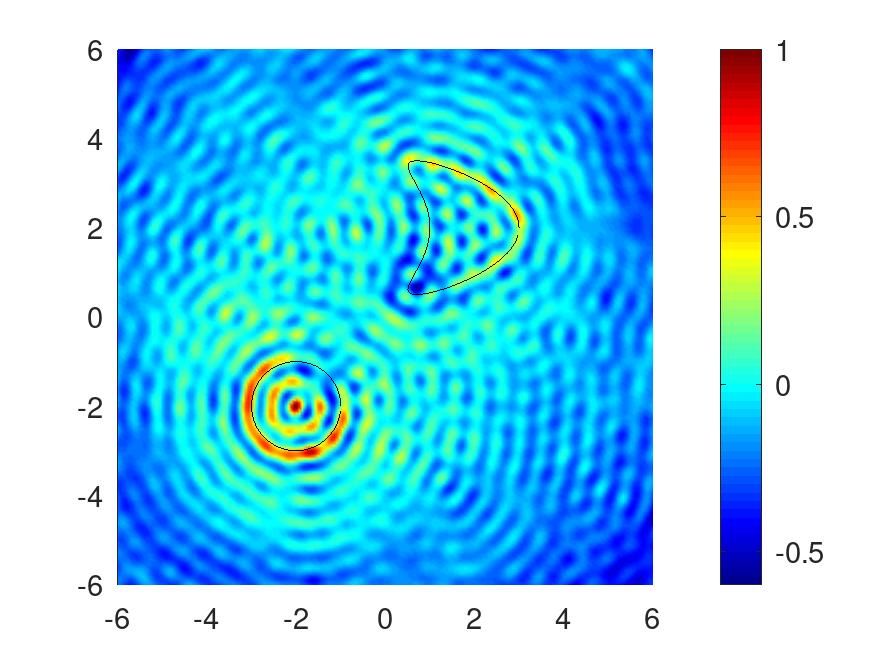}
         \caption{$10\%$ noise}
         \label{fig3_11_3}
       \end{subfigure}

               \caption{Example 3: The $j-$th row presents the reconstruction results of the indicator function $I_j$ $(j=1,\cdots,11)$.}
        \label{fig3:obstalces}
\end{figure}

\section*{Acknowledgments}

The research of TZ is supported by the Natural Science Foundation of Henan Province (No. 242300421677). The resaearch of ZG is supported by the National Natural Science Foundation of China (No. 12371393) and Natural Science Foundation of Henan Province (No. 242300421047).

\end{document}